\documentclass[10pt]{article}
\usepackage{amssymb,amsmath,amsfonts,amsthm,enumerate,color}
\usepackage[toc,page,title,titletoc,header]{appendix}
\usepackage{graphicx}
\usepackage{indentfirst}
\usepackage{multicol}
\usepackage{booktabs}
\usepackage{url}

\numberwithin{equation}{section}

\newtheorem{Theorem}{Theorem}[section]
\newtheorem{Lemma}[Theorem]{Lemma}

\newtheorem{Definition}[Theorem]{Definition}
\newtheorem{Corollary}[Theorem]{Corollary}
\newtheorem{Conjecture}[Theorem]{Conjecture}

\newtheorem{algorithm}{Algorithm}

\numberwithin{equation}{section}

 \def\p{\partial} 
\def \Vh0{\stackrel{\circ}{V}_h} \def\to{\rightarrow}

    \def\R{{\mathbb R}}

\newcommand{\lc}
{\mathrel{\raise2pt\hbox{${\mathop<\limits_{\raise1pt\hbox
{\mbox{$\sim$}}}}$}}}

\newcommand{\gc}
{\mathrel{\raise2pt\hbox{${\mathop>\limits_{\raise1pt\hbox{\mbox{$\sim$}}}}$}}}

\newcommand{\ec}
{\mathrel{\raise2pt\hbox{${\mathop=\limits_{\raise1pt\hbox{\mbox{$\sim$}}}}$}}}

\def\bb{\begin{equation}} \def\ee{\end{equation}}

\def\beqn{\begin{eqnarray}}  \def\eqn{\end{eqnarray}}

\def\beqnx{\begin{eqnarray*}} \def\eqnx{\end{eqnarray*}}

\def\bn{\begin{enumerate}} \def\en{\end{enumerate}}

\def\bd{\begin{description}} \def\ed{\end{description}}

\makeatletter
\newenvironment{tablehere}
  {\def\@captype{table}}
  {}
\newenvironment{figurehere}
  {\def\@captype{figure}}
  {}
\makeatother

\title{Algorithm for Overcoming the Curse of Dimensionality for State-dependent Hamilton-Jacobi equations}

\begin{document}

\author{
Yat Tin Chow\footnote{Department of Mathematics, UCLA, Los Angeles, CA 90095-1555. (ytchow@math.ucla.edu, sjo@math.ucla.edu, wotaoyin@math.ucla.edu)
Research supported by ONR grant N000141410683, N000141210838, N000141712162 and DOE grant DE-SC00183838.}
\and
J\'er\^ome Darbon\footnote{Division of Applied Mathematics, Brown University, Providence RI, 02912, USA. (jerome\_darbon@brown.edu).}
\and
Stanley Osher\footnotemark[1]
\and
Wotao Yin\footnotemark[1]
}

\date{}
\maketitle

\begin{abstract}

In this paper, we develop algorithms to overcome the curse of dimensionality in possibly non-convex state-dependent Hamilton-Jacobi partial differential equations (HJ PDEs) arising from optimal control and differential game problems.  The subproblems are independent and they can be implemented in an embarrassingly parallel fashion.  This is an ideal setup for perfect scaling in parallel computing.
The algorithm is proposed to overcome the curse of dimensionality \cite{Dim1, Dim2} when solving HJ PDE. 

The major contribution of the paper is to change either the solving of a PDE problem or an optimization problem over a space of curves to an optimization problem of a single vector, which goes beyond the work of \cite{boyd}.
We extend the method in \cite{Hopf_Lax_2, Hopf_Lax_3, Hopf_Lax}.  We \emph{conjecture} a (Lax-type) minimization principle to solve \emph{state-dependent} HJ PDE when the Hamiltonian is convex, as well as \emph{conjecture} a (Hopf-type) maximization principle to solve \emph{state-dependent} HJ PDE when the Hamiltonian is \emph{non-convex}.
In particular the conjectured (Hopf-type) maximization principle is a generalization of the well-known Hopf formula in \cite{Evans,Hopf_forumula,rublev}.
We showed the validity of the formula under restricted assumption for the sake of completeness, and would like to bring our readers to \cite{proofimportant} which validates that our conjectures hold in a more general setting after a previous version of our paper is on arXiv.    We conjectured the weakest assumption of our formula to hold is a psuedoconvexity assumption similar to one stated in \cite{rublev}.

The optimization problems are of the same dimension as the dimension of the HJ PDE.
The evaluation of the functional inside the minimization/maximization principles comes along with numerical ODE solvers and numerical quadrature rules.
We suggest a coordinate descent method for the minimization procedure in the generalized Lax/Hopf formula, and numerical differentiation is used to compute the derivatives.   This method is preferable since the evaluation of the function value itself requires some computational effort, especially when we handle higher dimensional optimization problem.
Similar to \cite{Hopf_Lax_3}, numerical errors come in because we use a numerical quadrature rule for computing integrals within the minimization/maximization principles, numerical differentiation to minimize the number of calculation procedures in each iteration, and in addition, with the choice of the numerical ODE solver.  These errors can be effectively controlled by choosing an appropriate mesh-size in time and the method does not use a mesh in space.  The use of multiple initial guesses is suggested to overcome possibly multiple local extrema since the optimization process is no longer convex.  A certificate is proposed to check the correctness of the argument minimum computed from the descent algorithm, and any local optimal or critical points obtained from the algorithm are discarded and the algorithm is reinitialized with a new random initial guesses.

 Our method is expected to have application in control theory, differential game problems and elsewhere.

\end{abstract}

\noindent { \footnotesize {\bf Mathematics Subject Classification
(MSC2000)}:
35F21,46N10,49N70,49N90,90C90,91A23,93C95
}


\section{Introduction} \label{sec1}

Hamilton-Jacobi-Isaacs partial differential equations (HJ PDE) are crucial in analyzing continuous/differential dynamic games, control theory problems, and dynamical systems coming from the physical world, e.g. \cite{diff_game}.  An important application is to compute the evolution of geometric objects \cite{Propagation}, which was first used for reachability problems in \cite{con1,con2}, to our knowledge.

Numerical solutions to HJ PDE have attracted a lot of attention.  Most of the methods involve the introduction of a grid and a finite difference discretization of the Hamiltonian.  Some of these well-known methods using discretization include ENO/WENO-type methods \cite{WENO} and Dijkstra-type \cite{Dijkstra} methods such as fast marching  \cite{FM} and fast sweeping \cite{FS}.
However, with their discretization nature, these numerical approaches of HJ PDE  suffer from poor scaling with respect to dimension $d$, hence rendering them impossible to be applied to problems in high dimensions.

Research has therefore been conducted by several groups in search of possible algorithms that can scale reasonably with dimension.  Some new algorithms are introduced in e.g. \cite{darbon_HJ,Tensor, Sparse}.
In \cite{Hopf_Lax, Hopf_Lax_2, Hopf_Lax_3}, the authors proposed a causality-free method for solving possibly non-convex and time dependent HJ PDE based on the generalized Hopf-Lax formula.    Using the Hopf-Lax formula, the PDE becomes decoupled and the solution at each point can be effectively calculated by $d$-dimensional minimization, with $d$ the space dimension.

In this work, we propose to extend the method in \cite{Hopf_Lax, Hopf_Lax_2, Hopf_Lax_3}.  
The major contribution of the paper is to change either the solving of a PDE problem or an optimization problem over a space of curves to an optimization problem of a single vector.
We \emph{conjecture} the (Lax-type) minimization principle to solve \emph{state-dependent} HJ PDE when the Hamiltonian is convex.  We also \emph{conjecture} a (Hopf-type) maximization principle to solve \emph{state-dependent} HJ PDE when the Hamiltonian is \emph{non-convex} but when the initial data is \emph{convex}. 
In particular the conjectured (Hopf-type) maximization principle is a generalization of the well-known Hopf formula in \cite{Evans,Hopf_forumula,rublev}.
 We validated our conjectures under restricted assumptions and refers to \cite{proofimportant} for the validation of our formula under a less restricted setting.
The optimization problems are of the same dimension as the dimensions of the HJ PDE.
A coordinate descent method is suggested for the minimization procedure in the generalized Lax/Hopf formula, and numerical differentiation is used to compute the derivatives.   This method is preferable since the evaluation of the function value itself requires some computational effort, especially when we handle higher dimensional optimization problems.
The use of multiple initial guesses is suggested to overcome possibly multiple local extrema since the optimization process is no longer convex.
A simple numerical ODE solver is used to compute the bi-characteristics in the Hamiltonian system.
A numerical quadrature rule is used for computing integrals with respect to time inside the minimization/maximization principles.
Coordinate descent is also used and a numerical differentiation is performed to minimize the number of calculation procedures in each iteration steps. 
In this paper we illustrate the practicality of our method using the simplest ODE solvers, quadrature rules and finite difference methods, namely the forward Euler, rectangular rule, and forward difference methods.
Nonetheless, all these components of the optimization can be improved by using better numerical methods, e.g. pseudo-spectral methods for ODE, Gauss-Lobatto quadrature rules, etc.
These choices, together with choosing an appropriate mesh-size in time, minimizes errors effectively.  Our method does not use a mesh in space, and solutions can be computed at each point $(x,t)$ in a totally decoupled and embarrassingly parallel manner.
Moreover, the solution to the HJ PDE is evaluated at each point $(x,t)$ by coordinate minimization described by the minimization/maximization principles at that point.  

As for high dimensional control, we would like to compare our work with \cite{boyd}, which is concerned with discrete approximations of a particular set of control, namely the discrete linear-convex control problems involving a quadratic control. In the paper, both the final parameter $v$ as well as the whole curve representing controls are unknowns.  The major contribution of our paper are to remove the optimization over the control (getting the HJ PDE as in \cite{Hopf_Lax_2, Hopf_Lax_3, Hopf_Lax}), as well as extending our method in \cite{Hopf_Lax_2, Hopf_Lax_3,Hopf_Lax} to a much more general setting, i.e. differential games/nonconvex problems and the Hamilton-Jacobi limit. Also,  our algorithm is faster comparing with \cite{boyd} even at the discrete level because we remove the optimization over the control.   We have also proposed both the Lax and the Hopf version, while \cite{boyd} focuses only on the analogue of a Lax type formula.  We would like to remark that by considering and working on the dual, the algorithm can be made much faster in many cases, especially when the initial data is convex.

Our formal statements of the conjectures that are used for computation will be given in section \ref{sec4}.  However, before we give an exact formulation of our conjectures, 
let us briefly provide our main conjectures.  In what follows, we denote $p$ as the co-state variable and $H(x,p,t)$ as the Hamiltonian, as well as $\varphi(x,t)$ as the value function satisfying \eqref{HJeqn}-\eqref{HJinitial}. We also denote $( \gamma(t) , p(t) ) $ as the bi-characteristic curve in the phase space that shall satisfy the constraints in the following formulae.  Then we conjecture the followings:
\begin{enumerate}
\item
\textbf{Minimization principle} (Lax Formula) 
Assume $H(x,p,t) \in C^2$ and is convex w.r.t. $p$, and (A5) is satisfied, then there exists $t_0$ such that the viscosity solution to \eqref{HJeqn}-\eqref{HJinitial} can be represented as for $t \leq t_0$ that:
\beqn
 \varphi(x,t)
&=&\min_{v \in \mathbb{R}^d } \bigg\{ g( \gamma(x,v,0)) + \int_{0}^t \left\{ 
\langle p(x,v,s), \partial_p H( \gamma(x,v,s), p(x,v,s), s) \rangle
 - H(\gamma(x,v,s), p(x,v,s),  s) \right\} ds  :  \notag \\
& & \qquad \qquad
 \begin{matrix} \dot{ \gamma }(x,v,s) =  \partial_p H( \gamma(x,v,s), p(x,v,s),s), \\  \dot{ p }(x,v,s) = -\partial_x H( \gamma(x,v,s), p(x,v,s), s), \\ \gamma(x,v,t) = x, \, p(x,v,t) = v\end{matrix} 
 \bigg\}
\label{lax_formula}
\eqn
and its discrete approximation given a small $\delta$,
\beqnx
 \varphi(x,t)
&\approx&\min_{v \in \mathbb{R}^d } \bigg\{ g( x_0(x,v) ) + \delta \sum_{n=1}^{N-1} \left\{ 
\langle p_n(x,v) , \partial_p H(x_n(x,v), p_n(x,v), t_n) \rangle
 - H(x_n(x,v), p_n(x,v), t_n) \right\}   :  \\
& & \qquad \qquad
 \begin{matrix} x_{n+1}(x,v) - x_n(x,v) =  \delta \partial_p H(x_n(x,v), p_n(x,v), t_n), \\ p_{n-1}(x,v) - p_n(x,v) =  \delta \partial_x H(x_n(x,v), p_n(x,v), t_n), \\ x_N= x, \, p_N = v\end{matrix} 
 \bigg\}
\eqnx
We would like to remark that the bi-characteristics $\left(  \gamma (x,v,\cdot) , p(x,v,\cdot) \right) $ depend on the two initial conditions $x$ and $v$, and this dependence is emphasized using a notations that include the two independent variables.
\item
\textbf{Maximization principle} (Hopf Formula) 
Assume $H(x,p,t) \in C^2$ and $g(p) \in C^2$ is convex w.r.t. $p$ that satisfies (A5).  Assume the pseudoconvex condition stated in Conjecture \ref{conj2} holds, then there exists $t_0$ such that the viscosity solution to \eqref{HJeqn}-\eqref{HJinitial} can be represented as,
for $t < t_0$, that:
\beqn
& & \varphi(x,t) \notag \\
&=&\max_{v \in \mathbb{R}^d } \bigg\{  \langle x,v \rangle  - g^*( p(x,v,0)) - \int_{0}^t  \bigg\{ H(\gamma(x,v,s), p(x,v,s), s) -  \langle \partial_x H( \gamma(x,v,s), p(x,v,s),s),  \gamma(x,v,s) \rangle  \bigg\} ds  :  \notag \\
& & \qquad \qquad
 \begin{matrix} \dot{ \gamma }(x,v,s) =  \partial_p H( \gamma(x,v,s), p(x,v,s),s), \\  \dot{ p }(x,v,s) = - \partial_x H( \gamma(x,v,s), p(x,v,s), s), \\ \gamma(x,v,t) = x, \,  p(x,v,t) = v\end{matrix} 
 \bigg\}
\label{hopf_formula}
\eqn
and its discrete approximation given a small $\delta$
\beqnx
& & \varphi(x,t) \\
&\approx&\max_{v \in \mathbb{R}^d } \bigg\{  \langle x_N,v_N \rangle   -g^*( p_1(x,v) ) 
- \delta \sum_{n=1}^{N}  H(x_n(x,v), p_n(x,v), t_n)
+ \delta \sum_{n=1}^{N-1}  \langle x_n(x,v) , \partial_x H(x_n(x,v), p_n(x,v), t_n) \rangle
  :  \\
& & \qquad \qquad
 \begin{matrix} x_{n}(x,v) - x_{n-1}(x,v) =  \delta \partial_p H(x_n(x,v), p_n(x,v), t_n), \\ p_{n}(x,v) - p_{n+1}(x,v) =  \delta \partial_x H(x_n(x,v), p_n(x,v), t_n), \\ x_N= x, \, p_N = v\end{matrix} 
 \bigg\}
\eqnx
\end{enumerate}
and in both cases the argument $v$ attaining the minimum or the maximum in the respective formula is $\p_x \varphi(x,t) $ when $\varphi$ is smooth at $(x,t)$.  These are the key conjectures from a practical point of view.
When $H(x,p,t)$ is non-smooth (e.g. in the case when $H(x,p,t)$ is homogeneous of degree $1$ w.r.t. $p$), more general forms of the respective conjectures are also available.  We must emphasize that these modifications of the formulas are \textbf{necessary} when $H$ is non-smooth or otherwise some part of sub-gradient flow will be missed and the formula will be incorrect. We notice that, in particular, our conjectured (Hopf-type) maximization principle is a generalization of the well-known Hopf formula in \cite{Evans,Hopf_forumula,rublev}.

Our way to approach the problem using the Lax minimization principle or the Hopf maximization principle goes between the indirect method (Pontryagin's maximum principle) and the direct method (direct optimization over the spaces of curves) and is optimal for computational efficiency.   This way we are able to minimize the number of variables to be optimized, since a discretization of curve needs a lot of variables; and on the other hand we keep a variable and the functional such that we have a descent algorithm that guarantee convergence to a local minimum.  The correctness of the computed limit (i.e. if it is a global minimum) can be checked by the condition $p(0) \in \partial g (\gamma (0))$.

The rest of our paper is organized as follows: in subsection \ref{sec2} we introduce the general class of HJ-PDE that we are interested in, and then we briefly explain the connection between HJ-PDE and differential games subsection \ref{sec3}.  Then in section \ref{sec4}, first in subsection \ref{proof_simple}, we showed the formula are true in a restricted set of assumptions; and then we discuss the major formulae that we use in our work: the conjectured minimization principle (Lax formulation) in subsection \ref{con/con_ham} and the conjectured maximization principle (Hopf formulation) in subsection \ref{noncon/con_ham}. 
We then go on to explain briefly our numerical techniques in section \ref{sec5} and the certificate of correctness.  Numerical examples are given in section \ref{sec6}.

We would like to bring the readers to notice that after a previous version of our paper in arXiv is out, in \cite{proofimportant} showed that our conjectures hold in a more general setting than we did.   However for the sake of completeness, we keep the proof with the restricted assumptions, and we refer our readers to the more general case shown in \cite{proofimportant}

\section{Review of Hamilton-Jacobi Equations and Differential Games}  

\subsection{Hamilton-Jacobi Equations} \label{sec2}

In this work, we are concerned with an approximation scheme for solving the following HJ PDE:
\beqn
\frac{\p}{ \p t} \varphi(x,t) + H( x, \nabla_x \varphi(x,t), t) = 0 \quad \text{ in } \mathbb{R}^d \times (0, \infty) \,,
\label{HJeqn}
\eqn
where $H: \mathbb{R}^d \times  \mathbb{R}^d  \times \mathbb{R} \rightarrow  \mathbb{R}$ is a continuous Hamiltonian function bounded from below by an affine function,   $\frac{\p}{ \p t} \varphi$ and $  \nabla_x \varphi $ respectively denote the partial derivatives with respect to $t$ and the gradient vector with respect to  $x$ of the function $\varphi:\mathbb{R}^d \times (0, \infty)\to \R$.  We are also given the initial data
\beqn
\varphi(x,0) = g(x)  \quad \text{ in } \mathbb{R}^d \,.
\label{HJinitial}
\eqn
We aim to compute the viscosity solution to \eqref{HJeqn}-\eqref{HJinitial} \cite{vis2, vis1} at a given point $x \in \mathbb{R}^d$ and time $t \in ( 0, \infty) $.

The viscosity solution to \eqref{HJeqn}-\eqref{HJinitial}  is explicitly given by the Hopf-Lax formulae when $H$ is $x$ independent, which holds because the integral curves of the Hamiltonian vector field (i.e., the bi-characteristics in the phase space) are straight lines when  projected to $x$-space.


In general, when $H$ is $x$-dependent, we have shown in section \ref{proof_simple} that under restricted assumptions as listed in that subsection and a finite difference approximation and numerical integration, that a minimization and maximization principle approximates the solution to the HJ PDE.  They are formulated using a KKT condition which gives bi-characteristics in the phase (i.e. following the flow along the Hamiltonian vector field given by $H$.)

With those principles holding true in special cases, we postulate that the minimization and maximization principles may hold true for more general cases, which we will state more precisely in section \ref{sec4}.  
Following a previous version of our paper in arXiv in \cite{preprint}, in \cite{proofimportant} the authors verified our main conjectures under a more relaxed set assumptions using the notions of minimax, in both the Hamilton-Jacobi-Ballman situation and in some special cases of Hamilton-Jacobi-Issac situation (i.e. when the Hamiltonian is non-convex).  This is a first and monumental work which suggests the validity of our conjecture in a more general situation.

Minimax viscosity solution is introduced in \cite{minimax1,minimax2,minimax3}, and connections are made to control problems and differential games can be found in \cite{minimax4,minimax5,minimax6,minimax7}.
The theory go in parallel with the definition and theory of viscosity solution \cite{diff_game} coming from differential game, and in the case with convex Hamiltonian, they always coincide.
For more exposition of the minimax solution and development of this relationship with control problem, we refer the readers to \cite{proofimportant}

A more general notion of minimax viscosity solution (with a general non-convex but smooth Hamiltonion) which helps to recast the solution using a formulation that involves patching the graph of multi-valued geometric solutions in a correct manner, e.g. in \cite{minimaxsingularity,convexminimax}.  In the convex case, under further assumption, the solution can be formulated as a mini-max saddle point problem of a functional over the spaces of curves \cite{convexminimax}.
However, it is known that the general formulation of minimax solution coincide with the viscosity solution only if the dynamically programming principle (i.e. the semigroup property) is satisfied for the minimax solution \cite{semigroup,limitingminimax}.  In fact, in the work \cite{limitingminimax}, the author shows further that the viscosity solution can be constructed by a limiting sequence of minimax solutions.  However it seems such a limiting process may defy the purpose of formulating the solution of the viscosity solution as a calculable process which is in the form of an optimization.
Besides, in such full generality e.g. in \cite{minimaxsingularity,convexminimax} the generating will need to be assumed to be quadratic at infinity and they are described using a more complicated with deformation retract and relative cohomology (in the spirit of Morse theory).  
Therefore such a full generality may bring difficulty for computational purpose.  
We would also like to mention that, In the case of a convex Hamiltonian, the minimax solution is also known as the principle of least action, e.g. \cite{convexminimax,rockafellar1}.  A slightly different but very similar direction, which uses Wentzell-Freidlin theory and construct minimum action methods, are also developped, as in e.g. \cite{VE1,VE2,VE3,VE4}

The purpose of our conjecture is to find an appropriate generalization the Hopf formula which still involves an optimization problem of $d$ dimensions, such that one may apply a good numerical method for solving the solution using an optimization algorithm.  We are not aiming at full generality when one might need to use more sophisticated mathematical language to describe the solution while the formulation is not easy to be numerically computed.  However we would also prefer to obtain generality that is sensible and not restricted to a restricted class of Hamiltonian $H(x,p,t)$ (convex w.r.t. $p$, concave w.r.t. $x$ and satisfying a finite concavity-convexity assumption, c.f. Assumptions (A) and Theorem 2.3 in \cite{rockafellar1}) (see also Theorem 4.8 in \cite{rockafellar2}).
Since it is unclear how far such conditions and notions can be generalized as well as remained computable, we only show in a restricted case that our formulas hold in subsection \ref{proof_simple}.  One very important part of a successful proof of our formula will be to show that the equality holds for the mini-max inequality, e.g. by \cite{sion1,sion2}, or by stability/Slater's criterion \cite{rockafellar}.
We wish to leave it as a conjecture with more general conditions.

On the other hand, there is a well-known notion of Pontryagin's maximum principle \cite{pontryagin} which will give us an optimality condition for the control.  
In particular, work e.g. in \cite{clarke} links the maximum principle and dynamic programming, which discussed in the non-smooth case when will the costate variable coincide with a partial generalized gradient of the value function (c.f.  \cite{clarke}) for a clear definition.
However the optimality condition has a final condition for the state variable and an initial condition for the co-state variable,  We hope to avoid such a formulation and forward-backward iteration to compute for an optimal (hoping that the fixed point iteration should converge).  The motivation of our conjecture over this KKT condition is to keep a Lyapunov functional such that a numerical algorithm can minimize, and therefore guarantee descent of a functional and convergence of a subsequence of an algorithm even with more pathogical behaviors of $H$ and $g$, e.g. non-convexity, etc.

Our way to approach the problem goes between the indirect method (Pontryagin's maximum principle) and the direct method (direct optimization over the spaces of curves) and is optimal for computational efficiency.   This way we are able to minimize the number of variables to be optimized, since a discretization of curve needs a lot of variables; and on the other hand we keep a variable and the functional such that we have a descent algorithm that guarantee convergence to a local minimum.  The correctness of the computed limit (i.e. if it is a global minimum) can be checked by the condition $p(0) \in \partial g (\gamma (0))$. 

In what follows we will prove the most simple version of the conjecture in the case with Hamilton-Jacobi-Ballman equation under certain convexity and other technical assumptions.  A more general version of the theorem with rigorous proofs can be found in \cite{proofimportant}.

\subsection{Differential Games and its Connection with Non-convex Hamilton-Jacobi Equations}  \label{sec3}

In this subsection, we give a brief introduction to the specific optimal control and differential game problems we are considering, and a brief explanation as to how we recast them as problems of solving HJ PDE's.
We follow discussions in \cite{Hopf_Lax_2,Hopf_Lax_3,Hopf_Lax,diff_game}, see also \cite{Evans}, about optimal control and also for differential games, and their links with HJ PDE.

To start with, we first consider two convex compact sets $C$ and $D$, in which control parameters lie.
Then let us denote $\mathcal{A} = \left \{ a: [t, T] \rightarrow C\, : \, a \text{ is measurable}  \right \} $, which is referred to as the admissible set of Player I; and  $\mathcal{B} = \left \{ b: [t, T] \rightarrow D\, : \,  b \text{ is measurable} \right \} $, which is referred to as the admissible set of Player II.
We call the measurable functions $a: [t, T] \rightarrow C$ in the set $\mathcal{A}$  and the function $b: [t, T] \rightarrow D$ in the set $\mathcal{B}$  as \textit{controls} performed by Players I and II respectively.

We start with a system of differential equations given as follows.  Fix $ 0 \leq t < T , \, x \in \mathbb{R}^d$. We consider
\beqn
\begin{cases}
\frac{dx}{ds} (s) = f(s, x(s), a(s), b(s) )  & t \leq s \leq T  \,, \\
 x(t) = x  \,,&
\end{cases}
\label{4.1}
\eqn
We assume that the function
\beqnx
f : [0,T] \times \mathbb{R}^d \times A \times B \rightarrow \mathbb{R}^m
\eqnx
is uniformly continuous and
\beqnx
\begin{cases}
| f(t,x,a,b)| \leq C_1 \\
| f(t,x,a,b) - f(t,y,a,b) | \leq C_1 |x - y| \,,
\end{cases}
\eqnx
for some constant $C_1$ and for all $0 \leq t \leq T$, $x,y \in \mathbb{R}^m$, $a \in A$, $b \in B$.
The unique solution to \eqref{4.1} is called the response of the controls $a(\cdot), b(\cdot)$.
Then we introduce the \textit{payoff} functional for a given pair of $(x,t)$:
\beqnx
P(a,b) := P_{t,x} (a (\cdot) , b(\cdot))  := \int_t^T h (s, x(s), a(s), b(s)) \, ds + g(x(T)) \,,
\eqnx
where $g: \mathbb{R}^d \rightarrow \mathbb{R}$ satisfies
\beqnx
\begin{cases}
| g(x)| \leq C_2 \\
| g(x) - g(y) | \leq C_2 |x - y| \,,
\end{cases}
\eqnx
and $h$ satisfies
\beqnx
\begin{cases}
| h(t,x,a,b)| \leq C_3 \\
| h(t,x,a,b) - h(t,y,a,b) | \leq C_3 |x - y| \,,
\end{cases}
\eqnx
for some constants $C_2,C_3$ and all $0 \leq t \leq T$, $x,y \in \mathbb{R}^m$, $a \in A$, $b \in B$. In a differential game, the goal of player I is to maximize the functional $P$ by choosing his control $a$ whereas that of player II is to minimize $P$ by choosing his control $b$.

Now we are ready to define the lower and upper values of the differential game, based on the notation introduced above.
We first define the two sets containing the respective controls of players I and II:\beqnx
M(t) := \{a : [t,T] \rightarrow A : a \text{ is measurable.}\} \,, \\
N(t) := \{b : [t,T] \rightarrow B : a \text{ is measurable.}\} \,.
\eqnx
Define a strategy for player I as the map
\beqnx
\alpha : N(t) \rightarrow M(t)
\eqnx
for each $t \leq s \leq T$ and $b, \hat{b} \in B$ such that
\beqnx
b(\tau) = \hat{b} (\tau) \text{ for a.e. } t \leq \tau \leq s  \quad
\Rightarrow \quad  \alpha[b](\tau) = \alpha[\hat{b}] (\tau)  \text{ for a.e. } t \leq \tau \leq s \,.
\eqnx
Therefore a strategy for player I $ \alpha[b]$ is the control of player I given that of player II as $b$.
Similarly, let us define a strategy for player II as
\beqnx
\beta : M(t) \rightarrow N(t)
\eqnx
for each $t \leq s \leq T$ and $a, \hat{a} \in A$ such that
\beqnx
a(\tau) = \hat{a} (\tau) \text{ for a.e. } t \leq \tau \leq s  \quad
\Rightarrow \quad  \beta[a](\tau) = \beta[\hat{a}] (\tau)  \text{ for a.e. } t \leq \tau \leq s \,.
\eqnx
Again a strategy for player II $ \beta[a]$ is the control of player II given that of player I as $a$.

Now let  $\Gamma(t)$ denote the set of all strategies for I  and  $\Delta(t)$  for II beginning at time $t$. We are well equipped to define the upper and lower values of the differential game.
The lower value $V(x,t)$ is defined as
\beqnx
V(x,t) &:=& \inf_{\beta \in \Delta(t)} \sup_{a \in M(t)} P_{t,x}(a, \beta[a])  \\
&:=& \inf_{\beta \in \Delta(t)} \sup_{a \in M(t)} \left \{ \int_t^T h (s, x(s), a(s), \beta[a](s)) \, ds + g(x(T))  \right\} \,,
\eqnx
where $x(\cdot)$ solves \eqref{4.1} for a given pair of $(x,t)$.
Likewise, the upper value $U(x,t)$ is defined as
\beqnx
U(x,t) &:=& \sup_{\alpha \in \Gamma(t)} \inf_{b \in N(t)} P_{t,x}(\alpha[b],b)  \\
&:=&\sup_{\alpha \in \Gamma(t)} \inf_{b \in N(t)} \left \{ \int_t^T h (s, x(s), \alpha[b](s),b(s)) \, ds + g(x(T))  \right\} \,,
\eqnx
where $x(\cdot)$ again solves \eqref{4.1} for a given pair of $(x,t)$.

In fact, derived from dynamic programming optimality conditions in \cite{diff_game}, the lower and upper values $V$ and $U$ are the viscosity solutions of a certain possibly nonconvex HJ PDE.
For the sake of exposition, we first define the following two Hamiltonians:
\beqnx
\tilde{H}^+(x,p,t) = \max_{b \in B} \min_{a \in A} \{- \langle f (t,x,a,b) , p \rangle - h(t,x,a,b)\} \,, \\
\tilde{H}^-(x,p,t) = \min_{a \in A} \max_{b \in B} \{- \langle f (t,x,a,b) , p \rangle - h(t,x,a,b)\} \,.
\eqnx
A very important case of this class of Hamiltonian is when $\tilde{H}^\pm(t,x,p)$ are homogeneous of degree $1$, which is highlighted in this work.  In fact, in the case where 
\beqnx
f(t,x,a,b) &=&  a - b \\
h(t,x,a,b) &=& -  \mathcal{I}_{ c_1(x,t)  A(x,t)} (a) +   \mathcal{I}_{ c_2(x,t) B(x,t)} (b) \,,
\eqnx
where $\mathcal{I}_\Omega$ is the indicator functions of the sets $\Omega$, i.e.
\beqnx
\mathcal{I}_\Omega(x) = 
\begin{cases}
0  & \text{ if } x \in\Omega\\
\infty & \text{ if } x \notin \Omega
\end{cases}
\eqnx
and $A(x,t)$ and $B(x,t)$ are balanced, then
\beqnx
\tilde{H}^+(x,p,t) &=& \max_{b \in B} \min_{a \in A} \{ \langle b - a  , p \rangle  + \mathcal{I}_{ c_1(x,t) A(x,t)} (a) -   \mathcal{I}_{ c_2(x,t) B(x,t)} (b) \}  \\
&=& \max_{b \in c_2(x,t)B(x,t)} \min_{a \in c_1(x,t)  A(x,t)} \{ \langle b - a  , p \rangle \}  \\
&=& \max_{ [c_2(x,t)]^{-1} b \in B(x,t)} \min_{ [c_1(x,t)]^{-1} a \in A(x,t)} \{ \langle b - a  , p \rangle \}  \\
&=& \max_{  b \in B(x,t)} \min_{ a \in A(x,t)} \{ \langle c_2(x,t) b -  c_1(x,t) a  , p \rangle \} \,, \\
\tilde{H}^-(x,p,t)&=&\min_{a \in A}  \max_{b \in B}  \{ \langle b - a  , p \rangle  + \mathcal{I}_{ c_1(x,t) A(x,t)} (a) -   \mathcal{I}_{ c_2(x,t) B(x,t)} (b) \}  \\
&=&  \min_{a \in c_1(x,t)  A(x,t)} \max_{b \in c_2(x,t)B(x,t)}  \{ \langle b - a  , p \rangle \}  \\
&=& \min_{ [c_1(x,t)]^{-1} a \in A(x,t)} \max_{ [c_2(x,t)]^{-1} b \in B(x,t)}  \{ \langle b - a  , p \rangle \}  \\
&=& \min_{ a \in A(x,t)} \max_{  b \in B(x,t)}  \{ \langle c_2(x,t) b -  c_1(x,t) a  , p \rangle \} \,.
\eqnx
thus it holds that $\tilde{H}^+$ and $\tilde{H}^-$ coincide, as well as the following relationship:
\beqnx
\tilde{H}^\pm(x,p,t) &=& \max_{b \in B(x,t)} \min_{a \in A(x,t)} \{  c_1(x,t) \langle a, p \rangle  -  c_2(x,t)  \langle b, p \rangle \}  \\
&=&   c_1(x,t)  \min_{a \in A(x,t)} \{  \langle a, p \rangle \} -  c_2(x,t)  \max_{b \in B(x,t)} \{ \langle b, p \rangle \}  \\
&=& - c_1(x,t) \mathcal{I}_{A(x,t)}^* (p) +  c_2(x,t) \mathcal{I}_{B(x,t)}^* (p)  \,.
\eqnx
In this case, $H^\pm(x,p,t) $ can be written as a difference of two positively homogeneous (of degree $1$) Hamiltonians $ \Phi_1(x,\cdot,t),  \Phi_2 (x,\cdot,t)$, namely,
\beqnx
\tilde{H}^\pm(x,p,t) = - c_1(x,t)  \Phi_1(x,p,t) +  c_2(x,t) \Phi_2(x,p,t)
\eqnx
where $\Phi_1(x,\cdot,t)$ and $\Phi_2(x,\cdot,t)$ have their respective Wulff sets as $A(x,t)$ and $B(x,t)$ (see
\cite{gauge,Hiriart,Wulff} for more details of the Wulff set.)

Now, for a general pair of $\tilde{H}^\pm(x,p,t)$, we have the following well-known theorem:
\begin{Theorem} \cite{diff_game} \label{diff_diff}
The function $U$ is the viscosity solution to the HJ PDE :
\beqnx
\begin{cases}
\frac{\partial }{\partial t} U - \tilde{H}^+( x, \nabla_x U ,t ) = 0 &\text{ on } \mathbb{R}^d \times (-\infty,T) \,,\\
U(x,T) = g(x) &\text{ on } \mathbb{R}^d  \,.
\end{cases}
\eqnx
Similarly, the function $V$ is the viscosity solution to the HJ PDE :
\beqnx
\begin{cases}
\frac{\partial }{\partial t} V - \tilde{H}^-(x, \nabla_x V ,t ) = 0 &\text{ on } \mathbb{R}^d \times (-\infty,T) \,,\\
V(x,T) = g(x) &\text{ on } \mathbb{R}^d  \,.
\end{cases}
\eqnx
\end{Theorem}
\noindent It is worth mentioning again that, in a general setting where $h$ is possibly nonconvex, the two Hamiltonians $\tilde{H}^+(x,p,t)$ and $\tilde{H}^-(x,p,t)$ may not coincide. But, when they do, there is the following corollary:
\begin{Corollary} \cite{diff_game}
If \beqnx
\tilde{H}^+( x, p,t) = \tilde{H}^-( x, p,t) \text{ on } [t,T]  \times \mathbb{R}^d \times  \mathbb{R}^d \,,
\eqnx
then it holds that $U = V$.
\end{Corollary}
\noindent Hereafter, when $U=V$, we write  $ \varphi(x,t) := U(x,T-t) = V(x,T-t) $, and write $H(x,p,t) = \tilde{H}^\pm(x,p,T-t) $, then 
\beqnx
\begin{cases}
\frac{\partial }{\partial t} \varphi + H(x, \nabla_x \varphi ,t ) = 0 &\text{ on } \mathbb{R}^d \times (0,\infty) \,,\\
\varphi(x,0) = g(x) &\text{ on } \mathbb{R}^d  \,.
\end{cases}
\eqnx
Note that in general, the Hamiltonians $H$ can be nonconvex and/or nonconcave, and this is one very
important occasion in which nonconvex HJ PDE arises.
We would like to mention that the convention to write the HJ-PDE as an initial value problem or termianl value problem is a matter of convention.  One may write that either with the variable $t$ or with the variable $T-t$ and switch between the two formulations.  Since for many applications it is stated as an initial value problem, we would like to stick to the convention using initial value problem.

In the next section, we will discuss possible representation formulae of the HJ-PDE equation, which may help us to compute the solution quickly and in parallel.  We will prove they hold for restricted assumptions, and refer to the readers to \cite{proofimportant} for the case with less restricted assumptions.

\section{Representation formulae for viscosity solution of HJ PDE}  \label{sec4}

In this section, 
we prove in subsection \ref{proof_simple} that the two formulas hold under restricted assumptions.  
Then we go on to making a conjecture of a (Lax-type) minimization principle for the viscosity solution to \eqref{HJeqn}-\eqref{HJinitial} when $H$ is convex, and a (Hopf-type) maximization principle when $H$ is non-convex but when $g$ is convex that they shall still hold in this case other less restricted assumptions.   
In several examples given in the paper, when a Hopf formula is known for the solution, our conjectured representation reduces to these known formulae. 
We refer our readers to \cite{proofimportant} for a proof that our conjectures shall hold for a less restricted set of assumptions.

Before we provide our formal statements of the conjectures, for the sake of exposition, let us emphasize that the formulae stated in Section \ref{sec1} are the key conjectures from a practical point of view.
When $H(x,p,t)$ is non-smooth, more general forms of the respective conjectures are necessary.  The precise statements and proof under restricted assumptions will be given in the following subsections. 
Then we make our conjecture that the two formula still hold under less restricted assumptions.
One point to remark is that our conjectured Hopf formula is a generalization of the well-known Hopf formula in \cite{Evans,Hopf_forumula,rublev}.

\subsection{A simplified system} \label{simplified}

In order to show the formula is true in a simplified setting, for the sake of exposition, let us consider first a \textbf{simplified} system of differential equations:  Fix $ 0 \leq t < T , \, x \in \mathbb{R}^d$. We consider
\beqnx
\begin{cases}
\frac{dx}{ds} (s) = f(s, x(s), u(s))  & t \leq s \leq T  \,, \\
 x(t) = x  \,,&
\end{cases}
\eqnx
where
$u(\cdot) \in \mathcal{U}(t) = \left \{ u: [t, T] \rightarrow U\, : \,  u \text{ is measurable} \right \}$
is again a control and $U$ is the admissible set.
We again assume that the function
\beqnx
f : [0,T] \times \mathbb{R}^d \times U \rightarrow \mathbb{R}^m
\eqnx
is uniformly continuous and
\beqnx
\begin{cases}
| f(t,x,u)| \leq C_1 \\
| f(t,x,u) - f(t,y,u) | \leq C_1 |x - y| \,,
\end{cases}
\eqnx
for some constant $C_1$ and for all $0 \leq t \leq T$, $x,y \in \mathbb{R}^m$, $u \in U$.

We also consider the following simplified payoff function
\beqnx
P(u) := P_{t,x} (u (\cdot))  := \int_t^T h (s, x(s), u(s)) \, ds + g(x(T)) \,,
\eqnx
We consider also the value function
\beqnx
U(x,t) := \inf_{u \in \mathcal{U}(t)} P_{t,x}(u)  = \inf_{u \in \mathcal{U}(t)}\left \{ \int_t^T h (s, x(s), u(s)) \, ds + g(x(T))  \right\} \,,
\eqnx
and the Hamiltonian
\beqn
\tilde{H}(x,p,t) := \max_{u\in U} \{- \langle f (t,x,u) , p \rangle - h(t,x,u)\}  \text{ as well as } H(x,p,t) = \tilde{H}(x,p,T-t) \,.
\label{hahahahaha}
\eqn
This is actually the special case of the setting as discussed in Section \ref{sec3} when the set $C = \{0\}$ is a singleton, after we denote $b(\cdot)$ as $u(\cdot)$ instead. In this special case $H(t,x,p)$ is always convex w.r.t. $p$.

One point to note is that, the argument to get either a Lax and a Hopf formula in the general case with differential games as discussed in sec \ref{sec3} shall be similar with the standard assumption on the set of strategy following the causality.

\subsection{Verification of minimization/maximization principles under restricted assumptions} \label{proof_simple}
 
Before we go to our statement of conjectures, let us show that the formula 
in section \ref{sec1} holds for some restricted assumptions for the sake of completeness.  
A proof under less restricted assumptions can be found in \cite{proofimportant}.

In what follows, we first state the set of assumptions that we may use.  
For notational sake, let us write $$\mathcal{H} (x,p,u,s) := h (s, x, u) + \langle p, f (s, x, u ) \rangle \,. $$
Let us consider the following list of assumptions that we will consider:
\begin{enumerate}
\item[(A1)] 
$U$ is a compact convex set in $\mathbb{R}^d$;

\item[(A2)] 
$\mathcal{H} (x,p,u,s)  $ is proper lower semi-continuous and quasi-convex w.r.t. $u$.

\item[(A3)] 
$H (x,p,s) $ as defined in \eqref{hahahahaha} is proper upper semi-continuous and quasi-concave w.r.t. $x$.

\item[(A4)] 
$H (x,p,s) $ is equi-coercive (under parameters $(p,s)$ ) w.r.t. $x$ in the following sense: for all $N >0$, there exists $K$ (independent of $(s,p)$)
\beqnx
| H (x,p,s)  | \geq K 
\eqnx
whenever $\|x\| \geq N$. 

\item[(A5)] 
$g(x) $ is proper lower semi-continous and convex w.r.t. $x$, and is coercive w.r.t. $x$ in the following sense: 
\beqnx
\| g(x) \| \rightarrow \infty \text{ as } \| x \| \rightarrow \infty
\eqnx

\item[(A6)] 
$H (x,p,s) $ is proper concave w.r.t. $x$, and is $H(x,p,s)$ is equi-coercive (under parameters $(x,s)$ ) w.r.t. $p$.

\item[(A7)]
$\mathcal{H} (x,p,u,s) $, $H (x,p,s) $, $g(x)$ and $g^*(p)$ are all in $C^2$ in all its variables.



\end{enumerate}

\subsubsection{Verification of minimization principle under restricted assumptions}

Before we proceed, we would also like to state the following lemma directly from definition:
\begin{Lemma}  \label{lolXD2}
Under the assumption (A2), $H(x,p,t)$ as defined in \eqref{hahahahaha} is convex and lower semi-continuous w.r.t. $p$.
\end{Lemma}

\begin{proof}
Both properties follow from the fact that $\mathcal{H} (p,x,u,s)$ is linear w.r.t. $p$.  In fact, in order to check convexity, 
for all $ 0 \leq \lambda \leq 1$ and $p_1,p_2 \in \mathbb{R}^d$,
\beqnx
&  & H(x,\lambda p_1 + (1-\lambda) p_2 ,t )  \\
&=& \max_{u\in U} \{- \langle f (T-t,x,u) , \lambda p_1 + (1-\lambda) p_2 \rangle - h(T-t,x,u)\}  \\
&=& \max_{u\in U} \{- \langle f (T-t,x,u) , \lambda p_1 \rangle - \lambda h(T-t,x,u)  - \langle f (T-t,x,u) , (1-\lambda) p_2 \rangle - (1-\lambda) h(T-t,x,u) \}  \\
&\leq & \max_{u\in U} \{- \langle f (T-t,x,u) , \lambda p_1 \rangle - \lambda h(T-t,x,u) \}  + \max_{u\in U} \{ - \langle f (T-t,x,u) , (1-\lambda) p_2 \rangle - (1-\lambda) h(T-t,x,u) \}  \\
&= &\lambda H(x, p_1,t)  + (1-\lambda)  H(x, p_2,t ) \,.
\eqnx
Lower semi-continous w.r.t. $p$ follows from the fact that $ H(x,p,t) = \max_{u\in U} \{ - \mathcal{H} (T-t,x,p,u)  \}$ where for all $u$, $- \mathcal{H} (T-t,x,p,u) $ is lower semi-continous w.r.t. $p$.
\end{proof}

With these assumptions at hand, we have the following lemmas. We would like to remark that the following is a more discrete version of the minimax formula appeared in \cite{convexminimax}.

\begin{Lemma} \label{lolXD}

Let  $ (\textbf{X},\textbf{U},\textbf{P}) := \left( \{ x_n\}_{n=0}^{N-1}, \{ u^*_n\}_{n=0}^{N-1},   \{ p^(_n\}_{n =0}^{N-1} \right) \in \mathbb{R}^{3dN}$ and
$$F_1(\textbf{X},\textbf{U},\textbf{P}) := g(x_0)  + \delta \sum_{n=0}^{N-1} h (T-s_n, x_n, u_n ) + \sum_{n=0}^{N-1}  \langle p_n , x_{n+1} - x_n \rangle + \delta \sum_{n=0}^{N-1}  \langle p_n , f (T-s_n, x_n, u_n ) \rangle$$
and
$$\tilde{F_1}(\textbf{X},\textbf{P}) := g(x_0) + \sum_{n=0}^{N-1}   \langle p_n ,x_{n+1} - x_n \rangle  - \delta \sum_{n=0}^{N-1} H(x_n, p_n, s_n) $$
If (A1) and (A2) are satisfied, then we have
\begin{eqnarray}
 \min_{ \{ u_n\}_{n=0}^{N-1} \in U^N }  \inf_{ \{ x_n\}_{n=0}^{N-1} }    \sup_{ \{ p_n\}_{n =0}^{N-1} } F_1 (\textbf{X},\textbf{U},\textbf{P}) =  \inf_{ \{ x_n\}_{n=0}^{N-1} }  \sup_{ \{ p_n\}_{n=0}^{N-1} }  \tilde{F_1}(\textbf{X},\textbf{P})   \,.
\label{lolXDform}
\end{eqnarray}

\end{Lemma}

\begin{proof}

Fixing $\textbf{X}$, consider the function $F_1(\textbf{X},\cdot,\cdot): (\textbf{U},\textbf{P})\mapsto F_1(\textbf{X},\textbf{U},\textbf{P})$. 
Since 
$\mathcal{H} (s,p,x,u)$ satisfies (A2), we have $F_1(\textbf{X},\textbf{U},\textbf{P})$ is lower-semicontinous and quasi-convex w.r.t. $U$. Moreover since $U$ satisfies (A1), we have that $U^N$ is compact and convex.  It is clear that since $F$ is linear w.r.t. $P$, $F_1(\textbf{X},\textbf{U},\textbf{P})$ is upper-semicontinous and quasi-concave w.r.t. $P$.   Therefore we may apply Sion's minimax theorem \cite{sion1,sion2} to obtain that for a fixed $X$, we have
\begin{eqnarray*}
  \inf_{ \{ u_n\}_{n=0}^{N-1} \in U }  \sup_{ \{ p_n\}_{n =0}^{N-1} } F_1(\textbf{X},\textbf{U},\textbf{P})  =  \sup_{ \{ p_n\}_{n=0}^{N-1} }   \inf_{ \{ u_n\}_{n=0}^{N-1} \in U } F_1(\textbf{X},\textbf{U},\textbf{P})  \,.
\end{eqnarray*}
Now by definition of $H(t,x,p)$ in \eqref{hahahahaha}
\beqnx
& & \inf_{ \{ u_n\}_{n=0}^{N-1} \in U } F_1(\textbf{X},\textbf{U},\textbf{P}) \\
&=&  g(x_0) - \delta \sum_{n=1}^{N-1} \max_{ u \in U } \left\{ - \langle p_n , f (T- s_n, x_n, u) \rangle - h (T- s_n, x_n, u) \right\} +  \sum_{n=1}^{N-1}   \langle p_n , x_{n+1} - x_n \rangle  \,, \\
&=&  \tilde{F_1}(\textbf{X},\textbf{P})  \,, 
\eqnx
The result now follows by taking infrimum at both hand sides w.r.t $\textbf{X}$ and that $\min$ and $\inf$ can swap.

\end{proof}

Notice Lemma \ref{lolXD} is a version of representation formula with a minimal amount of assumption, and this expressions gives the finite dimensional analogy well-known principle of least action as in e.g. \cite{convexminimax,rockafellar1}.

\begin{Lemma} \label{halolXD} 
If (A1), (A2), (A4), (A6), (A7) are satisfied, then we have
\begin{eqnarray}
& & \min_{ \{ u_n\}_{n=0}^{N-1} \in U^N }  \inf_{ \{ x_n\}_{n=0}^{N-1} }    \sup_{ \{ p_n\}_{n =0}^{N-1} } F_1 (\textbf{X},\textbf{U},\textbf{P})  \notag  \\
&=& \inf_{ v \in \mathbb{R}^d }    \bigg \{ g(x_0) + \delta \sum_{n=0}^{N-1}   \langle p_n , \frac{x_{n+1} - x_n}{\delta} \rangle  - \delta \sum_{n=0}^{N-1} H(x_n, p_n, s_n) :  \notag  \\
& & \qquad \qquad
 \begin{matrix} x_{n+1}(x,v) - x_n(x,v) =  \delta \partial_p H(x_n(x,v), p_n(x,v), s_n), \\ p_{n-1}(x,v) - p_n(x,v) = \delta \partial_x H(x_n(x,v), p_n(x,v), s_n), \\ x_N= x,  p_{N} = v \end{matrix}  \bigg\} \label{hahalolXDform} \,.
\end{eqnarray}

\end{Lemma}

\begin{proof}

Let $ (\tilde{ \textbf{X} }, \textbf{P}) := \left( \{ x_n\}_{n=1}^{N-1},   \{ p_n\}_{n =0}^{N-1} \right) \in \mathbb{R}^{(2N-1) d}$.
From by Lemma \ref{lolXD}, since (A1) and (A2), are satisfied, we have that \eqref{lolXDform} holds.  Therefore it remains to show that 
the term $ \inf_{ \{ x_n\}_{n=0}^{N-1} }  \sup_{ \{ p_n\}_{n=0}^{N-1} }  \tilde{F_1}(\textbf{X},\textbf{P}) $ equals to \eqref{hahalolXDform}.  Notice that
$$ \inf_{ \{ x_n\}_{n=0}^{N-1} }  \sup_{ \{ p_n\}_{n=0}^{N-1} }  \tilde{F_1}(\textbf{X},\textbf{P}) = \inf_{ x_0 }   \inf_{ \{ x_n\}_{n=1}^{N-1} }  \sup_{ \{ p_n\}_{n=0}^{N-1} }  \tilde{F_1}( x_0 , \tilde{ \textbf{X} },\textbf{P}) \,. $$

We now wish to argue that under assumptions (A6), (A7) and (A8), for a fixed $x_0$, either that there exists  $(\tilde{\textbf{X}}^*(x_0),\textbf{P}^*(x_0))$ (depending on $x_0$) such that the mini-max problem $ \inf_{ \{ x_n\}_{n=1}^{N-1} }  \sup_{ \{ p_n\}_{n=0}^{N-1} }  \tilde{F_1}( x_0 , \tilde{ \textbf{X} },\textbf{P}) =  \tilde{F_1}( x_0 , \tilde{ \textbf{X} }^*(x_0),\textbf{P}^*(x_0)) $, or one has $\inf_{ \{ x_n\}_{n=1}^{N-1} }  \sup_{ \{ p_n\}_{n=0}^{N-1} }  \tilde{F_1}(\textbf{X},\textbf{P}) = \infty$.


In fact, for a fixed set of $(x_0 , \tilde{ \textbf{X} }) $, by (A6), either $ \sup_{ \{ p_n\}_{n=0}^{N-1} }  \tilde{F_1}( x_0 , \tilde{ \textbf{X} },\textbf{P}) $ is attained and thus, by (A7) there is $\text{P}^* ( x_0, \tilde{ \textbf{X} } ) $ smoothly depending on $( x_0, \tilde{ \textbf{X} } )$ s.t. $$x_{n+1}(x,v) - x_n(x,v) =  \delta \partial_p H(x_n(x,v), p_n^*(x,v, ( x_0, \tilde{ \textbf{X} } )  ), s_n)$$ holds, or the supremum is infinity.  Now let us take infrimum over $ \tilde{ \textbf{X} }$.  For $x_0$ such that for all $ \tilde{ \textbf{X} }$ the value $ \sup_{ \{ p_n\}_{n=0}^{N-1} }  \tilde{F_1}( x_0 , \tilde{ \textbf{X} },\textbf{P}) = \infty $, we have that $\inf_{ \{ x_n\}_{n=1}^{N-1} }  \sup_{ \{ p_n\}_{n=0}^{N-1} }  \tilde{F_1}( x_0 , \tilde{ \textbf{X} },\textbf{P})  = \infty$.  Otherwise, for $x_0$ such that there exists $ \tilde{ \textbf{X} }$ with $ \sup_{ \{ p_n\}_{n=0}^{N-1} }  \tilde{F_1}( x_0 , \tilde{ \textbf{X} },\textbf{P}) \neq \infty $, by (A4) and (A7), we again either have $ \tilde{ \textbf{X} }^*$ satisfies $$p_{n-1}^* (x,v, ( x_0, \tilde{ \textbf{X} }^* )  ) - p_n^*(x,v, ( x_0, \tilde{ \textbf{X} }^* )  ) = \delta \partial_x H(x_n^*(x,v, x_0), p_n^*(x,v, ( x_0, \tilde{ \textbf{X} }^* )  ), s_n) \, ,$$ or that $\inf_{ \{ x_n\}_{n=1}^{N-1} }  \sup_{ \{ p_n\}_{n=0}^{N-1} }  \tilde{F_1}( x_0 , \tilde{ \textbf{X} },\textbf{P})  = -\infty$.  However the case that the infrimum get to $-\infty$ will arrive at absurdity since we have tracing back Lemma \ref{lolXD}
\beqnx
 -\infty <  \inf_{ \{ u_0\}_{n=1}^{N-1} }   \inf_{ \{ x_n\}_{n=1}^{N-1} }  \sup_{ \{ p_n\}_{n=0}^{N-1} }  \tilde{F_1}( x_0 , \tilde{ \textbf{X} },\textbf{P}) =  \inf_{ \{ x_n\}_{n=0}^{N-1} }  \sup_{ \{ p_n\}_{n=0}^{N-1} }  \tilde{F_1}(\textbf{X},\textbf{P})   =   -\infty
\eqnx
which arrives at contradiction.

Concluding the above argument, for each $x_0$, either we have $ \inf_{ \{ x_n\}_{n=1}^{N-1} }  \sup_{ \{ p_n\}_{n=0}^{N-1} }  \tilde{F_1}( x_0 , \tilde{ \textbf{X} },\textbf{P}) = \infty$ or $ \inf_{ \{ x_n\}_{n=1}^{N-1} }  \sup_{ \{ p_n\}_{n=0}^{N-1} }  \tilde{F_1}( x_0 , \tilde{ \textbf{X} },\textbf{P})=  \tilde{F_1}( x_0 , \tilde{ \textbf{X} }^*(x_0),\textbf{P}^*(x_0)) $ for some $(\textbf{X}^* (x_0),\textbf{P}^*(x_0)) = (\tilde{\textbf{X}}^* (x_0),\textbf{P}^*(x_0,\tilde{\textbf{X}}^*))$ s.t. 
\beqn
 \begin{cases} x_{n+1}(x,v) - x_n(x,v) =  \delta \partial_p H(x_n(x,v), p_n(x,v), s_n), \\ p_{n-1}(x,v) - p_n(x,v) = \delta \partial_x H(x_n(x,v), p_n(x,v), s_n), \\ x_N= x \end{cases} 
\label{again}
\eqn
for all $n = 0,1,..,N-1$. (Notice the condition as the initial condition $\partial_x g(x_0) = p_0$ do not appear because we fixed a value $x_0$.)
Now for any choice of $(\textbf{X} (x_0),\textbf{P}(x_0))$ satisfying \eqref{again} will give the same value $ \tilde{F_1}( x_0 , \tilde{ \textbf{X} }(x_0),\textbf{P}(x_0))$ by concavity of $\tilde{F_1}( x_0 , \tilde{ \textbf{X} },\textbf{P})$ w.r.t. $\textbf{P}$.
Therefore,
$$
\inf_{ x_0}   \inf_{ \{ x_n\}_{n=1}^{N-1} }  \sup_{ \{ p_n\}_{n=0}^{N-1} }  \tilde{F_1}( x_0 , \tilde{ \textbf{X} },\textbf{P}) =\min \left \{ \infty,  \inf_{ x_0 \in \{ x_0: \exists (\textbf{X}^* (x_0),\textbf{P}^*(x_0)) \text{ satisfying } \eqref{again} \} } \inf_{v_N}   \tilde{F_1}( x_0 , \tilde{ \textbf{X} }^*(x_0),\textbf{P}^*(x_0))
 \right\}$$
The conclusion of the lemma follows from the surjection between $v \in \mathbb{R}^d$ and $ x_0 \in \{ x_0: \exists (\textbf{X}^* (x_0),\textbf{P}^*(x_0)) \text{ satisfying } \eqref{again} \} $ via the correspondence $p_N^*(x_0) = v$.

\end{proof}

We would like to remark that in fact if $\mathcal{H}$ is non-smooth w.r.t. $u$, even $ H(x, p, s) =  \max_{u \in U} \{  - \mathcal{H} (x,p,u,T-s) \} $, in general we do not have $ \partial_p H(x, p, s) = -  \partial_p   \mathcal{H} (x,p,u^*,T-s) \} $ where $u^*$ is such that $ \mathcal{H} (x,p,u^*,T-s)  = \max_{u \in U} \{  - \mathcal{H} (x,p,u,T-s) \} $.  Therefore in general the KKT condition (which is also referred to as Pontryagin's maximum principle \cite{pontryagin} in the continuous case) would be stated using $\mathcal{H}$.  However under appropriate regularity assumption (e.g. (A2), (A6) and (A8)) , we have the two partial derivatives coinciding, ie.. $ \partial_p H(x, p, s) = -  \partial_p   \mathcal{H} (x,p,u^*,T-s) \} $.  
A clear connection made between the maximum principle and dynamic programming is discussed in literature, e.g. in \cite{clarke}, which discussed in the non-smooth case when will the costate variable coincide with a partial generalized gradient of the value function (c.f.  \cite{clarke}) for a clear definition.

\subsubsection{Verification of maximization principle under restricted assumptions}

Using a similar but slightly different argument as in Lemma \ref{lolXD}, we also obtain similarly the following lemma:

\begin{Lemma} \label{hahalolXD}

Write $ (\textbf{X},\textbf{U},\textbf{P}):= \left( \{ x_n\}_{n=0}^{N-1}, \{ u^*_n\}_{n=0}^{N-1},   \{ p^(_n\}_{n =1}^{N} \right) \in \mathbb{R}^{3dN}$, $ (\tilde{\textbf{X}},\textbf{U},\textbf{P}):= \left( \{ x_n\}_{n=1}^{N-1}, \{ u^*_n\}_{n=0}^{N-1},   \{ p^(_n\}_{n =1}^{N} \right) \in \mathbb{R}^{(3N-1)d}$, 
$$F_2(\textbf{X},\textbf{U},\textbf{P}) := g(x_0)  + \delta \sum_{n=0}^{N-1} h (T- s_{n+1}, x_{n+1}, u_{n+1} ) + \sum_{n=0}^{N-1}  \langle p_{n+1} , x_{n+1} - x_n \rangle + \delta \sum_{n=0}^{N-1}  \langle p_{n+1} , f (T- s_{n+1}, x_{n+1}, u_{n+1} ) \rangle \,,$$ and
$$
\tilde{\tilde{F_2}}(\tilde{\textbf{X}},\textbf{P}) := \langle p_N , x \rangle  -g^*(p_1) - \delta \sum_{n=1}^{N} H(x_{n}, p_n, s_{n})  + \sum_{n=1}^{N-1}   \langle p_{n} - p_{n+1}, x_{n}  \rangle 
$$
If (A1), (A2), (A3), (A4) and (A5) are satisfied, then we have
then we have
\begin{eqnarray}
 \min_{ \{ u_n\}_{n=0}^{N-1} \in U^N } 
\inf_{ \{ x_n\}_{n=0}^{N-1} } \sup_{ \{ p_n\}_{n=1}^{N} } F_2(\textbf{X},\textbf{U},\textbf{P})  =   \sup_{ \{ p_n\}_{n=1}^{N} }    \inf_{ \{ x_n\}_{n=1}^{N-1} } \tilde{\tilde{F_2}}(\tilde{\textbf{X}},\textbf{P})   \,. \label{wawawawa} 
\end{eqnarray}
\end{Lemma}

\begin{proof}

The first part of the argument goes similar as in the proof of the previous lemma. 
In fact, fixing $\textbf{X}$, with the map $F_2 (\textbf{X},\cdot,\cdot): (\textbf{U},\textbf{P})\mapsto F_2(\textbf{X},\textbf{U},\textbf{P})$ being lower-semicontinous and quasi-convex w.r.t. $U$ (from (A2)) and upper-semicontinous and quasi-concave w.r.t. $P$ (from linearity), as well as $U^N$ being compact and convex (from (A1)), we have again that:
\begin{eqnarray*}
  \inf_{ \{ u_n\}_{n=0}^{N-1} \in U }  \sup_{ \{ p_n\}_{n =1}^{N} } F_2(\textbf{X},\textbf{U},\textbf{P})  =  \sup_{ \{ p_n\}_{n=1}^{N} }   \inf_{ \{ u_n\}_{n=0}^{N-1} \in U } F_2(\textbf{X},\textbf{U},\textbf{P})  \,.
\end{eqnarray*}
Again from of $H(t,x,p)$ in \eqref{hahahahaha}, the same argument as in the previous lemma gives
\beqnx
 \inf_{ \{ u_n\}_{n=0}^{N-1} \in U } F_2(\textbf{X},\textbf{U},\textbf{P})  =  \tilde{F_2}(\textbf{X},\textbf{P})  \,, 
\eqnx
where
$$
\tilde{F_2}(\textbf{X},\textbf{P}) :=
g(x_0)  -  \langle p_1 , x_0 \rangle  - \delta \sum_{n=0}^{N-1} H(x_{n+1}, p_{n+1}, s_{n+1})  +  \sum_{n=1}^{N-1}   \langle p_n - p_{n+1} , x_{n}  \rangle  + \langle p_{N} , x \rangle 
$$
Hence  swapping $\min$ and $\inf$, we get
\begin{eqnarray*}
 \min_{ \{ u_n\}_{n=0}^{N-1} \in U^N } \inf_{ \{ x_n\}_{n=0}^{N-1} } \sup_{ \{ p_n\}_{n=1}^{N} } F_2(\textbf{X},\textbf{U},\textbf{P})  =   \inf_{ \{ x_n\}_{n=0}^{N-1} }  \sup_{ \{ p_n\}_{n=1}^{N} }   \tilde{F_2}(\textbf{X},\textbf{P})  \,.
\end{eqnarray*}
after an application of summation by part.
Now since $H (s,x,p) $ satisfies (A3), we have $  \tilde{F_2}(\textbf{X},\textbf{P})$ is lower-semicontinous and quasi-convex w.r.t. $X$. 
Since (A4) and (A5) are satisfied, there exists a compact (and convex, w.l.o.g.) $C^N \subset \mathbb{R}^{dN}$  such that
\beqnx
\inf_{ \{ x_n\}_{n=0}^{N-1} }  \sup_{ \{ p_n\}_{n=1}^{N} }   \tilde{F_2}(\textbf{X},\textbf{P})  = \inf_{ \{ x_n\}_{n=0}^{N-1}  \in C^N }  \sup_{ \{ p_n\}_{n=1}^{N} }   \tilde{F_2}(\textbf{X},\textbf{P})
\eqnx
From Lemma \ref{lolXD2} $\tilde{F_2}(X,U,P)$ is upper-semicontinous and quasi-concave w.r.t. $P$.   Therefore we may apply Sion's minimax theorem \cite{sion1,sion2} to get
\begin{eqnarray*}
 \inf_{ \{ x_n\}_{n=0}^{N-1}  \in C^N }  \sup_{ \{ p_n\}_{n=1}^{N} }   \tilde{F_2}(\textbf{X},\textbf{P}) =  \sup_{ \{ p_n\}_{n=1}^{N} }  \inf_{ \{ x_n\}_{n=0}^{N-1}  \in C^N }   \tilde{F_2}(\textbf{X},\textbf{P})
\end{eqnarray*}
and in fact $ \inf_{ \{ x_n\}_{n=0}^{N-1}  \in C^N }  $ can be replaced by $ \inf_{ \{ x_n\}_{n=0}^{N-1}  \in C^N } $.  Now considering (A4) and (A5) again, we obtain
$$  \sup_{ \{ p_n\}_{n=1}^{N} }  \inf_{ \{ x_n\}_{n=0}^{N-1}  \in C^N }   \tilde{F_2}(\textbf{X},\textbf{P}) =   \sup_{ \{ p_n\}_{n=1}^{N} }  \inf_{ \{ x_n\}_{n=0}^{N-1} }   \tilde{F_2}(\textbf{X},\textbf{P})   \,. $$
Our conclusion now comes from the fact that
$$   \inf_{  x_0 }   \tilde{F_2}(\textbf{X},\textbf{P})   =  \tilde{\tilde{F_2}}(\tilde{\textbf{X}},\textbf{P}) $$
by the definition of Fenchel-Legendre transform.
\end{proof}

\begin{Lemma}  \label{hahalolXDsayonara}

If (A1), (A2), (A4),(A5), (A6), (A7) are satisfied, then we have
\begin{eqnarray}
& &  \min_{ \{ u_n\}_{n=0}^{N-1} \in U^N } 
\inf_{ \{ x_n\}_{n=0}^{N-1} } \sup_{ \{ p_n\}_{n=1}^{N} } F_2(\textbf{X},\textbf{U},\textbf{P})   \notag  \\
&=& \sup_{ v \in \mathbb{R}^d }    \bigg \{  \langle p_N , x \rangle  -g^*(p_1) - \delta \sum_{n=1}^{N} H(x_{n}, p_n, s_{n})  + \delta \sum_{n=1}^{N-1}   \langle \frac{p_{n} - p_{n+1}}{\delta} , x_{n}  \rangle   :  \notag \\
& & \qquad \qquad
 \begin{matrix} x_{n}(x,v) - x_{n-1}(x,v) =  \delta \partial_p H(x_{n}(x,v), p_n(x,v), t_{n}), \\ p_{n}(x,v) - p_{n+1}(x,v) =  \delta \partial_x H(x_{n}(x,v), p_{n}(x,v), t_{n}), \\ x_N= x,  p_N = v \end{matrix}  \bigg\}   \label{hahalolXDform2} \,.
\end{eqnarray}

\end{Lemma}

\begin{proof}

Let $ (\tilde{ \textbf{X} }, \tilde{\textbf{P} } ) := \left( \{ x_n\}_{n=1}^{N-1},   \{ p_n\}_{n =0}^{N-1} \right) \in \mathbb{R}^{(2N-2) d}$.
Again since (A6) and (A7) imply (A3), from Lemma \ref{hahalolXD},  we have that \eqref{wawawawa} holds.  
Therefore it remains to show that 
the term $ \sup_{ \{ p_n\}_{n=1}^{N} }    \inf_{ \{ x_n\}_{n=1}^{N-1} } \tilde{\tilde{F_2}}(\tilde{\textbf{X}},\textbf{P})  $ equals to \eqref{hahalolXDform2}.  Notice again that
$$\sup_{ \{ p_n\}_{n=1}^{N} }    \inf_{ \{ x_n\}_{n=1}^{N-1} } \tilde{\tilde{F_2}}(\tilde{\textbf{X}},\textbf{P})  = \sup_{ p_1 }   \sup_{ \{ p_n\}_{n=2}^{N} }    \inf_{ \{ x_n\}_{n=1}^{N-1} } \tilde{\tilde{F_2}}(\tilde{\textbf{X}},p_1,\tilde{\textbf{P}})   \,. $$
Now as in lemma \ref{hahalolXD}, we now wish to argue that under assumptions (A6), (A7) and (A8), for a fixed $p_1$, either that there exists  $(\tilde{\textbf{X}}^*(p_1),\tilde{\textbf{P}}^*(p_1))$ (depending on $v_1$) such that the max-min problem $ \sup_{ \{ p_n\}_{n=2}^{N} }    \inf_{ \{ x_n\}_{n=1}^{N-1} } \tilde{\tilde{F_2}}(\tilde{\textbf{X}},p_1,\tilde{\textbf{P}})  = 
\tilde{\tilde{F_2}}(\tilde{ \textbf{X} }^*(p_1),p_1,\tilde{ \textbf{P} }^*(p_1))$, or one has $ \sup_{ \{ p_n\}_{n=2}^{N} }    \inf_{ \{ x_n\}_{n=1}^{N-1} } \tilde{\tilde{F_2}}(\tilde{\textbf{X}},p_1,\tilde{\textbf{P}})  = -\infty$.

Now, for a fixed set of $(p_1 , \tilde{ \textbf{P} }) $, by (A4) and (A5), either $ \inf_{ \{ x_n\}_{n=1}^{N-1} }  \tilde{\tilde{F_2}}(\tilde{\textbf{X}},p_1,\tilde{\textbf{P}}) $ is attained and thus, by (A7) there is $\tilde{ \text{X}}^* ( p_1, \tilde{ \textbf{P} } ) $ smoothly depending on $( x_0, \tilde{ \textbf{X} } )$ s.t. 
$$p_{n}^* (x,v ) - p_{n+1}^*(x,v ) = \delta \partial_x H(x_n^*(x,v,  ( p_1, \tilde{ \textbf{P} } ) ), p_n(x,v ), s_n)  \, ,$$
holds, or the infrimum is minus infinity.  Now let us take supremum over $ \tilde{ \textbf{P} }$.  For $p_1$ such that for all $ \tilde{ \textbf{P} }$ the value $ \inf_{ \{ p_n\}_{n=0}^{N-1} }  \tilde{\tilde{F_2}}( \tilde{ \textbf{X} },  p_1 ,\tilde{\textbf{P}}) =  - \infty $, we have that $ \sup_{ \{ p_n\}_{n=2}^{N} }    \inf_{ \{ x_n\}_{n=1}^{N-1} } \tilde{\tilde{F_2}}(\tilde{\textbf{X}},p_1,\tilde{\textbf{P}})   = - \infty$.  Otherwise, for $p_1$ such that there exists $ \tilde{ \textbf{P} }$ with $\inf_{ \{ x_n\}_{n=1}^{N-1} } \tilde{\tilde{F_2}}(\tilde{\textbf{X}},p_1,\tilde{\textbf{P}})  \neq  - \infty $, by (A6) and (A7), we again either have $ \tilde{ \textbf{X} }^*$ satisfies 
$$x_{n}^* (x,v, ( p_1, \tilde{ \textbf{P} }^* )  ) - x_{n-1}^*(x,v, ( p_1, \tilde{ \textbf{P} }^* )  ) = \delta \partial_p H(x_n^*(x,v, ( p_1, \tilde{ \textbf{P} }^* ) ), p_n^*(x,v, p_1 ), s_n) \, ,$$
or that $ \sup_{ \{ p_n\}_{n=2}^{N} }  \inf_{ \{ x_n\}_{n=1}^{N-1} } \tilde{\tilde{F_2}}(\tilde{\textbf{X}},p_1,\tilde{\textbf{P}})  =  \infty$.  However again case that the supremum get to $\infty$ will arrive at absurdity since we have tracing back Lemma \ref{lolXD}
\beqnx
\infty >  \inf_{ \{ u_0\}_{n=1}^{N-1} }   \inf_{ \{ x_n\}_{n=1}^{N-1} }  \sup_{ \{ p_n\}_{n=0}^{N-1} }  \tilde{F_1}( x_0 , \tilde{ \textbf{X} },\textbf{P}) =   \sup_{ \{ p_n\}_{n=1}^{N} }  \inf_{ \{ x_n\}_{n=1}^{N-1} } \tilde{\tilde{F_2}}(\tilde{\textbf{X}},p_1,\tilde{\textbf{P}})  =  \infty
\eqnx
which arrives at contradiction.

Concluding the above argument, for each $p_1$, either we have $ \sup_{ \{ p_n\}_{n=2}^{N} }    \inf_{ \{ x_n\}_{n=1}^{N-1} } \tilde{\tilde{F_2}}(\tilde{\textbf{X}},p_1,\tilde{\textbf{P}})   = - \infty$ or $  \sup_{ \{ p_n\}_{n=2}^{N} }    \inf_{ \{ x_n\}_{n=1}^{N-1} } \tilde{\tilde{F_2}}(\tilde{\textbf{X}},p_1,\tilde{\textbf{P}})  = 
\tilde{\tilde{F_2}}(\tilde{ \textbf{X} }^*(p_1),p_1,\tilde{ \textbf{P} }^*(p_1)) $ for some $(\textbf{X}^* (p_1),\textbf{P}^*(p_1)) = (\tilde{\textbf{X}}^* (p_1),\textbf{P}^*(p_1,\tilde{\textbf{X}}^*))$ s.t. 
\beqn
 \begin{cases}
x_{n}(x,v) - x_{n-1}(x,v) =  \delta \partial_p H(x_{n}(x,v), p_n(x,v), t_{n}), \\ p_{n}(x,v) - p_{n+1}(x,v) =  \delta \partial_x H(x_{n}(x,v), p_{n}(x,v), t_{n}), \\ x_N= x
\end{cases} 
\label{againaa}
\eqn
for all $n = 0,1,..,N-1$. (Notice the condition as the initial condition $\partial_x g^*(p_1) = x_1$ do not appear because we fixed a value $p_1$.)
Now for any choice of $(\textbf{X} (x_0),\textbf{P}(x_0))$ satisfying \eqref{again} will give the same value $ \tilde{F_1}( x_0 , \tilde{ \textbf{X} }(x_0),\textbf{P}(x_0))$ by concavity of $\tilde{F_1}( x_0 , \tilde{ \textbf{X} },\textbf{P})$ w.r.t. $\textbf{P}$.
Therefore,
$$
\sup_{ p_1 }   \sup_{ \{ p_n\}_{n=2}^{N} }    \inf_{ \{ x_n\}_{n=1}^{N-1} } \tilde{\tilde{F_2}}(\tilde{\textbf{X}},p_1,\tilde{\textbf{P}})   =\max \left \{ - \infty,  \sup_{ p_1 \in \{ p_1: \exists (\textbf{X}^* (p_1),\textbf{P}^*(p_1)) \text{ satisfying } \eqref{againaa} \} } \sup_{v_N} 
\tilde{\tilde{F_2}}(\tilde{\textbf{X}}^*,p_1,\tilde{\textbf{P}}^* )
 \right\}$$
The conclusion of the lemma follows from the surjection between $v \in \mathbb{R}^d$ and $ p_1 \in \{ p_1: \exists (\textbf{X}^* (p_1),\textbf{P}^*(p_1)) \text{ satisfying } \eqref{againaa} \}  $ via the correspondence $p_N^*(p_1) = v$.

\end{proof}

\subsection{Generalized Lax minimization principle for Convex/Concave Hamiltonian} \label{con/con_ham}

We first describe how we obtain to a Lax formula under restricted assumptions and applying Lemma \ref{lolXD} and Lemma \ref{hahalolXD}:
\beqnx
 \varphi(x,t) :=
U(x,T-t) &:=& \inf_{u \in \mathcal{U}(T-t)} \left \{ \int_{T-t}^T h (s, x(s), u(s)) \, ds + g(x(T))  \right\}    \,.
\eqnx
We devive our formula as follows.
Following \cite{boyd}, we shall first consider the following discretization (approximation) for a given $\delta$ such that $\delta N = t$, by denoting $s_n = \delta n $ and $x_N = x$ (and flipping the sign),
$$ \varphi(x,t)
\approx \varphi^1_N(x,t) $$
where
\beqnx
 & & \varphi^1_N(x,t)\\
&:=& \min_{ \{ u_n\}_{n=0}^{N-1} \in U} \left \{ g(x_0)  + \delta \sum_{n=0}^{N-1} h (T-s_n, x_n, u_n )  : x_{n+1} - x_n = - \delta f (T-s_n, x_n, u_n ) \text{ for }
n=0,\ldots,N-1 \,, x_N = x  \right\} \,, \\
&=& \min_{ \{ u_n\}_{n=0}^{N-1} \in U^N }  \inf_{ \{ x_n\}_{n=0}^{N-1} }    \sup_{ \{ p_n\}_{n =0}^{N-1} } F_1 (\textbf{X},\textbf{U},\textbf{P})  \,.
\eqnx
where $F_1$ is given as in Lemma \ref{lolXD}, and the second equality comes from reformulating the problem with Lagrange multiplier.

Therefore, applying Lemma \ref{lolXD}  and Lemma \ref{hahalolXD}, we get the following:


\begin{Theorem}
If (A1), (A2) are satisifed, then 
\begin{eqnarray*}
 \varphi^1_N(x,t) &=&  \inf_{ \{ x_n\}_{n=0}^{N-1} }  \sup_{ \{ p_n\}_{n=0}^{N-1} }  \tilde{F_1}(\textbf{X},\textbf{P})   \,.
\end{eqnarray*}
If (A1), (A2), (A4), (A6), (A7) are satisfied, then we have
\begin{eqnarray}
 \varphi^1_N(x,t) &=& \inf_{ v \in \mathbb{R}^d }    \bigg \{ g(x_0) + \delta \sum_{n=0}^{N-1}   \langle p_n , \frac{x_{n+1} - x_n}{\delta} \rangle  - \delta \sum_{n=0}^{N-1} H(x_n, p_n, s_n) :  \notag  \\
& & \qquad \qquad
 \begin{matrix} x_{n+1}(x,v) - x_n(x,v) =  \delta \partial_p H(x_n(x,v), p_n(x,v), s_n), \\ p_{n-1}(x,v) - p_n(x,v) = \delta \partial_x H(x_n(x,v), p_n(x,v), s_n), \\ x_N= x,  p_{N} = v \end{matrix}  \bigg\} \,.
\end{eqnarray}
\end{Theorem}
\noindent We would like to remark that (A5) is not needed for the validity of the above formulae.

However, we notice that the resulting formula that we conjectured seems to be correct beyond these assumptions, as the numerical results show (especially when we take the minimum over all the paths satisfying the KKT conditions). We hope to get rigorous criteria for these formula to hold in the future.  In fact, passing to the limit, in the special case when $H(x,p,t)$ is smooth also w.r.t $p$, we conjecture the following Lax formula.

\begin{Conjecture} \label{conj1} Assume $H(x,p,t) \in C^2$ and is convex w.r.t. $p$, and (A5) is satisfied, then there exists $t_0$ such that the viscosity solution to \eqref{HJeqn}-\eqref{HJinitial} can be represented as \eqref{lax_formula} for $t \leq t_0$.
Moreover, if $\phi(x,t)$ is differentiable w.r.t. $x$ at a neighbourhood of $(x,t)$ and the infrimium is attained by $\tilde{v}$, then we have
$ \partial_x  \varphi(x,t) =\tilde{v} $,
\end{Conjecture}
\noindent We refer to it as the \textbf{minimization principle}, or the generalized Lax formula.
This conjecture was validated under less restricted assumption in \cite{proofimportant} after a previous version of our paper in arXiv in \cite{preprint} was launched.  \cite{proofimportant}  is a very important and monumental work which suggest the validity of our conjecture in a more general situation.

\noindent \textbf{Remark 1}: When $H(x,p,t)$ is not differentiable at some given point $p$, then we believe that in formula \eqref{lax_formula}, the Mordukhovich subdifferential, $ \partial_x^{-} H$, as defined in \cite{subdiff,subdiff2,subdiff3}, should be used instead of $ \partial_x H$.  In that case the constraint becomes the inclusion $\dot{ \gamma }(x,v,s) \in \partial_p^- H( \gamma(x,v,s), p(x,v,s),s)$, and infrimum is taken over also all the curves $(\gamma,p)\in C^\infty$ satisfying the inclusion.

In below there are several examples for the conjecture, that we only provide a brief account.
\vskip 5mm

\noindent
\textbf{Example 1}
For Hamiltonian $H(x,p,t)$ which is convex w.r.t. $p$, concave w.r.t. $x$ and satisfying Assumptions (A) in \cite{rockafellar1}, the following minimization principle holds for the viscosity solution to \eqref{HJeqn}-\eqref{HJinitial} (see Theorem 2.3 in \cite{rockafellar1} and Theorem 4.8 in \cite{rockafellar2}):
\beqnx
 \varphi(x,t) 
&=& \inf_{\gamma \in C^\infty, \gamma(t) = x } \{ g(  \gamma(0)  ) + \int_{0}^t L(\gamma(s),\dot{\gamma}(s ),s) ds \} 
\eqnx
where the Lagrangian $L$ is defined as
\beqnx
L(x,q,s) = \sup_{p} \{ \langle p,q \rangle - H(x,p,t) \} \,.
\eqnx
This example may not satisfy either the assumptions in Lemma \ref{lolXD} and or Lemma \ref{hahalolXD}, since $U$ may not be compact.

\noindent
\textbf{Example 2}
When $H(x,p,t)$ is a convex homogeneous degree-$1$ functional w.r.t. $p$ of the following special form
\begin{equation*}
H(x,p,t) = c(x) \Phi (p ) \,,
\end{equation*}
where $c \in C^\infty$ with $C_0 \geq c(x) \geq c_0$ for some $C_0, c_0 >0$, and $\Phi$ is homogeneous of degree $1$ functional.  We recall the definition of the Wulff set $W$ \cite{Wulff} of $\Phi$ defined as the set $W$ such that the following equality holds:
\beqnx
\Phi(p) = \max_{x \in W} \langle p,x \rangle =  \mathcal{I}_W^*(p)\,.
\eqnx
with $*$ denoting the Fenchel-Legendre transform.
We furthermore assume closed, strictly convex, balanced (i.e. $-W = W$), absorbing (i.e. for all $y \in \mathbb{R}^d$, $y \in \tau W$ for some $\tau > 0$) with smooth boundary $\partial W \in C^\infty$ \cite{gauge,Hiriart}.
Then it is ready to check that the subdifferential of $\Phi$ is given as follows:
\beqnx
\partial_p^- \Phi(p) = 
\begin{cases}
\partial_p \Phi(p) \in \partial W & \text{ if } p \neq 0 \\
W & \text{ if } p = 0 \,. \\
\end{cases} 
\eqnx
For convenience sake, let us also define, for a closed, strictly convex, balanced, absorbing set $W$ with smooth boundary $\partial W \in C^\infty$, the Minkowski functional of $W$ \cite{gauge,Hiriart}, as follows
\beqnx
\rho_{ W } (y) : = \inf \{ \tau > 0 : y \in \tau W\}  \,.
\eqnx
With this, we are ready to define a metric on $\mathbb{R}^d$ as follows:
\beqnx
\tilde{d}(x,y) := \inf_{ t > 0 }\left\{ t :  \gamma \in C^\infty, \gamma(0) : = x, \gamma(t) = y,   \dot{ \gamma} (s)  \in c( \gamma (s) ) W  \right\}  \,. 
\eqnx
It is ready to check that $\tilde{d}$ defines a metric and thus $(\mathbb{R} , \tilde{d} )$ forms a metric space.  

\begin{Lemma}
\label{metric}
Assume the metric space $(\mathbb{R} , \tilde{d} )$ is complete.
If we define
\beqnx
C(x,t) :=  \bigcup_{v \in \mathbb{R}^d } \left\{ \gamma(x,v,0) : 
\begin{matrix} (\gamma,p ) \in C^\infty  \\ \dot{ \gamma }(x,v,s) = \partial_p^- H( \gamma(x,v,s), p(x,v,s)), \\  \dot{ p }(x,v,s) = - \partial_x H( \gamma(x,v,s), p(x,v,s) ), \\ \gamma(x,v,t) = x, p(x,v,t) = v\end{matrix}  
 \right\} 
\eqnx
and
\beqnx
B(x,t) := \bigcup_{0 \leq r \leq t} \bigcup_{v \in \mathbb{R}^d \backslash \{0\}  } \left\{ \gamma(x,v,0) : 
\begin{matrix} \dot{ \gamma }(x,v,s) = \partial_p H( \gamma(x,v,s), p(x,v,s)), \\  \dot{ p }(x,v,s) = - \partial_x H( \gamma(x,v,s), p(x,v,s) ), \\ \gamma(x,v,r) = x, p(x,v,r) = v\end{matrix}  
 \right\} 
\eqnx
Then for all $(x,t) \in \mathbb{R}^{d+1} \times (0,T)$, we have $B(x,t)$ is well-defined and 
\beqnx
C(x,t) = B(x,t)
\eqnx
\end{Lemma}

\begin{proof}

Before we get to the proof of well-definedness of $B(x,t)$ and the equivalence of $B(x,t)$ and $C(x,t)$, let us first notice that for any curve $(\gamma,p ) \in C^\infty$ satisfying the following inclusion 
\beqn
 \dot{ \gamma }(x,v,s) \in  c(\gamma (x,v,s)  )  \partial_p^- \Phi(p (x,v,s) ) \,, \quad 
\dot{ p }(x,v,s) = - \partial_x c(\gamma (x,v,s)  )   \Phi(p(x,v,s) )
\label{pairpair}
\eqn
We have that
\beqnx
\dot{ \Phi(p (x,v,s)  ) } = - \langle l(x,v,s) , \partial_x c(\gamma (x,v,s)  ) \rangle   \Phi(p(x,v,s) )
\eqnx
where $ l(x,v,s) \in  \partial_p^- \Phi(p (x,v,s) )$, i.e. we get that
$$ \Phi(p (x,v,s)  ) = \Phi(v) \exp\left( \langle l(x,v,s) , \partial_x c(\gamma (x,v,s)  ) \rangle   \right) > 0 \text { for all }  s \quad  \Leftrightarrow   \quad  \Phi(v)  > 0 $$
Fromt the fact that $W$ is absorbing, we conclude that
$$ p (x,v,s)  \neq  0 \text { for all }  s   \quad  \Leftrightarrow  \quad v \neq 0 \,. $$ 
With the above observation, now we would get to the proof of our lemma:
\begin{enumerate}
\item
$B(x,t)$ is well-defined:\\
Since for all $(\gamma,p ) \in C^\infty$ satisfying \eqref{pairpair} with $\gamma(x,v,r) = x, p(x,v,r) = v$, if $v \neq 0$, by the above statement, we have that $ p (x,v,s)  \neq  0$ for all $ 0 \leq s \leq r$, and therefore $\Phi(p (x,v,s) )$ is smooth.  Therefore $ \partial_p \Phi(p (x,v,s) ) $ is well-defined for all $s$.  By uniqueness and existence of ODE system, we have well-definedness of $B(x,t)$ for all $(x,t) \in \mathbb{R}^{d+1} \times (0,T)$.

\item
$ B(x,t) \subset C(x,t) $:\\
For $y \in B(x,t)$, there exists a pair of curves $( \gamma,p ) \in C^\infty$ and $v \neq 0 $, $0 \leq r \leq t$ such that
\beqnx
\begin{matrix} \dot{ \gamma }(x,v,s) = \partial_p H( \gamma(x,v,s), p(x,v,s)), \\  \dot{ p }(x,v,s) = - \partial_x H( \gamma(x,v,s), p(x,v,s) ), \\ \gamma(x,v,r) = x, p(x,v,r) = v\end{matrix} 
\eqnx
and $\gamma(0)  = y$.  Then $\dot{ \gamma }(x,v,s) = \partial_p H( \gamma(x,v,s), p(x,v,s)) \in \partial W \in W$  Now let us define $\tilde{\gamma}$ as a rescaling by $\tilde{\gamma} := \gamma (x,v,s r/t)$ where $r/t \leq 1$, then we notice that $\tilde{\gamma} \in C^\infty$, $\dot{\tilde{\gamma}} \in r/t \partial W \in W$ and $\tilde{\gamma}(t) = x$, $\tilde{\gamma}(0) = y$.  Now choose $\tilde{p}(s) = 0$ for all $0 \leq s \leq t$.  Thus we have $(\tilde{\gamma},\tilde{p}) \in C^\infty$ satisfying the differential inclusion defining $C(x,t)$ and $\tilde{\gamma}(t) = x, \tilde{p}(t) = 0$ initial value. Thus $y =  \tilde{\gamma}(x,v,t)  \in C(x,t)$.

\item
$C(x,t)  \subset   B(x,t)$:\\
For all $y \in C(x,t) $, there exists $ (\gamma,p ) \in C^\infty $ such that
\beqnx
 \dot{ \gamma }(x,v,s) = \partial_p^- H( \gamma(x,v,s), p(x,v,s))  \,, \quad \,   \dot{ p }(x,v,s) = - \partial_x H( \gamma(x,v,s), p(x,v,s) )
\eqnx
and
\beqnx
\gamma(x,v,0) = y, \gamma(x,v,t) = x, p(x,v,t) = v \,.
\eqnx
Since the metric space $(\mathbb{R} , \tilde{d} )$ is complete (and locally compact), by the Hopf-Rinow-Cohn-Vossen theorem \cite{cohnvossen,hopfrinow}, there exists a (globally) length minimizing geodesic $ \bar{\gamma} \in C^\infty$ such that $ \bar{\gamma}(0) = y $, $\bar{\gamma} (r) = x$, and $\tilde{d}(x,y) = r$ for some $r$.
Notice that 
\beqnx
 r = \tilde{d}(x,y) \leq t \,. 
\eqnx
by definition since $\gamma (x,v,\cdot) $ satisfies the conditions $ \gamma \in C^\infty, \gamma(0) : = x, \gamma(t) = y,   \dot{ \gamma} (s)  \in c( \gamma (s) ) W$. 
Consider the vector $ l(s) := \dot { \bar \gamma } (s) / c(\bar \gamma (s)) $.
Now rewrite
$$ \tilde{d}(x,y) = \inf_{\gamma \in  \gamma \in C^\infty, t >0, \gamma(0) : = x, \gamma(t) = y }  \int_0^t \left( 1 +  \mathcal{I}_{ W } \left( \frac{ \dot{ \gamma} (s)}{c( \gamma (s) ) }  \right) \right) ds 
 \,. $$
From the fact that for any given $\gamma \in C^\infty$, $$\frac{ \dot{ \gamma }(s) + \dot{h}(s) }{ c(\gamma(s)  + h(s))} = \frac{ \dot{ \gamma }(s) }{ c(\gamma (s)  )} +  \frac{  \dot{h}(s)  }{ c(\gamma (s)  )} -  \frac{ \dot{ \gamma }(s)  \langle  \nabla_x c(\gamma(s)) ,  h(s) \rangle  }{ [ c(\gamma(s) )]^2 }+ O(  | h(s) |^2 + | \dot{h}(s) |^2)  $$ for all $h \in C_c^\infty$, we obtain directly from variational calculus  that the following optimality condition that $ \bar{\gamma}$ necessarily satisfies, for all perturbation $h \in C_c^\infty$:
$$ \int_0^r \left \langle h(s) , \big( \dot{\bar p}(s) +   \nabla_x c( \bar  \gamma(s)) \left \langle \bar p(s) ,  l(s)  \right \rangle   \big) \right \rangle  ds = 0
\quad  \text{ for some } \quad  \bar  p(s) \in \partial^- [ \mathcal{I}_{ W }] \left(  l(s) \right) \,. $$
By the definition of Legendre transformation, the above condition is equivalent to the existence of $\bar p $ such that
\beqn
 \dot{ \bar  p}(s) 
 = -   \nabla_x c( \bar  \gamma(s)) \Phi ( \bar p(s)  )   \,, \quad  l(s) \in   \partial_p^- \Phi(   \bar  p(s) )  
\label{complstrict}
\eqn
Now we would like to show that we may furthermore choose the dual vector $\bar{p}(s) \neq 0$ for some $s$, (i.e. the strict complementary condition is satisfied and $\bar{p} \neq 0$ as a function.)
In order to do so, we would like to argue that if we enlarge our constraint set from $W$ to $W^\epsilon := (1+\epsilon) W$ for any $\epsilon >0$, we still have a global optimizer but the optimizer differs from the original case when $\epsilon =0$. In fact, for all $\epsilon >0$, consider the problem
\beqnx
\tilde{d}^\epsilon (x,y) := \inf_{ t > 0 }\left\{ t :  \gamma \in C^\infty, \gamma(0) : = x, \gamma(t) = y,   \dot{ \gamma} (s)  \in c( \gamma (s) ) W^\epsilon  \right\}  \,. 
\eqnx
then by reparametrization of each curve $\gamma \in C^\infty$ by $ a:  s \mapsto s/ (1+\epsilon)$, we get that
\beqnx
\tilde{d}^\epsilon (x,y) = (1+\epsilon) ^{-1} \tilde{d} (x,y) 
\eqnx
and the new metric space $(\mathbb{R}^d , \tilde{d}^\epsilon )$ is also complete, and by Hopf-Rinow-Cohn-Vossen theorem \cite{cohnvossen,hopfrinow}, we have a unique minimizing geodesic $\bar{\gamma}^\epsilon$.   Since the new optimzation problem to obtain $\tilde{d}^\epsilon (x,y) $ comes from a rescaling of the original problem $\tilde{d}(x,y)$, we have $\bar{\gamma}^\epsilon = \bar{\gamma} \circ a $.  Therefore the optimizer $\bar{\gamma}^\epsilon \neq \bar{\gamma}$ unless $\epsilon = 0$.  Hence there is a dual vector  $\bar p$ satisfying \eqref{complstrict} that satisfies a strict complementary condition, i.e. $\bar{p}(s) \neq 0$ for some $s$.
Now since $ ( \bar \gamma, \bar p ) $ satisfies \eqref{pairpair} and $\bar p(s) \neq 0$ for some $s$, one concludes that $\bar p(s) \neq 0$ for all $s$, and therefore  $  \partial_p^- \Phi(   \bar  p(s) )  =   \partial_p \Phi(   \bar  p(s) ) $.  Hence there exists $ ( \bar \gamma, \bar p )  \in C^\infty$ that satisfies the bi-characteristic equation and $\bar \gamma(0) = y$, $\bar \gamma(r) = x$ and $0 \leq r \leq t$.  Hence $y \in B(x,t)$

\end{enumerate}

\end{proof}


\begin{Corollary}
Assume Conjecture \ref{conj1}, ie. \eqref{lax_formula}, and Remark 1 are true. 
Assume the metric space $(\mathbb{R} , \tilde{d} )$ is complete.  Then
\beqn
 \varphi(x,t) 
&=&\inf_{v \in \mathbb{R}^d \backslash \{0\}  } \min_{0 \leq r \leq t} \left\{ g( \gamma(x,v, r))  : 
\begin{matrix} \dot{ \gamma }(x,v,s) = - c( \gamma(x,v,s) ) \partial_p \Phi ( p(x,v,s) ) , \\  \dot{ p }(x,v,s) = \partial_x  c(\gamma(x,v,s)) \Phi(p(x,v,s)) , \\\gamma(x,v,0) = x, p(x,v,0) = v 
\end{matrix} 
\right\} 
\label{hugyen_1}
\eqn
\end{Corollary}
\begin{proof}
Apply Lemma \ref{metric} and a change of variable from $v$ to $-v$ and from $s$ to $t-s$.
\end{proof}
\noindent We wish to remind that if $ \Phi ( p ) = |p|_2$, we have $ \partial_p \Phi ( p(v,s) ) = p(v,s) / |p(v,s)|_2 $. 
However, notice this example satisfies the assumption in Lemma \ref{lolXD}, but does not satisfies the assumptions in Lemma \ref{hahalolXD}. 
In fact the formula \eqref{hugyen_1} is the Hugyens principle in disguise. 

\subsection{Generalized Hopf maximization with Hamiltonian $H$ that are possibly neither convex nor concave, but the initial data $g$ is convex}  \label{noncon/con_ham}

When the Hamiltonian $H$ is are neither convex nor concave, but the initial data $g$ is convex, the aforementioned conjectured minimization principle does not seem to hold any longer.  In light of the fact that in some special cases a Hopf formula holds, we conjecture that a generalized Hopf-type maximization principle shall hold for a wide class of problem.

We now describe how we get to a Hopf formula under restricted assumption by Lemma \ref{hahalolXD} and Lemma \ref{hahalolXDsayonara}:
\beqnx
 \varphi(x,t) := 
U(x,T-t) & := & \inf_{u \in \mathcal{U}(T-t)} \left \{ \int_{T-t}^T h (s, x(s), u(s)) \, ds + g(x(T))  \right\} \,.
\eqnx
We again derive our formula formally as follows.
Again, following \cite{boyd}, we shall first consider the following discretization (approximation) for a given $\delta$ such that $\delta N = t$, by denoting $s_n = \delta n $ and $x_N = x$
$$ \varphi(x,t) \approx \varphi^2_N(x,t) $$
where
\beqnx
& & \varphi^2_N(x,t) \\
&:= & \min_{ \{ u_n\}_{n=1}^{N} \in U} \left \{ g(x_0)  + \delta \sum_{n=1}^{N} h (T- s_n, x_n, u_n )  : x_{n+1} - x_n = - \delta f (T- s_{n+1}, x_{n+1}, u_{n+1} ) \text{ for }
n=0,\ldots,N-1 \,, x_N = x  \right\} \,, \\
&=& \inf_{ \{ x_n\}_{n=0}^{N-1} } \sup_{ \{ p_n\}_{n=1}^{N} } F_2(\textbf{X},\textbf{U},\textbf{P}) \,.
\eqnx
where $F_2$ is defined as in Lemma \ref{hahalolXD} and the second equality comes from reformulating the problem with Lagrange multiplier.

Now again if we apply Lemma \ref{hahalolXD} and  Lemma \ref{hahalolXDsayonara}, we get the following:


\begin{Theorem} 
If (A1), (A2), (A3), (A4) and (A5) are satisfied, then we have
then we have
\begin{eqnarray}
\varphi^2_N(x,t) 
&=& \sup_{ \{ p_n\}_{n=1}^{N} }    \inf_{ \{ x_n\}_{n=1}^{N-1} } \tilde{\tilde{F_2}}(\tilde{\textbf{X}},\textbf{P})   \,. 
\end{eqnarray}
If (A1), (A2), (A4),(A5), (A6), (A7) are satisfied, then we have
\beqnx
\varphi^2_N(x,t) 
&=& \sup_{ v \in \mathbb{R}^d }    \bigg \{  \langle p_N , x \rangle  -g^*(p_1) - \delta \sum_{n=1}^{N} H(x_{n}, p_n, s_{n})  + \delta \sum_{n=1}^{N-1}   \langle \frac{p_{n} - p_{n+1}}{\delta} , x_{n}  \rangle   :  \notag \\
& & \qquad \qquad
 \begin{matrix} x_{n}(x,v) - x_{n-1}(x,v) =  \delta \partial_p H(x_{n}(x,v), p_n(x,v), t_{n}), \\ p_{n}(x,v) - p_{n+1}(x,v) =  \delta \partial_x H(x_{n}(x,v), p_{n}(x,v), t_{n}), \\ x_N= x,  p_N = v \end{matrix}  \bigg\}  
\eqnx
\end{Theorem}
\noindent Again, nonetheless, we notice that the resulting formula that we conjectured seems to be correct beyond these assumptions, as the numerical results show (again especially when we take the maximum over all the paths satisfying the KKT conditions). We hope to get some rigorous criteria for these formula to hold in the future.

In order to state our conjecture, in view of \cite{rublev}, let us define the following before we proceed.

\begin{Definition}
Given a set $S \in \mathbb{R}^d$. A function $G : \text{co}(S) \rightarrow \mathbb{R}$ is pseudoconvex on $S$, where $\text{co}(S)$ is the convex hul of $S$, if for all $\{ s_i\}_{i=1}^{d+1} \in S $ and $ \{ \alpha_i \}_{i=1}^{d+1}$ such that $\alpha_i \geq 0$ and $\sum_{i=1}^{d+1} \alpha_i = 1$, we have
\beqnx
G \left( \sum_{i=1}^{d+1} \alpha_i s_i \right) \leq \sum_{i=1}^{d+1} \alpha_i G(s_i) \,.
\eqnx
\end{Definition}

Now, passing $N$ to the limit, we arrive at the following conjectured Hopf formula:

\begin{Conjecture}
\label{conj2}
Assume $H(x,p,t) \in C^2$ and $g(p) \in C^2$ is convex w.r.t. $p$ that satisfies (A5).  Consider the set
\beqnx
& & S_0(x,t) \\
&:=& \text{argmax}_{v \in \mathbb{R}^d } \bigg\{  \langle x,v \rangle  - g^*( p(x,v,0)) - \int_{0}^t  \bigg\{ H(\gamma(x,v,s), p(x,v,s), s) -  \langle \partial_x H( \gamma(x,v,s), p(x,v,s),s),  \gamma(x,v,s) \rangle  \bigg\} ds  :  \notag\\
& & \qquad \qquad
 \begin{matrix} \dot{ \gamma }(x,v,s) =  \partial_p H( \gamma(x,v,s), p(x,v,s),s), \\  \dot{ p }(x,v,s) = - \partial_x H( \gamma(x,v,s), p(x,v,s), s), \\ \gamma(x,v,t) = x, \,  p(x,v,t) = v\end{matrix} 
 \bigg\} \,.
\eqnx
Assume further that for all $(x,t)$, we have that the function  
\beqnx
G: co \left(  S_0(x,t) \right) & \rightarrow & \mathbb{R} \\
v &\mapsto& H(x, v , s) -  \langle \partial_x H( x, v, t ),  x \rangle 
\eqnx
is pseudoconvex on $S_0(x,t)$. Then there exists $t_0$ such that the viscosity solution to \eqref{HJeqn}-\eqref{HJinitial} can be represented as
\eqref{hopf_formula} 
for $t < t_0$. 
Moreover, if $\phi(x,t)$ is differentiable w.r.t. $x$ at a neighbourhood of $(x,t)$ and the infrimium is attained by $\tilde{v}$, then we have
$ \partial_x  \varphi(x,t) =\tilde{v} $,
\end{Conjecture}
\noindent We may refer to it as the \textbf{maximization principle}, or the generalized Hopf formula.
One point to remark is that the conjectured Hopf formula is a generalization of the well-known Hopf formula in \cite{Evans,Hopf_forumula,rublev}.
A proof that our formula holds under less restricted assumption that shown above is also done in \cite{proofimportant}, after a previous version of our paper \cite{preprint} was arXiv-ed.  A proof that our formula holds under less restricted assumption that shown above is also done in \cite{proofimportant}, after a previous version of our paper \cite{preprint} was arXiv-ed.  We suspect the weakest assumption of our formula to hold is the above assumption of psuedoconvexity, similar to one stated in \cite{rublev}.

\noindent
\textbf{Remark 2}: When $H(x,p,t)$ is not differentiable at some given point $p$, then we believe again that in formula \eqref{hopf_formula}, the Mordukhovich subdifferential, $ \partial_x^{-} H$, as defined in \cite{subdiff,subdiff2,subdiff3}, should be used instead of $ \partial_x H$.  In that case the constraint becomes the inclusion $\dot{ \gamma }(x,v,s) \in \partial_p^- H( \gamma(x,v,s), p(x,v,s),s)$, and infrimum is taken over also all the curves $(\gamma, p )\in C^\infty $ satisfying the inclusion before supremum over $v$ is taken.

\noindent 
\textbf{Remark 3}: we expect a candidate of less restricted assumption than \cite{proofimportant} as the above convexity assumption in the variable $p$ in view of the theorem in \cite{rublev}.  A rigorous approach toward this this formula might be following an approach of \cite{rublev} to show the postulated formula is a minimax viscosity solution following the notations in \cite{rublev} and references therein.

Below we present several examples where this conjecture is valid.  We notice that \textbf{none} of the followings satisfies the assumptions of Lemma \ref{hahalolXD} and Lemma \ref{hahalolXDsayonara} given in the proof of our formula, since by Lemma \ref{lolXD2}, the Hamiltonian $H(x,p,t)$ given as in \eqref{hahahahaha} under the our assumptions are automatically convex. 
\vskip 5mm

\noindent \textbf{Example 1}   When $H(x,p,t) = H(p,t)$, we have $\dot{ p }(v,s) =  \partial_x H( \gamma(v,s), p(v,s), s) = 0$.  Therefore assuming Conjecture \ref{conj2}, i.e. \eqref{hopf_formula}, and Remark 2, the conjestured formula gives
\beqnx
 \varphi(x,t) 
&=&- \inf_{v \in \mathbb{R}^d } \bigg\{ g^*( p(x,v,0)) + \int_{0}^t  \bigg\{ H( p(x,v,s), s) -  \langle 0,  \gamma(x,v,s) \rangle  \bigg\} ds  - \langle x,v \rangle  :  
 \begin{matrix} (\gamma,p) \in C^\infty \\\dot{ \gamma }(x,v,s) \in  \partial_p^- H(  p(x,v,s),s), \\  \dot{ p }(x,v,s) = 0, \\ \gamma(x,v,t) = x, \\ p(x,v,t) = v\end{matrix} 
 \bigg\} \\
&=&- \inf_{v \in \mathbb{R}^d } \bigg\{ g^*(v ) + \int_{0}^t  H( v , s) ds  - \langle x,v \rangle  \bigg\}
\eqnx
which gets us back to the Hopf formula \cite{rublev}.  
Note that in this example  assumption in Lemma \ref{hahalolXD} or in Lemma \ref{hahalolXDsayonara} may not be satisfied since $U$ may not be compact.

\noindent \textbf{Example 2} 
When $H(x,p,t)$ is a \textbf{non-convex} homogeneous degree-$1$ functional w.r.t. $p$ in the following form:
\beqnx
H(x,p,t) = c_1 (x,t) \Phi_1 (p) - c_2 (x,t) \Phi_2 (p)
\eqnx
where $\Phi_1$ and $\Phi_2$ are with their Wulff sets $W_1,W_2$ \cite{Wulff} as strictly convex set with smooth boundary $\partial W_i \in C^\infty, i = 1,2$, then by definition of Mordukhovich subdifferential, we have
\beqnx
\partial_p^- H(x,p,t) = 
\begin{cases}
c_1 (x,t) \partial_p \Phi_1 (p) - c_2 (x,t)  \partial_p \Phi_2 (p) & \text{ if } p \neq 0 \\
\emptyset & \text{ if } p = 0 \,. \\
\end{cases} 
\eqnx
Therefore assuming Conjecture \ref{conj2}, i.e. Equation \eqref{hopf_formula} and Remark 2, we have:
\beqn
 \varphi(x,t)  
&=&-  \inf_{v \in \mathbb{R}^d \backslash \{0\} }   \bigg\{ g^*( p(x,v,0)) + \int_{0}^t  \bigg\{ \bigg( c_1(\gamma(x,v,s), s) -  \langle  \partial_x c_1 (\gamma(x,v,s), s),  \gamma(x,v,s) \rangle \bigg) \Phi_1(  p(v,s) )  \notag \\
& & \qquad \qquad \qquad -  \bigg( c_2(\gamma(x,v,s), s) -  \langle  \partial_x c_2 (\gamma(x,v,s), s) ,  \gamma(x,v,s) \rangle \bigg) \Phi_2(  p(x,v,s) ) \bigg\} 
ds  - \langle x,v \rangle  : \notag \\
& & \qquad \qquad
 \begin{matrix} \dot{ \gamma }(x,v,s) =   c_1 (\gamma(x,v,s),s)  \partial_p  \Phi_1 (p(x,v,s)) -  c_2 (\gamma(x,v,s),s) \partial_p  \Phi_2 (p(x,v,s)) \\
\dot{ p }(x,v,s) = - \partial_x c_1 (\gamma(x,v,s), s) \Phi_1 (p(x,v,s)) -  \partial_x c_2 (\gamma(x,v,s), s) \Phi_2 (p(x,v,s)) , \\ \gamma(x,v,t) = x, \,  p(x,v,t) = v 
\end{matrix} 
 \bigg\} \,. \label{formula_special}
\eqn
When $H(x,p,t) = - c (x) |p|_2$, we have, from \eqref{formula_special} and integration by parts, that
\beqnx
 \varphi(x,t) = \sup_{v \in \mathbb{R}^d  \backslash \{0\} } \bigg\{  g( \gamma(x,v,t)) :   \qquad \qquad
 \begin{matrix} \dot{ \gamma }(x,v,s) =   c_2 (\gamma(x,v,s)) p(x,v,s) /  |p(v,s)|_2  \\ 
\dot{ p }(x,v,s) =  -  \partial_x c_2 (\gamma(x,v,s) )  |p(x,v,s)|_2 , \\ \gamma(x,v,0) = x, \, p(x,v,0) = v\end{matrix} 
 \bigg\} 
\label{hugyen_2}
\eqnx
This is again an Hugyens principles in disguise.
In these examples, assumption in Lemma \ref{hahalolXD} or in Lemma \ref{hahalolXDsayonara} are not satisfied.

When $H(x,p,t)$ is, on the other hand, a \textbf{convex} homogeneous degree-one functional w.r.t. $p$, i.e. when $c_2(x,t) = 0$, and hence $H(x,p,t) = c(x,t) \Phi(p)$ for some $c >0$ and $\Phi$, then assuming Conjecture \ref{conj2}, i.e. Equation \eqref{hopf_formula} and Remark 2, we obtain
\beqnx
& & \varphi(x,t)  \notag \\
&=&- \inf_{v \in \mathbb{R}^d\backslash \{0\} }  \bigg\{ g^*( p(x,v,0)) + \int_{0}^t  \bigg\{ \bigg( c(\gamma(x,v,s), s) -  \langle  \partial_x c (\gamma(x,v,s), s) ,  \gamma(x,v,s) \rangle \bigg) \Phi(  p(v,s) )  \bigg\} 
ds  - \langle x,v \rangle  :  \notag \\
& & \qquad \qquad
 \begin{matrix} 
(\gamma, p ) \in C^\infty\\
\dot{ \gamma }(x,v,s)  \in   c (\gamma(x,v,s),s) \partial_p^- \Phi( p(x,v,s) ) \\
\dot{ p }(x,v,s) = - \partial_x c (\gamma(x,v,s), s) \Phi (p(x,v,s))  , \\ \gamma(x,v,t) = x,  \, p(x,v,t) = v 
\end{matrix} 
 \bigg\} \,.
\eqnx
Note that in this example assumptions in Lemma \ref{hahalolXD} are satisfied but that of Lemma \ref{hahalolXDsayonara} are not satsified.

\noindent \textbf{Example 3} 
When the HJ PDE comes from an ODE system in a differential game, we have the following finite horizon problem with initial state $x \in \mathbb{R}^d$. We consider the (Lipschitz) solution $x: [t,T] \rightarrow \mathbb{R}^d$ of the following linear dynamic system with initial condition $x$ at time $t$:
\beqnx
\begin{cases}
\frac{d x}{ d s}(s) = M x(s) + N_C(s) a(s) + N_D(s) b(s) \quad \text{ in } \quad (t,T)  \\
x(t) = x
\end{cases}
\label{ODE}
\eqnx
where $M$ is a given $d \times d$ matrix independent of time, and $\{N_C(s)\}_{t<s<T}$, $\{N_D(s)\}_{t<s<T}$ are two families of $d \times d$ matrices with real entries. Using the notation in section \ref{sec3}, if we let
\beqnx
f(t,x,a,b) &=& M x(s) + N_C(s) a + N_D(s) b \\
h(t,x,a,b) &=& -  \mathcal{I}_{C(t)} (p) +  \mathcal{I}_{D(t)} (p) \,,
\eqnx
for some family of convex sets $\{C(s)\}_{t<s<T} \subset C$, $\{D(s)\}_{t<s<T} \subset D$ ,
our Hamiltonian read:
\beqnx
\tilde{H}^\pm(x,p,t)  = \max_{b \in D(t)} \min_{a \in C(t)} \{ - \langle Mx +  N_C(t)a +  N_D(t) b , p \rangle \}  =  - \langle Mx, p\rangle + \Phi_{C(t)} (-N_C^*(t)p) - \Phi_{D(t)}  (-N_D^*(t)p)  \,.
\eqnx
where 
\beqnx
\Phi_{W}(p) : =  \max_{x \in W}  \langle x, p \rangle \,.
\eqnx
After a change of variable to get from a finite time PDE problem $(-\infty, T)$ to an initial time PDE problem $(0,\infty)$, we have the upper/lower values $ \varphi(x,t) := U(x,T-t) = V(x,T-t) $ satisfy:
\beqnx
\begin{cases}
\frac{\partial }{\partial t} \varphi + H(x, \nabla_x \varphi ,t ) = 0 &\text{ on } \mathbb{R}^d \times (0,\infty) \,,\\
\varphi(x,0) = g(x) &\text{ on } \mathbb{R}^d  
\end{cases}
\eqnx
where the Hamiltonian $H$ is now
\beqnx
H(x,p,t) = \tilde{H}^\pm(x,p,T-t) = - \langle Mx, p\rangle + \Phi_{C(T-t)} (-N_C^*(T-t)p) - \Phi_{D(T-t)}  (-N_D^*(T-t)p) \,.
\eqnx

Now, notice that $H$ is smooth w.r.t. $x$, assuming Conjecture \ref{conj2}, i.e. Equation \eqref{hopf_formula} and Remark 2, together with the fact that $\partial_x H(x,p,t) = - M^* p$ and a change of variable, we get to same formula as in (2.5) in \cite{Hopf_Lax_3}. In fact, The generalized Hopf formula gives
{\footnotesize
\beqnx
& & \varphi(x,t) \\
&=&- \min_{p \in \mathbb{R}^d } \bigg\{ g^*( e^{ - M^* T } p ) + \int_{0}^t  \bigg\{ 
\Phi_{C(T-s)} \left( [ - e^{ - M (T-s)}N_C(T-s)]^*     p  \right) - \Phi_{D(T-s)}  \left(  [- e^{ - M (T-s)}  N_D(T-s)]^*  p  \right) \bigg\}  ds  - \langle e^{- M(T- t)} x,    p \rangle 
 \bigg\} \\
\eqnx
}If we write $z = e^{ - M (T-t) } x$ and write 
$J( z ) = g( e^{ M T }  z ) $,
then this change of coordinate in the state variable by $e^{ M T } $ gives rise to a sympletic change of coordinate in the phase variables by $\text{diag}( e^{ M T },  e^{ -M^* T  } ) $, hence
\beqnx
J^*(p) = g^*(e^{ - M^* T } p) \,,
\eqnx
and therefore we get to 
\beqnx
& & \varphi(x(z),t) \\
&=& - \min_{p \in \mathbb{R}^d } \bigg\{ J^*( p ) + \int_{0}^t  \bigg\{ 
\Phi_{C(T-s)} \left(  [ - e^{ - M (T-s)}N_C(T-s)]^*     p  \right) - \Phi_{D(T-s)}  \left(  [- e^{ - M (T-s)}  N_D(T-s)]^*  p  \right) \bigg\}  ds  - \langle z ,    p \rangle 
 \bigg\} 
\eqnx
which is the same formula as in (2.5) in \cite{Hopf_Lax_3}.
Note that in this example, both the assmptions in Lemma \ref{hahalolXD} and Lemma \ref{hahalolXDsayonara} are satsified.

\section{Numerical methods}   \label{sec5}

\subsection{Optimization Methods: Coordinate Descent}

In order for computation of optimization (in either the Lax formulation or the Hopf formulation)  to be efficient, we have recast the initial value HJ PDE problem to minimization problem in $d$ dimensions, where the curves $\gamma$ and $p$ inside the function evaluation (given a vector $v$) are defined explicitly as the solutions to the bi-characteristic equation.  We suggest to solve the ODE's numerically, given a pair of $(x,v)$, up to time $t$ using any ODE solver.  This way the minimization/maximization problem reduces to optimization of a finite-dimensional problem (as a function of $v$.)  Similar to \cite{Hopf_Lax_3}, we suggest to apply coordinate descent to the following functionals with argument $v$ for a given pair $(x,t)$: 
\beqnx
\mathcal{F}^1_{x,t} (v) :=  g( \gamma(x,v,t,0)) + \int_{0}^t \left\{ \langle p(x,v,t,s), \partial_p H( \gamma(x,v,t,s), p(x,v,t,s),s ) \rangle - H(\gamma(x,v,t,s), p(x,v,t,s), s ) \right\} ds 
\eqnx
or (noticed we omitted $r$ in this functional, but $r$ is actually the final time of the ODE $(\gamma,p)$ stated below)
\beqnx
\mathcal{F}^2_{x,t} (v) :=  \min_{0 \leq r \leq t} \left\{ g( \gamma(x,v,r,0))  \right\}
\eqnx
or
\beqnx
\mathcal{G}_{x,t} (v) :=  g^*( p(x,v,t,0)) + \int_{0}^t  \bigg\{ H(\gamma(x,v,t,s), p(x,v,t,s), s) -  \langle \partial_x H( \gamma(x,v,t,s), p(x,v,t,s),s),  \gamma(x,v,t,s) \rangle  \bigg\} ds  - \langle x,v \rangle  
\eqnx
where in either case, the pair $\gamma(x,v,t,s), p(x,v,t,s) $ solves the following final value problem for the given pair of $x,t$ (with the dependence of the curves w.r.t. $x,t$ clearly indicated in the notion):
\beqnx
\begin{cases}
 \dot{ \gamma }(x,v,t,s) = \partial_p H( \gamma(x,v,t,s), p(x,v,t,s),s), \\  \dot{ p }(x,v,t,s) = - \partial_x H( \gamma(x,v,t,s), p(x,v,t,s), s), \\ \gamma(x,v,t,t) = x, \\ p(x,v,t,t) = v
\end{cases}
\eqnx

Now, to minimize $\mathcal{F}^{1}_{x,t} (\cdot)$, $\mathcal{F}^{2}_{x,t} (\cdot)$ or $\mathcal{G}_{x,t} (\cdot)$, we utilize a cyclic coordinate descent algorithm.  We illustrate our algorithm with the functional $\mathcal{G}_{x,t} (\cdot)$ :
\begin{algorithm}
Take an initial guess of the Lipschitz constant $L$, and set ${count} := 0$.
Initialize $j_{1} := 1$ and a parameter $\alpha := 1/L$.   For $k=1,....,M$, do:
\begin{itemize}
\item[\textbf{1}:]
\beqnx
\begin{cases}
v^{k+1}_i  =  v^k_i - \alpha \, \partial_i  \mathcal{G}_{x,t}  (v^{k})  & \text{ if } i = j_k,\\
v^{k+1}_i  =  v^k_i  & \text{ otherwise. }
\end{cases}
\eqnx
\item[\textbf{2}:]
\beqnx
j_{k+1} :=  j_k + 1.
\eqnx
If $j_{k+1}=d+1$, then reset $j_{k+1}=1$.
\item[\textbf{3}:]
If $| v^{k+1}_i  -  v^k_i| > \varepsilon$, then set $\text{count} := 0$. \
If $k =M$, then reset $k:=0$ and set $\alpha := \alpha/2$,  (i.e. let $L := 2L$.)
\item[\textbf{4}:]
If $| v^{k+1} -  v^k| < \varepsilon$, set $\text{count} := \text{count} +1$.
\item[\textbf{5}:]
If $\text{count} =d$, stop.
\end{itemize}
Return $v_\mathrm{final} =  v^{k+1}$.
\end{algorithm}
We minimize $\mathcal{F}^{1}_{x,t} (\cdot)$, $\mathcal{F}^{2}_{x,t} (\cdot)$ in a similar fashion.  In this algorithm, we will need to discuss how to evaluate the functional values and also their numerical derivatives.  This will be discussed in the next subsection.

\subsection{Evaluation of functional and its derivatives: ODE solver, numerical differentiation and integration} \label{int_diff_num}

In this subsection, we discuss several numerical approximation used in our numerical experiments. 
First, in our step, we need to devise by ODE solvers.  We suggest to use the standard forward Euler solver for a given stepsize $\Delta s$.  Of course the performance can be improved by using more advanced solvers, such as the pseudo-spectral method, etc.
We also need to deal with approximating integrals.   Similar to \cite{Hopf_Lax_3}, we suggest to evaluate either the derivatives of $\mathcal{F}^1_{x,t}$ or $\mathcal{G}_{x,t}$ by numerical quadrature rules for integral computations.
As for $\mathcal{F}^2_{x,t}$, we compute the maximum by directly choosing the maximum around the computed grid from the ODE solver in the case when $H(x,p,t) = H(x,p)$, since the ODE will be the same in this case and the ODE solver will get to the function values of $\gamma(v,s)$ easily.
\beqn
\mathcal{F}^1_{x,t} (v) \approx \min_{r_i = i \Delta s, i=0, ..., t/\Delta s} \left\{ g( \gamma(x,v,r_i, 0))  \right\}
\label{num_min}
\eqn
with a same choice of $\Delta s$ as the ODE solver. 
We also suggest, as in \cite{Hopf_Lax_3}, a numerical differentiation rule for derivative computations.
We integrate using a standard rectangular quadrature rule (we use $\mathcal{F}^1_{x,t}$ to illustrate):

\beqn
& & \mathcal{F}^1_{x,t} (v) \notag \\
& \approx& 
g( \gamma(x,v,t,0)) \notag \\
& & + \sum_i  \left\{ \langle p(x,v, t, i \Delta s), \partial_p H( \gamma(x,v, t, i \Delta s), p(x,v, t, i \Delta s), i \Delta s ) \rangle - H(\gamma(x,v,t, i \Delta s), p(x,v, i \Delta s),  t, i \Delta s) \right\}  \Delta s  \notag \\
\label{num_int}
\eqn
again with the same choice of $\Delta s$ as in the ODE solver. 
We suggest approximating the partial derivative $\partial_i \mathcal{F}_{x,t}(p)$ (which means the differentiation of the function w.r.t. the direction $e_i$) by a finite difference:
\beqn
\partial_i \mathcal{F}_{x,t}(v)  \approx \frac{ \mathcal{F}^1_{x,t} (v + \sigma e_i)  -  \mathcal{F}^1_{x,t}(v) }{\sigma}
\label{num_diff}
\eqn
with a given choice of $\sigma$.
By using numerical differentiation, we have the advantage of not necessarily handling tedious analytic computations of the derivative of Hamiltonian which might be singular at times.  Also, we only have two evaluations of the function value per iteration.
By performing numerical approximations, either ODE solvers, differentiation or integration, we are bound to introduces numerical errors.  These errors introduced by numerical approximation can be effectively controlled by choosing appropriately small sizes of $\Delta s$ and $\sigma$.

\subsection{Certificate of Correctness}

The method we compute a sequence is guaranteed to converge to a local minimum under an assumption of lowe semi-continuity and boundedness of the functional.  However such a descent type algorithm cannot guarantee convergence to global optimal.  However, we can checked the correctness of the vector $v$ that is computed $p(0) \in \partial g (\gamma (0))$. 
With this certificate, in case a local optimal does not satisfy the assumption, we can discard the value thus computed and restart the algorithm with another initial guess.

\section{Numerical Results} \label{sec6}

In this section, we provide numerical experiments which compute viscosity solutions to HJ PDE with a time-dependent
Hamiltonian arising from control system.  For a given set of points $(t,z)$ , we use
\textbf{Algorithm 1} to compute \eqref{HJeqn}-\eqref{HJinitial}.  We set $M = 500$ and have a different initial guess of the Lipschitz constant $L$ in each example.  We evaluate $(x,t)$
in a given set of grid points over some $2$ dimensional cross-sections of the form $ [-3 , 3]^2 \times \{0\}^{d-2} $. We choose our error tolerance
in the coordinate descent iteration as $\varepsilon = 0.5 \times 10^{-7}$, which acts as
our stopping criterion.
The step-size in the numerical quadrature rule in \eqref{num_min}, \eqref{num_int} as well as the forward Euler ODE solver is set to be different in each example, and they are all denoted as $\Delta s$. The step-size for numerical differentiation in \eqref{num_diff} denoted by $\sigma$, and the Lipschitz constants $L$ in \textbf{Algorithm 1}  are also chosen differently in each in each example.  
In all our examples, we set random initial starting points uniformly distributed in $ [-2 , 2]^d$.  
We always consider the initial value to be a function with zero level set as an ellipse enclosed by the equation $ \langle x, A x \rangle = 1 $ where $A^{-1} = \text{diag}(1, 25/4,1/4,1/4,...,1/4)$, i.e. our initial condition for the
HJ PDE is $$g(x) = \frac{1}{2} (   \langle x, A x \rangle -1) $$
with the aforementioned $A$.  
For a convex Hamiltonian, we make one initial guess, and for a non-convex Hamiltonian, we perform at most $20$ independent trials of random initial guesses to get rid of possible local minima or in places when the derivative of the viscosity solution does not exist.
We present our run-time in the format of
$ a \times 10^{-c} s \times k$ where $k$ is the number of initial guesses made in the respective example. 
When $d=2$, clear comparison is performed with our solution to the solution computed by a first order Lax-Friedrichs monotone scheme \cite{WENO} with $\Delta t=0.001$ and $\Delta x =0.005$. 
We perform such comparison because we do not have an explicit solution for neither one of our examles, and therefore Lax-Friedrichs scheme (which has theorectical convergence guarantee) is used to compare with the solution we computed using the minimization/maximization principles.
Our algorithm is implemented in C++ on an $1.7$ GHz Intel Core i7-4650U
CPU. In what follows, we present some examples.

\vskip 5mm

\noindent \textbf{Example 1} We solve for the state-dependent Hamiltonian of the linear form
\begin{equation*}
H(x,p,t) =  - 0.2 \, c(x) - \nabla c(x) \cdot p \,,
\end{equation*}
where
\begin{equation*}
c(x) = 1 + 3 \exp( - 4 | x - (1,1,0,0...,0 ) |_2^2) \,.
\end{equation*}
The minimization principle \eqref{lax_formula} is used to compute the solution $\varphi$.
The temporal stepsize is chosen as $\Delta s = 0.02$.  
In this example, since the Hamiltonian is linear, there exists only a \textbf{unique} $v$ such that the constraints of $(\gamma,p)$ in the bi-characteristics are satisfied, and therefore there is nothing to optimize, thus no choices of $\sigma$ and $L$ are necessary.  
Figure \ref{exp_0} (left) gives the solutions when $d = 2$, $T= 0.12$. The runtime using C++ is $ 9.929 \times 10^{-7} s \times 1$ per point.
Figure \ref{exp_0} (right) is the solution computed by the Lax-Friedrichs scheme, which illustrates that the two solutions coincide.
Figure \ref{exp_05} gives the solutions when $d = 1024$, $T= 0.5$. The runtime using C++ is $ 4.342 \times 10^{-2} s \times 1$ per point.
The computation time is reasonably fast for a $1024$ dimensional problem (although in this toy example, there is no optimization to solve).

This example satisfies the assumptions in Lemma \ref{lolXD}, but does not satisfies the assumptions in Lemma \ref{halolXD}, Lemma \ref{hahalolXD} and in Lemma \ref{hahalolXDsayonara}.

\begin{figurehere}
     \begin{center}
     \vskip -0.3truecm
   \scalebox{0.4}{\includegraphics{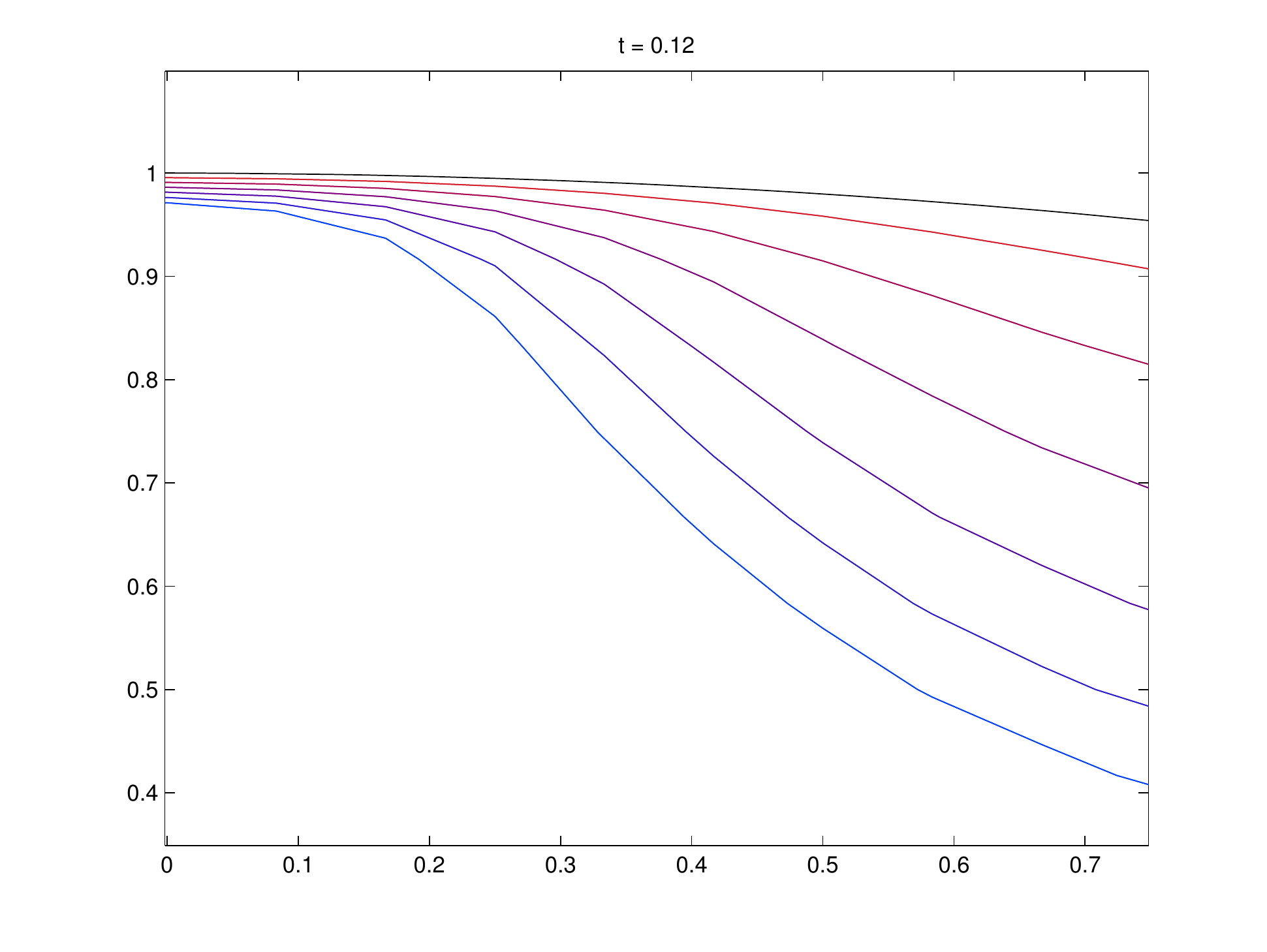}}
           \scalebox{0.4}{\includegraphics{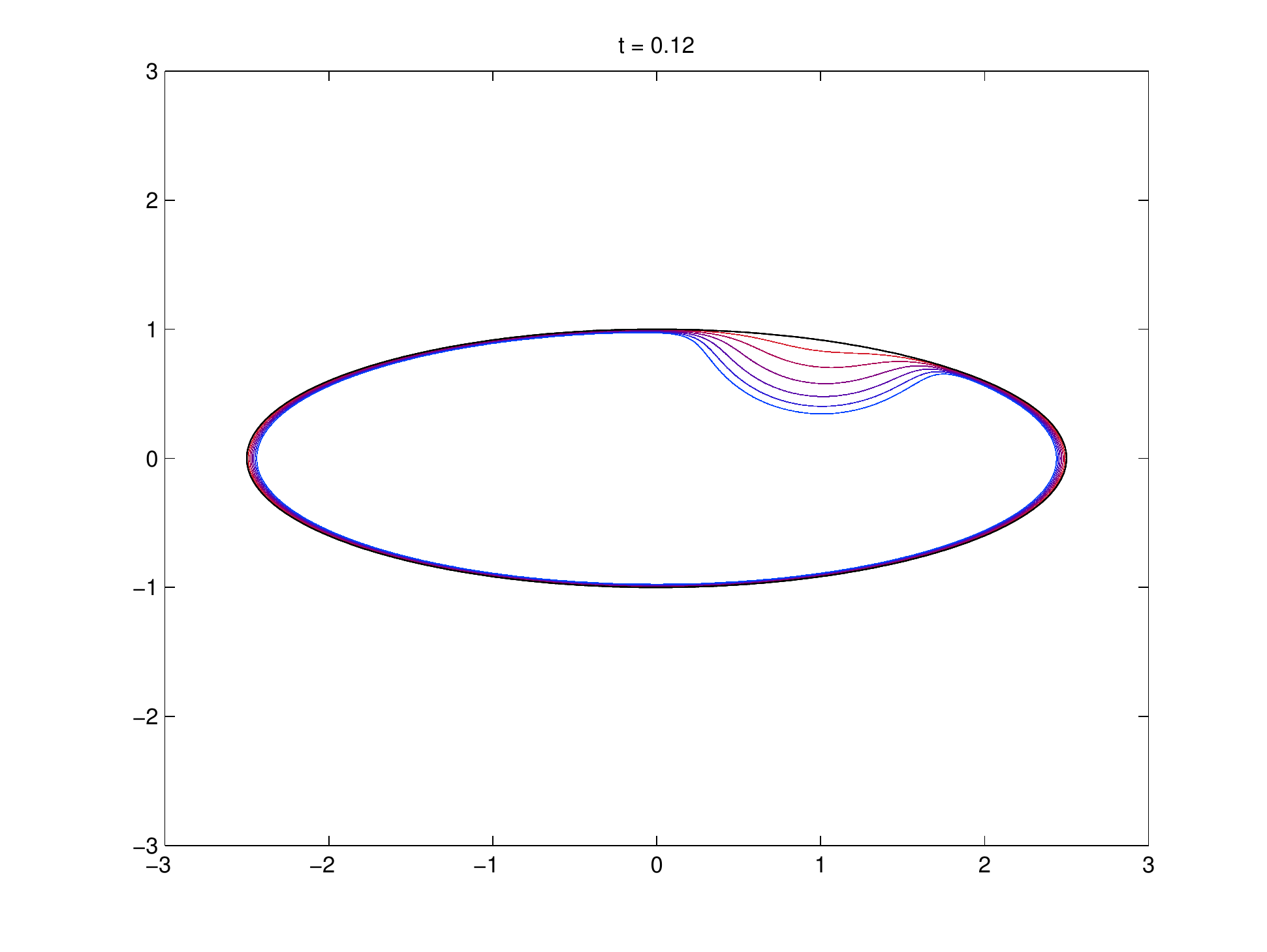}} \\
   \scalebox{0.4}{\includegraphics{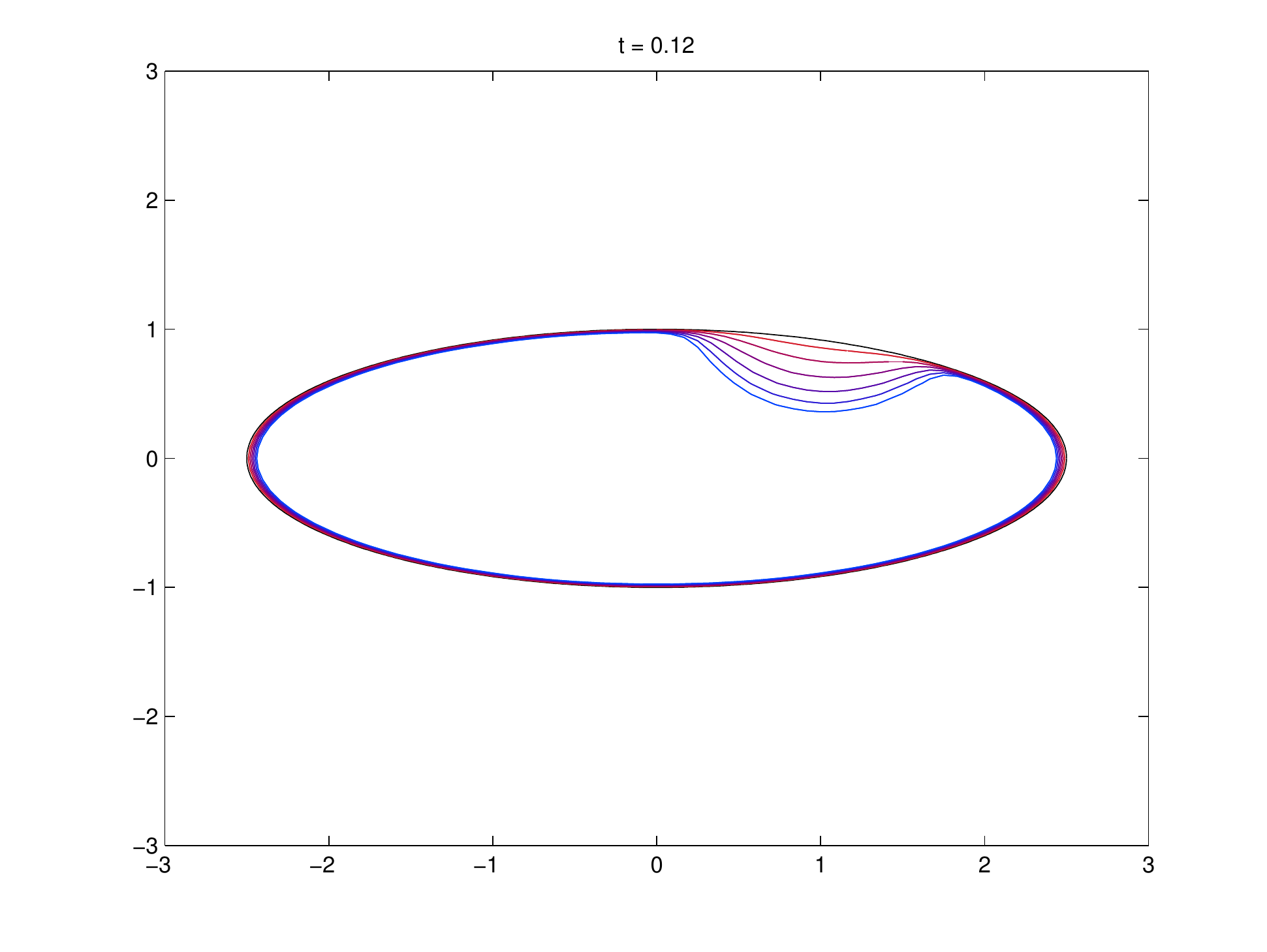}}
           \scalebox{0.4}{\includegraphics{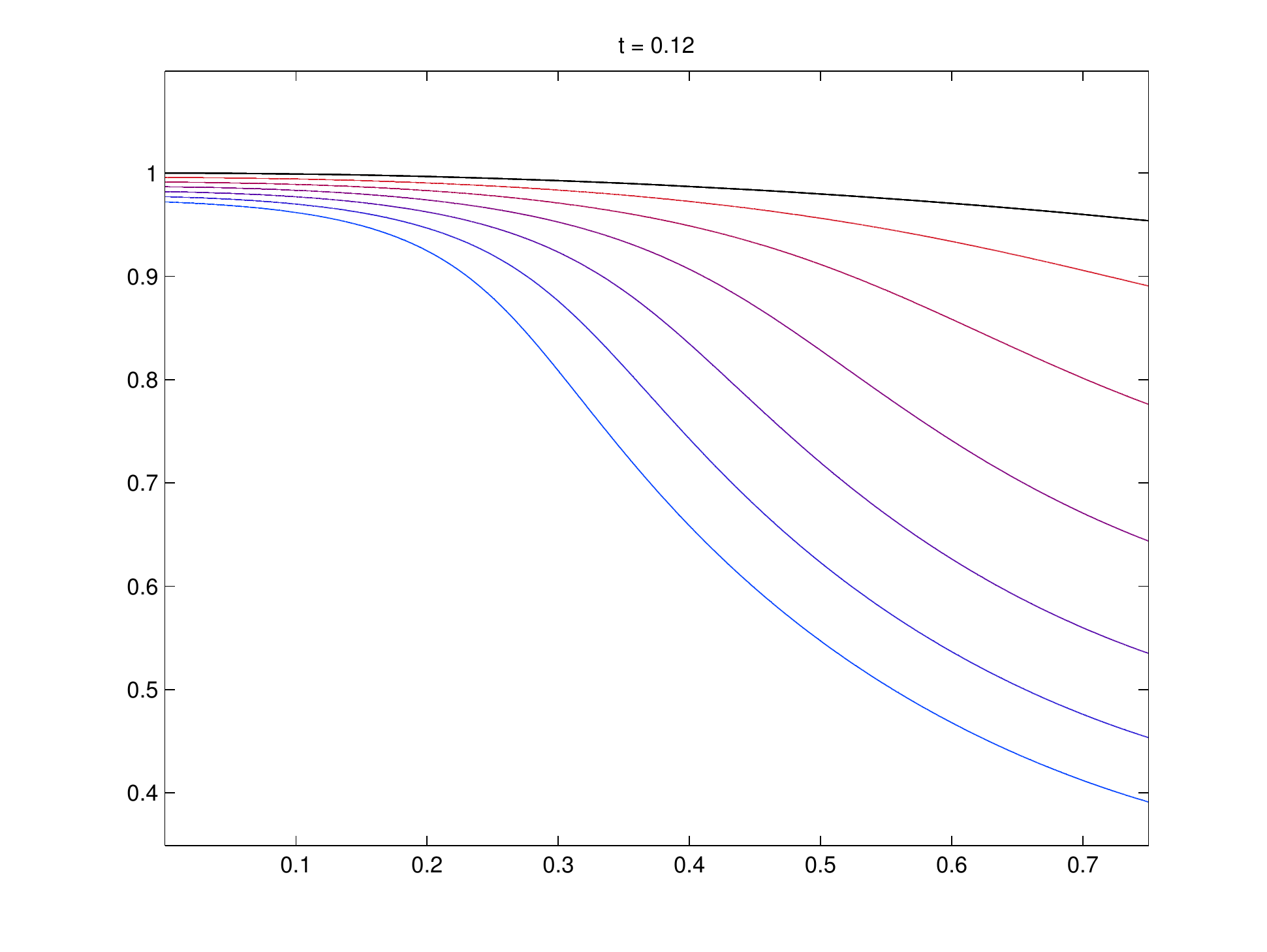}} \\
\caption{\small Zero level sets of the solution $\phi(\cdot,t)$ for $t = 0.02,0.04,...,0.12$ in \textbf{Example 1} with $d =2$; left: minimization/maximization principle, right: Lax-Friedrichs.}\label{exp_0}
     \end{center}
 \end{figurehere}

\begin{figurehere}
     \begin{center}
     \vskip -0.3truecm
   \scalebox{0.4}{\includegraphics{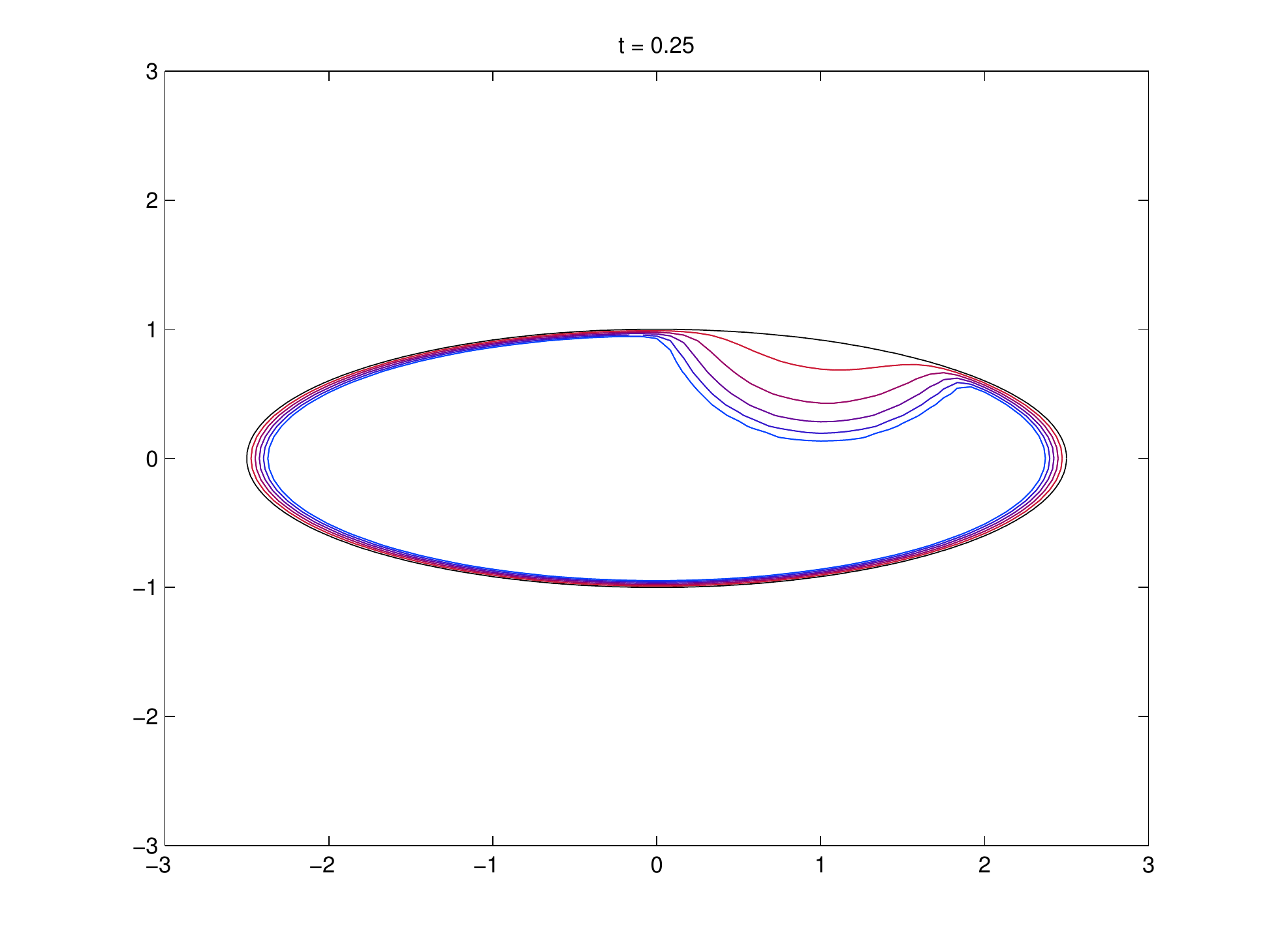}} 
   \scalebox{0.4}{\includegraphics{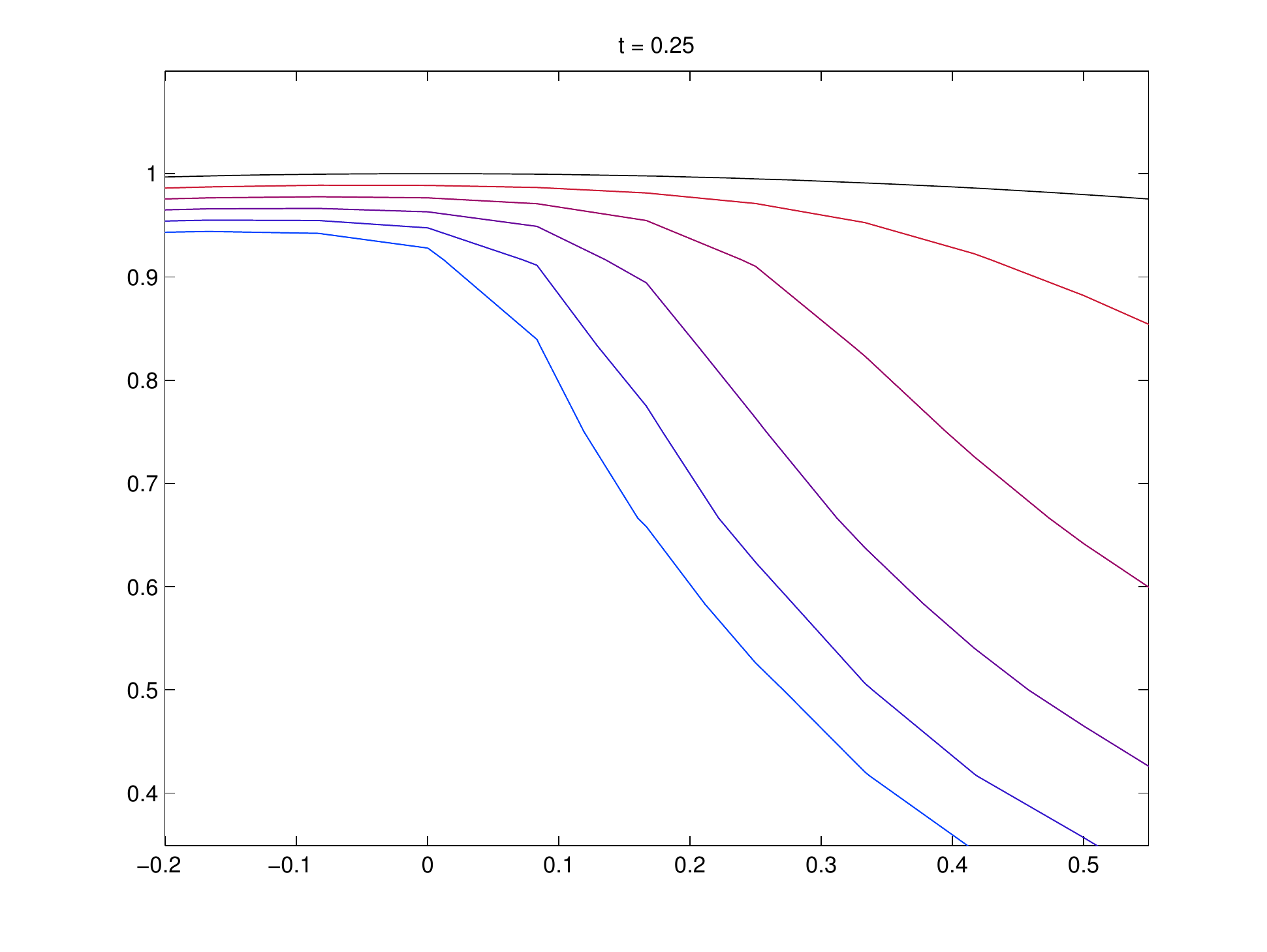}} 
\caption{\small  Zero level sets of the solution $\phi(\cdot,t)$ for $t = 0.05,0.1,...,0.25$ by minimization/maximization principle in \textbf{Example 1} with $d =1024$; left: full-size; right: close-up.}\label{exp_05}
     \end{center}
 \end{figurehere}

\noindent \textbf{Example 2} Next we solve for the state-dependent Hamiltonian, the well-known harmonic oscillator:
\begin{equation*}
H^{\pm}(x,p,t) = \pm \frac{1}{2} ( | p |^2_2 + | x |^2_2) \,,
\end{equation*}
The maximization principle \eqref{hopf_formula} is used to compute the solution $\varphi$ for both of the $\pm$ cases.
The temporal stepsize is chosen as $\Delta s = 0.02$.  
The other constants are chosen as follows: stepsize $\sigma = 0.001$ and $L= 3$.
Figure \ref{exp_07} (left) gives the solutions when $d = 2$ and sign of $H$ is negative, i.e. $H^{-}$, $T= 0.5$. 
The runtime using C++ is $ 2.366 \times 10^{-5} s \times 5$ per point.
Figure \ref{exp_07} (right) is the solution computed by the Lax-Friedrichs scheme for comparison.
Figure \ref{exp_08} (left) gives the solutions when $d = 2$ and sign of $H$ is positive, i.e. $H^{+}$, $T= 0.5$. 
The runtime using C++ is $ 2.605 \times 10^{-5} s \times 5$ per point.
Figure \ref{exp_08} (right) is the solution computed by the Lax-Friedrichs scheme for comparison.
Figure \ref{exp_09} gives the solutions when $d = 7$ and sign of $H$ is negative, i.e. $H^{-}$, $T= 0.2$.  
The runtime using C++ is $ 3.717 \times 10^{-4} s \times 5$ per point.
The computation time is minimal for a $7$ dimensional problem.

In order to test for an example with convex Hamiltonian but a non-convex initial condition, we take another initial function $g(x)$ as the following Rosenbrock function:
\beqn
g(x) = 
0.4  \times 10^{-3} \times  \left(  -100+ (1 - x_1 )^2 + 100 (1 + x_2 - x_1^2)^2 \right) \,.
\label{weird}
\eqn
The Hamiltonian is chosen as $H^{+}$ which is convex, and $T= 0.5$.  We choose $d = 2$ for comparision.
Figure \ref{exp_081} (left) gives the solutions using the minimization principle.
The runtime using C++ is $ 3.847 \times 10^{-5} $ per point.
Figure \ref{exp_081} (right) is the solution computed by the Lax-Friedrichs scheme for comparison.
To remark, both the convex case $H^+$ and nonconvex case $H^-$ in this example does not satisfy the assumptions of any of the lemmas proved in this work since $U$ may not be compact.

\begin{figurehere}
     \begin{center}
     \vskip -0.2truecm
   \scalebox{0.4}{\includegraphics{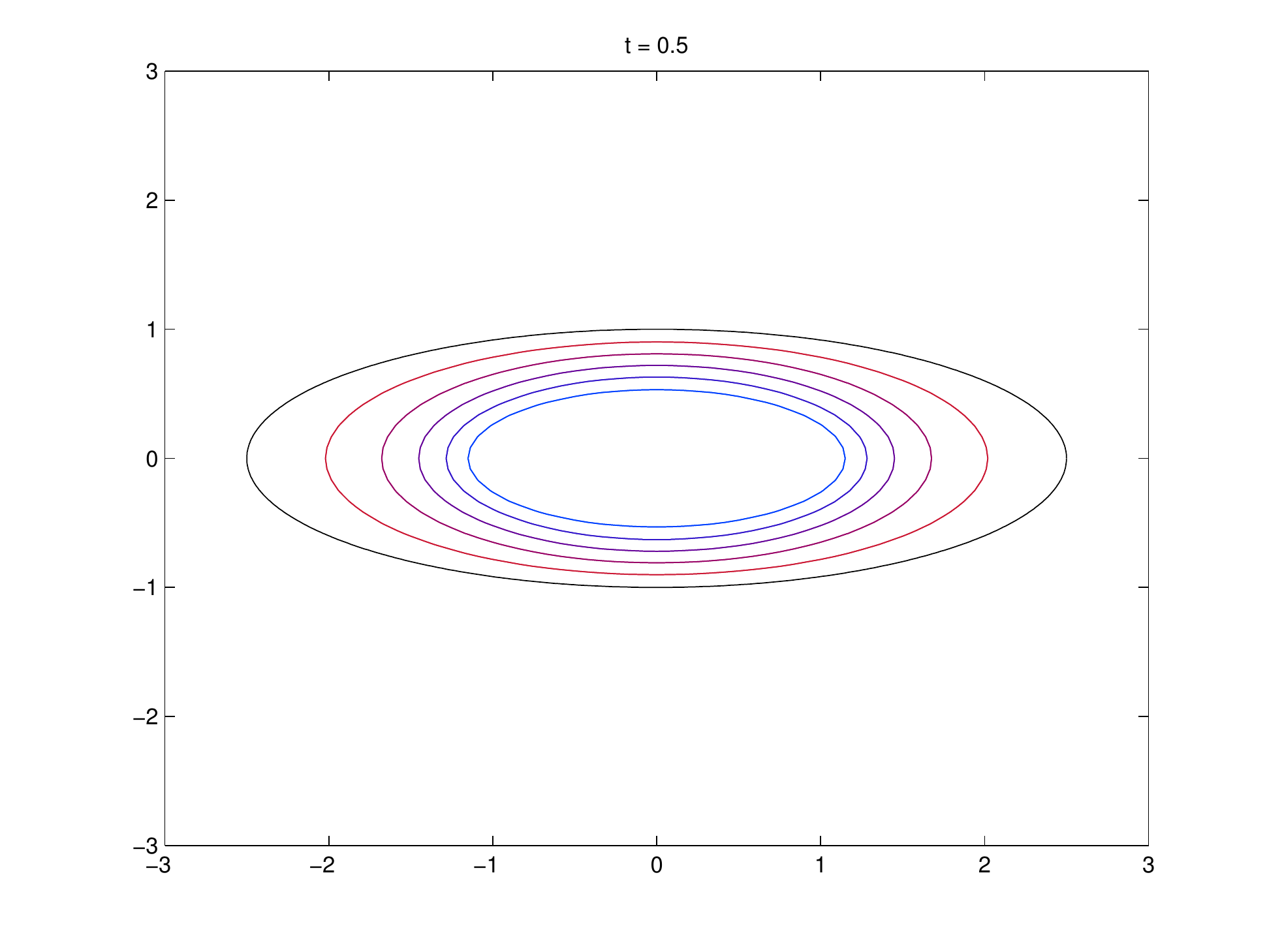}}
           \scalebox{0.4}{\includegraphics{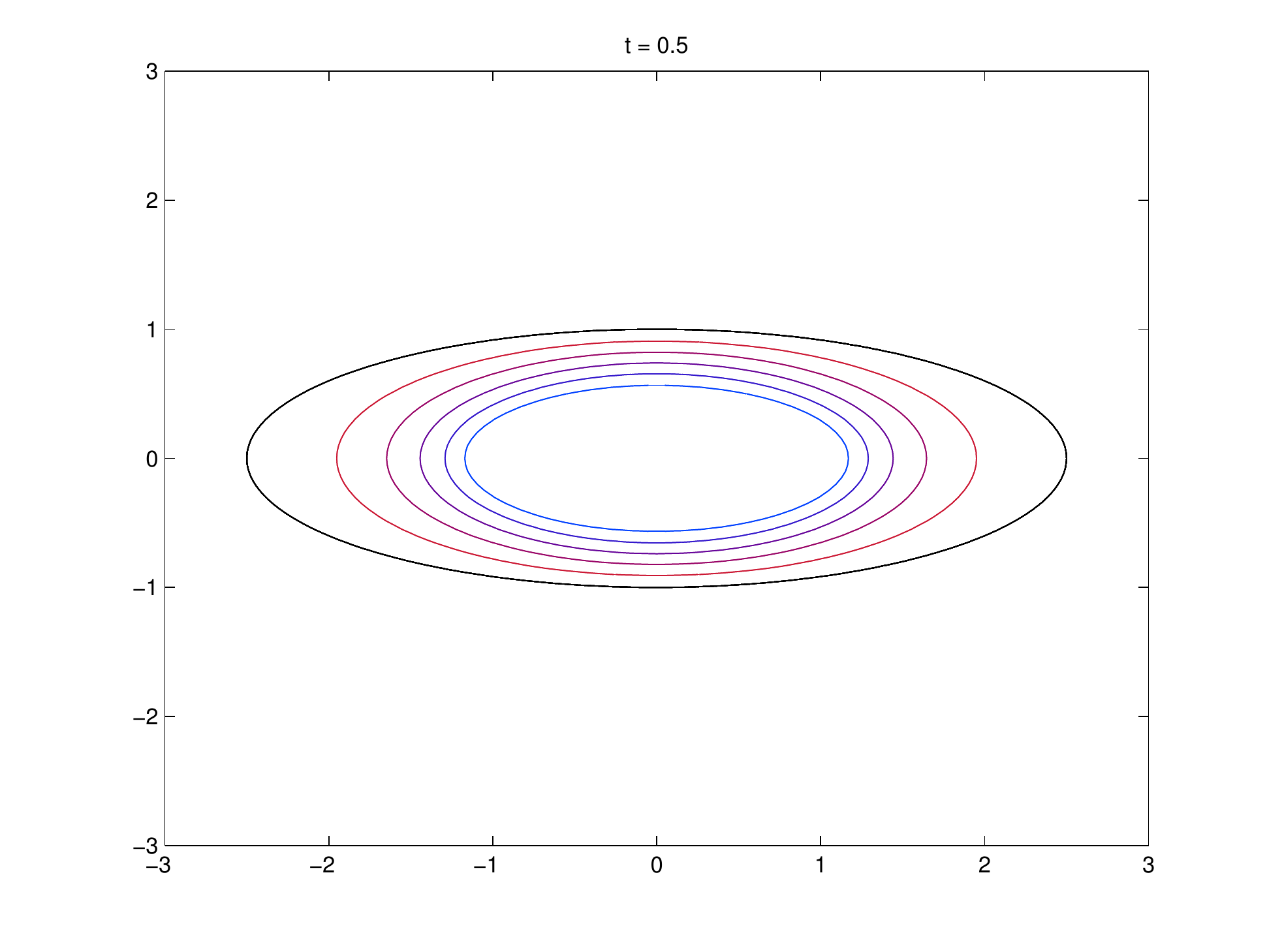}} \\
\caption{\small Zero level sets of the solution $\phi(\cdot,t)$ for $t = 0.1,0.2,...,0.5$ by minimization/maximization principle in \textbf{Example 2} with $H^{-}$ and $d =2$; top: full-size; bottom: close-up.}\label{exp_07}
     \end{center}
 \end{figurehere}

\begin{figurehere}
     \begin{center}
     \vskip -0.3truecm
   \scalebox{0.4}{\includegraphics{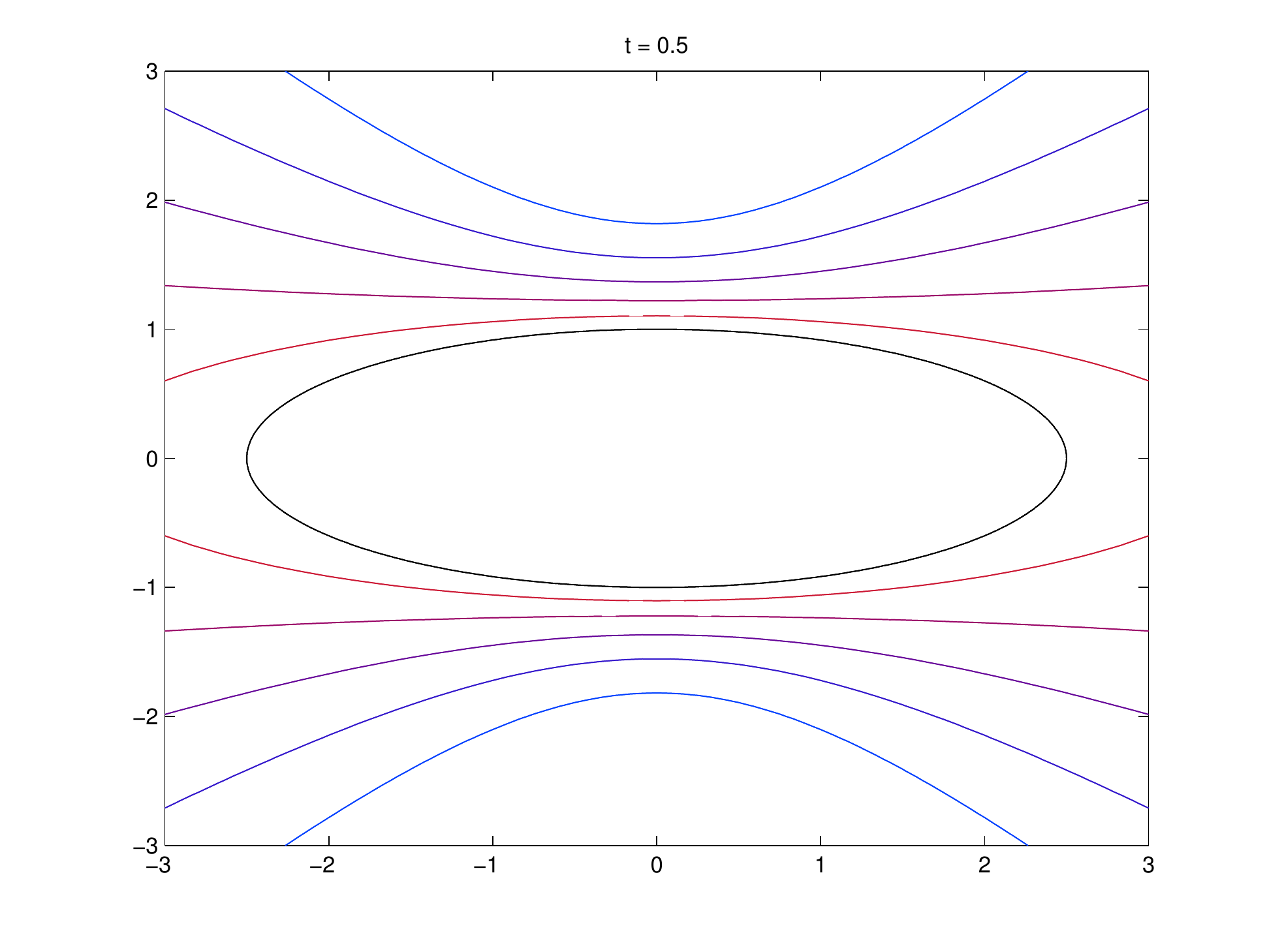}}
           \scalebox{0.4}{\includegraphics{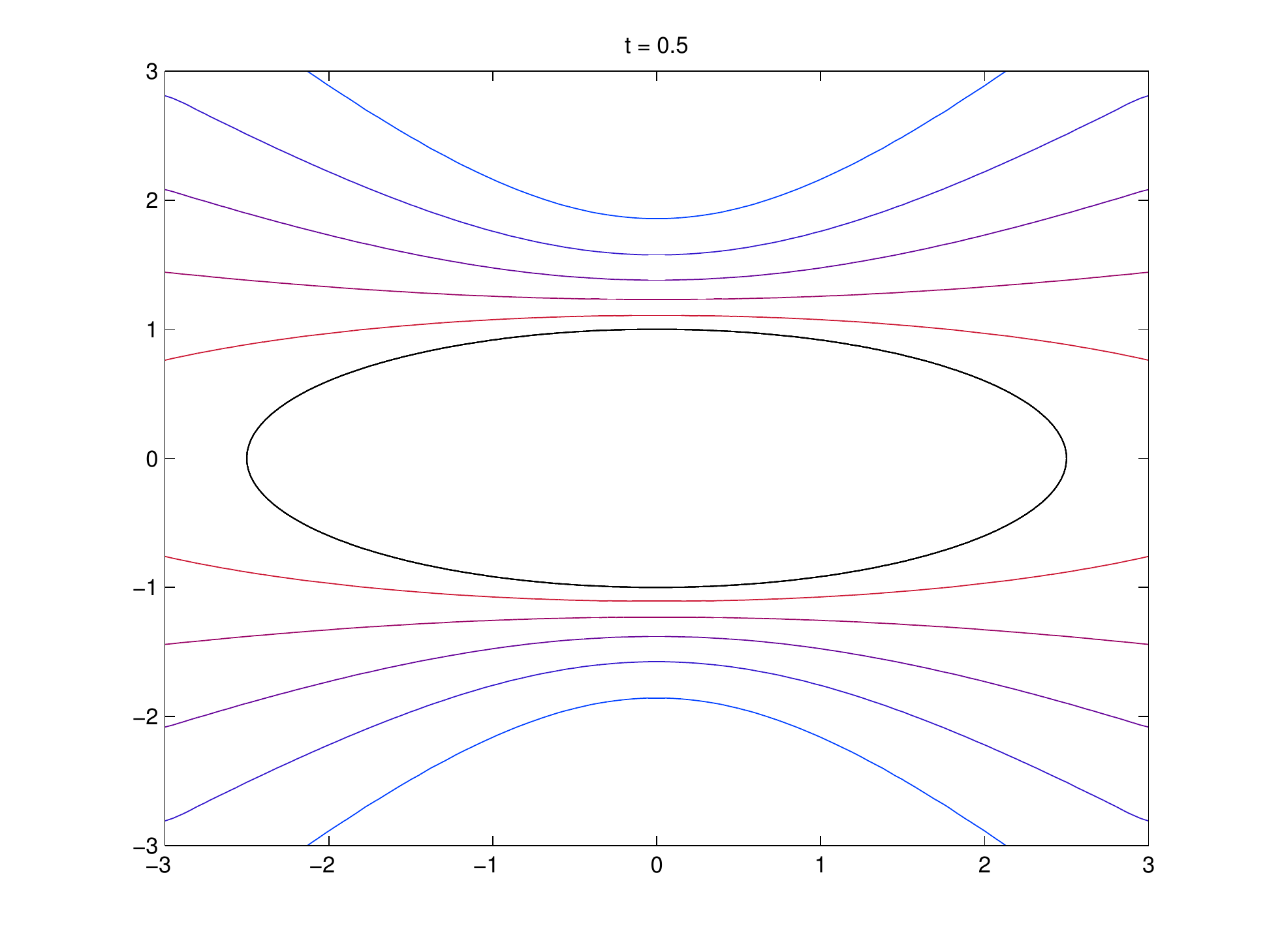}} \\
\caption{\small Zero level sets of the solution $\phi(\cdot,t)$ for $t = 0.1,0.2,...,0.5$ by minimization/maximization principle in \textbf{Example 2} with $H^{+}$ and $d =2$; top: full-size; bottom: close-up.}\label{exp_08}
     \end{center}
 \end{figurehere}

\begin{figurehere}
     \begin{center}
     \vskip -0.3truecm
    \scalebox{0.4}{\includegraphics{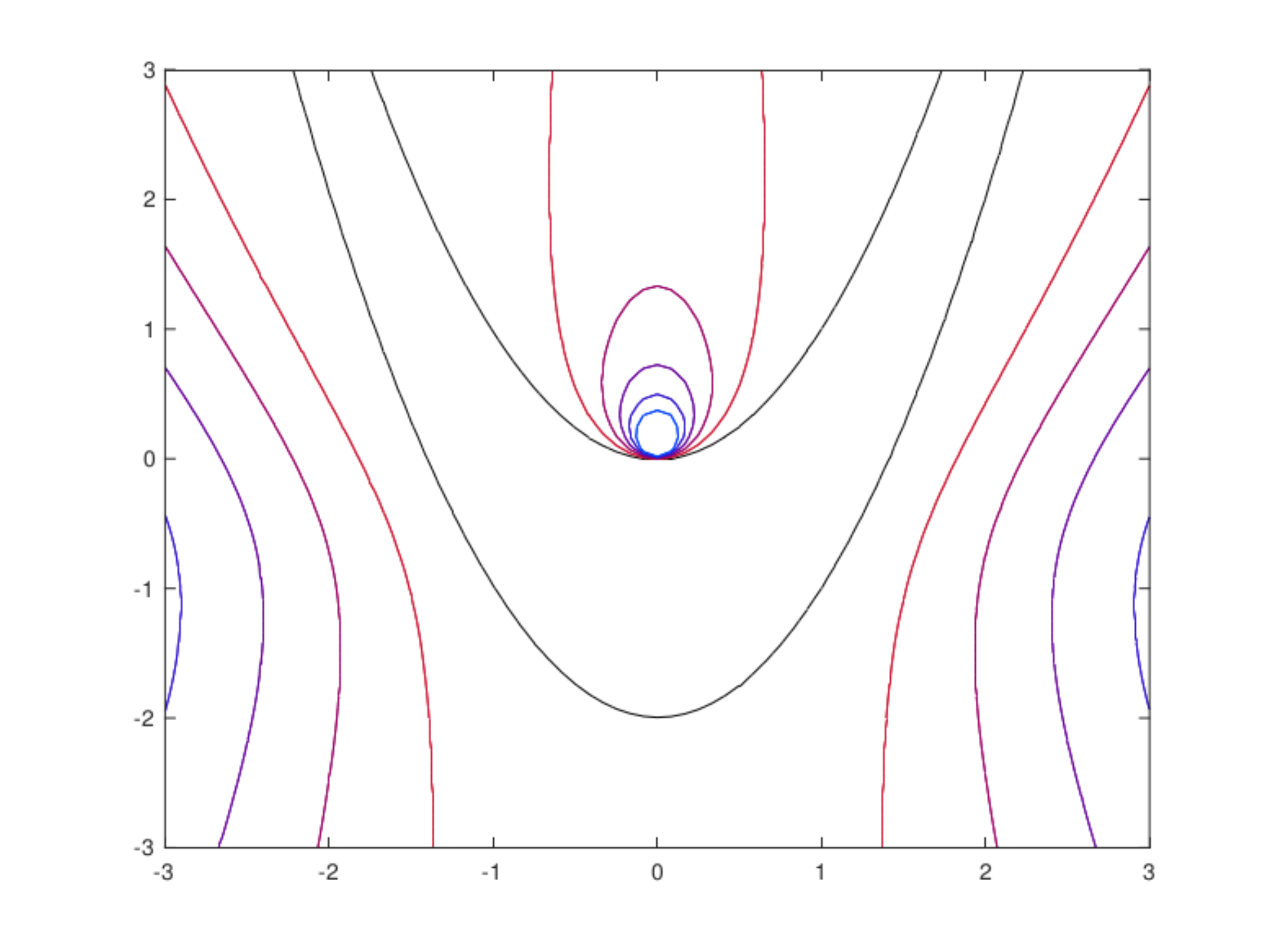}}
           \scalebox{0.4}{\includegraphics{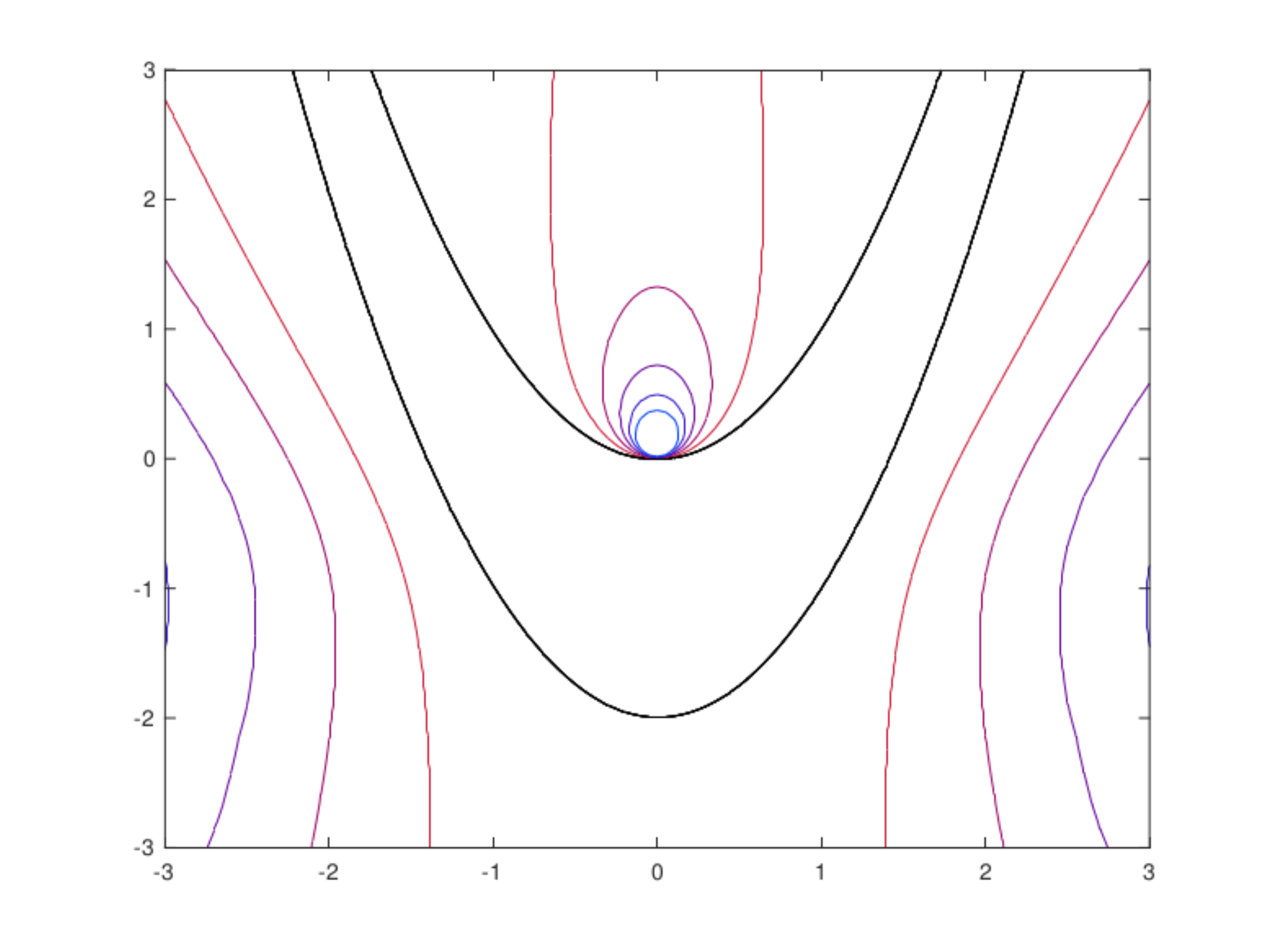}} \\
\caption{\small Zero level sets of the solution $\phi(\cdot,t)$ for $t = 0.1,0.2,...,0.5$ by minimization/maximization principle in \textbf{Example 2} with $H^{+}$, $d =2$ and non-convex initial data \eqref{weird}; top: full-size; bottom: close-up.}\label{exp_081}
     \end{center}
 \end{figurehere}

\begin{figurehere}
     \begin{center}
     \vskip -0.3truecm
   \scalebox{0.4}{\includegraphics{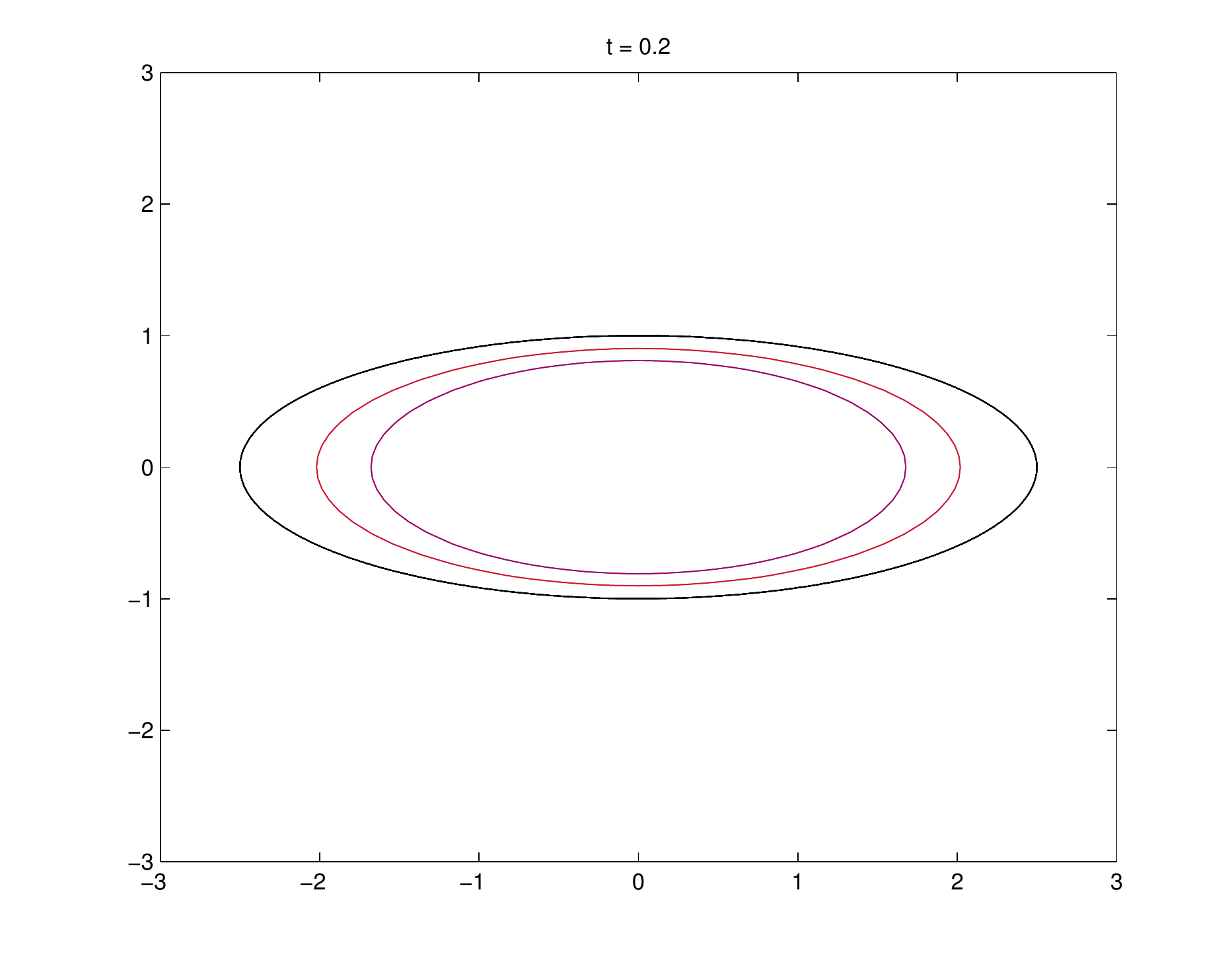}}
\caption{\small Zero level sets of the solution $\phi(\cdot,t)$ for $t = 0.1,0.2$  by minimization/maximization principle in \textbf{Example 2} with $H^{-}$ and $d =7$.}\label{exp_09}
     \end{center}
 \end{figurehere}

\noindent \textbf{Example 3} We solve for the state-dependent Hamiltonian of the form
\begin{equation*}
H^{\pm}(x,p,t) = \pm c(x) | p |_2 \,,
\end{equation*}
where
\begin{equation*}
c(x) = 1 + 3 \exp( - 4 | x - (1,1,0,0...,0 ) |_2^2) \,.
\end{equation*}
The minimization principle \eqref{lax_formula} is used to compute the solution $\varphi$ for both of the $\pm$ cases.
The temporal stepsize is chosen as $\Delta s = 0.02$.  
The other constants are chosen as follows: stepsize $\sigma = 0.001$ and $L= 0.02$.
Figure \ref{exp_1} gives the solutions when $d = 2$ and sign of $H$ is positive, i.e. $H^{+}$, $T= 0.3$. 
The runtime using C++ is $6.865 \times 10^{-4} s \times 5$ per point.
For comparison, Figure \ref{exp_1} (right) is the solution computed by Lax-Friedrichs scheme.
Figure \ref{exp_2} gives the solutions when $d = 2$ and sign of $H$ is negative, i.e. $H^{-}$, $T= 0.5$.
The runtime using C++ is $1.417 \times 10^{-3} s \times 5$ per point.
Figure \ref{exp_2} (right) is the solution computed by Lax-Friedrichs scheme for comparison.
The angles where a discontinuity of derivative is present are computed more accurately using our maximization principle.
Figure \ref{exp_3} gives the solutions when $d = 10$ and sign of $H$ is negative, i.e. $H^{-}$ , $T= 0.4$. 
The runtime using C++ is $ 2.470 \times 10^{-2} s \times 5$ per point.
The computation time is still excellent for a $10$ dimensional problem.

To remark, again the convex case $H^+$ in this example satisfies the assumptions in Lemma \ref{lolXD}, but does not satisfies the assumptions in Lemma \ref{halolXD}, Lemma \ref{hahalolXD} and in Lemma \ref{hahalolXDsayonara}.
Whereas for the nonconvex case $H^-$ in this example does not satisfy the assumptions of any of the lemmas proved in this work.

\begin{figurehere}
     \begin{center}
     \vskip -0.3truecm
   \scalebox{0.4}{\includegraphics{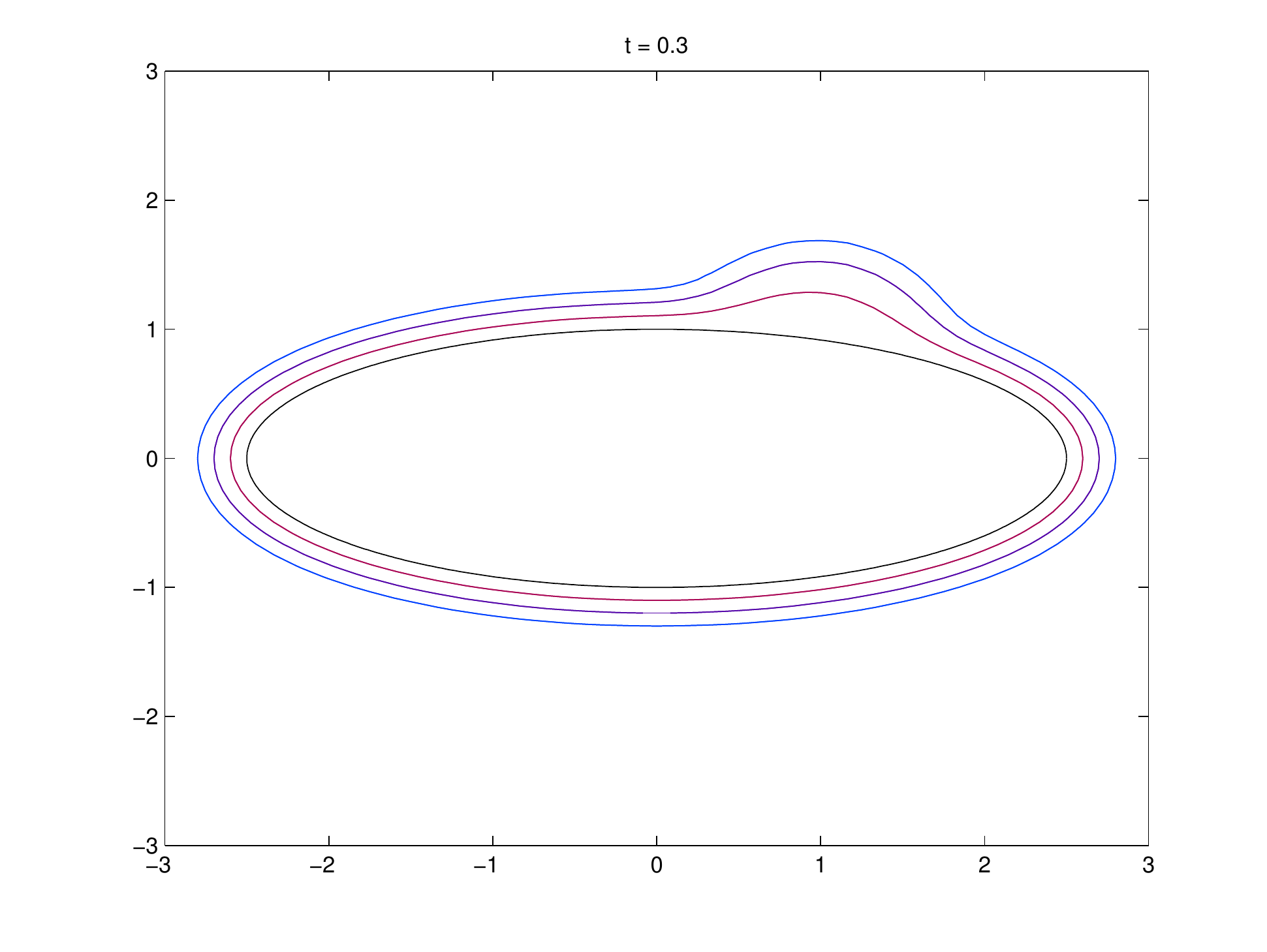}}
           \scalebox{0.4}{\includegraphics{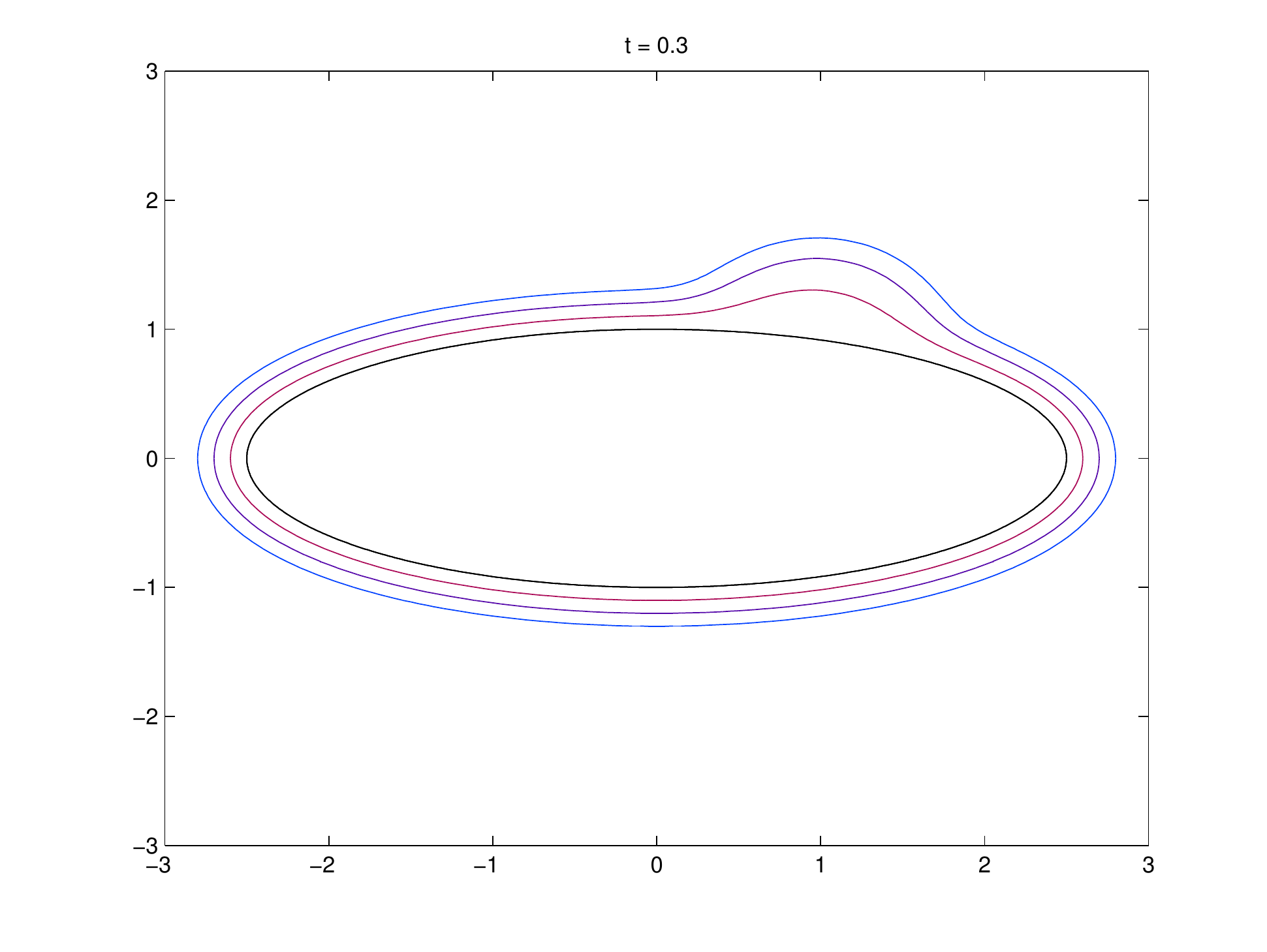}} \\
\scalebox{0.4}{\includegraphics{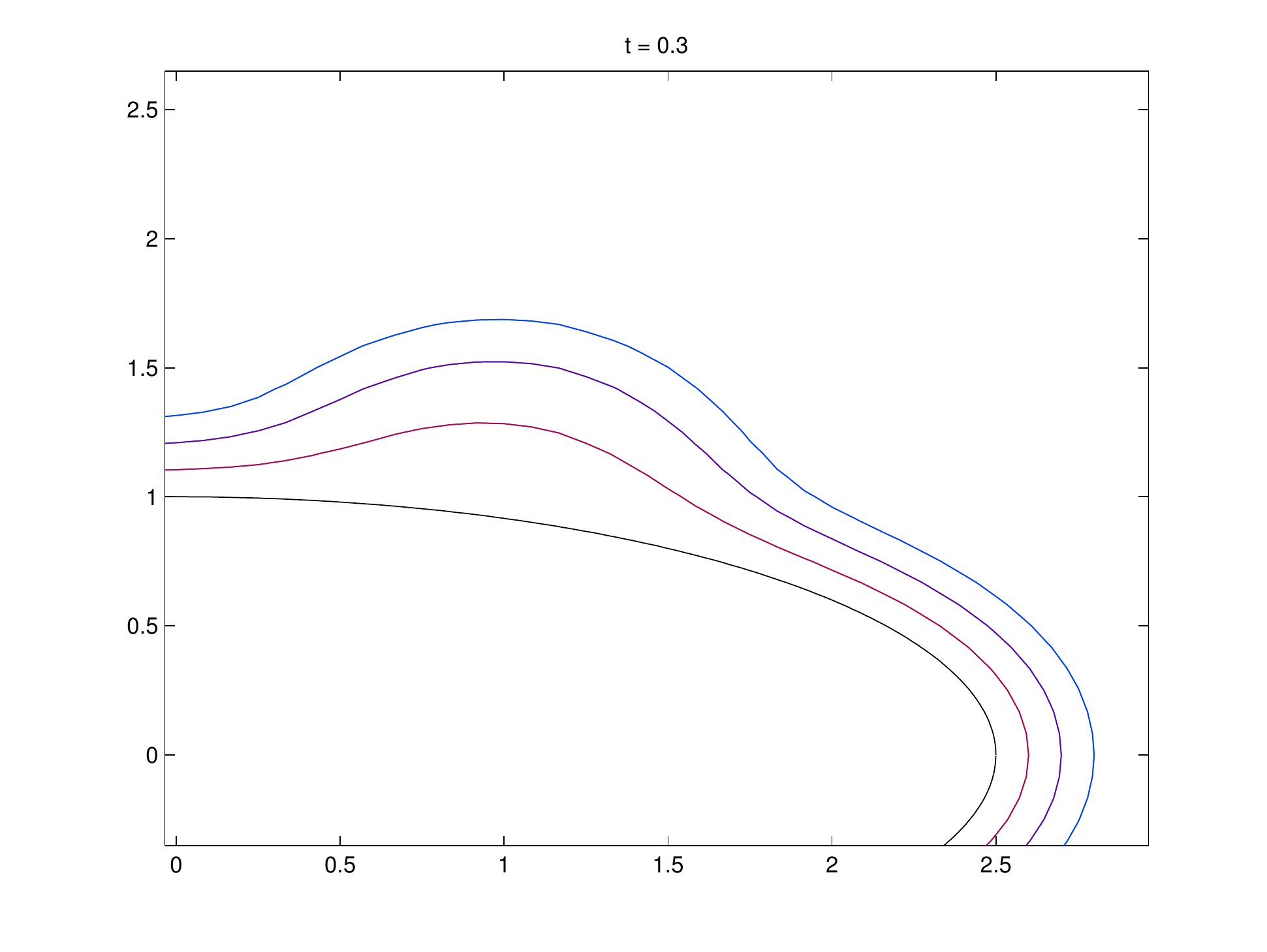}}
           \scalebox{0.4}{\includegraphics{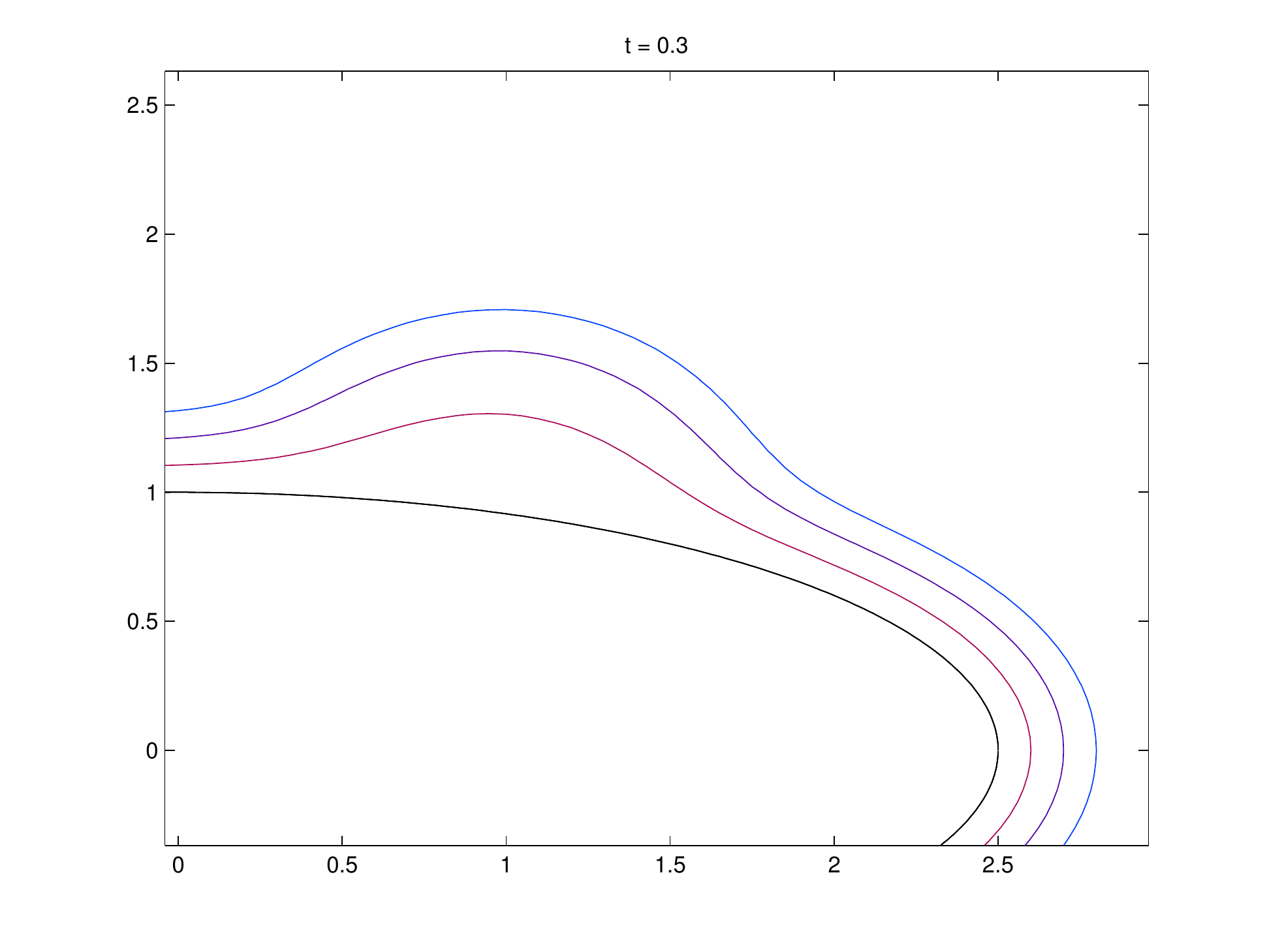}} \\
\caption{\small Zero level sets of the solution $\phi(\cdot,t)$ for $t = 0.1,0.2,0.3$  in \textbf{Example 3} with $H^{+}$ and $d =2$; left: minimization/maximization principle, right: Lax-Friedrichs.}\label{exp_1}
     \end{center}
 \end{figurehere}

\begin{figurehere}
     \begin{center}
     \vskip -0.3truecm
   \scalebox{0.4}{\includegraphics{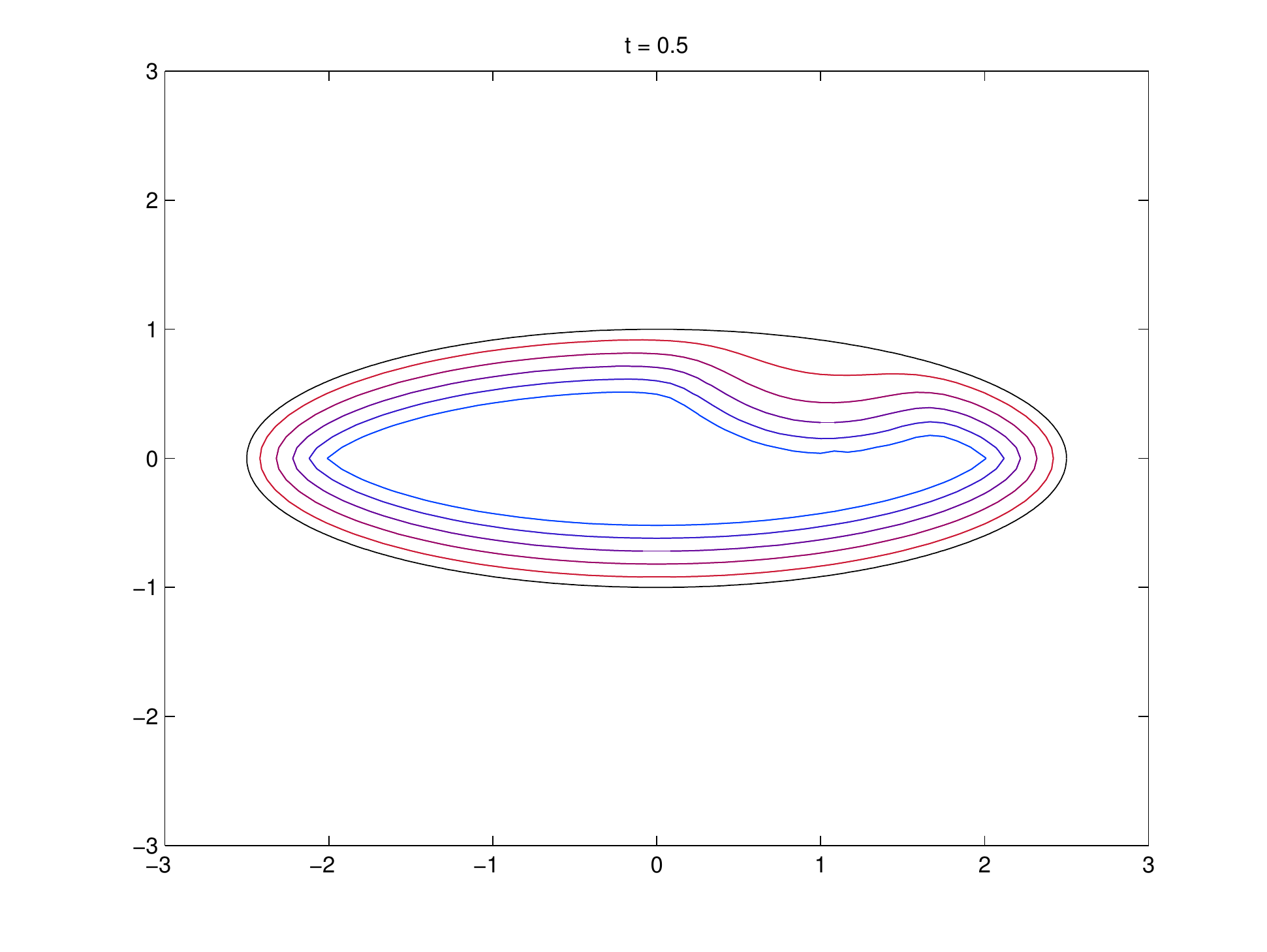}}
           \scalebox{0.4}{\includegraphics{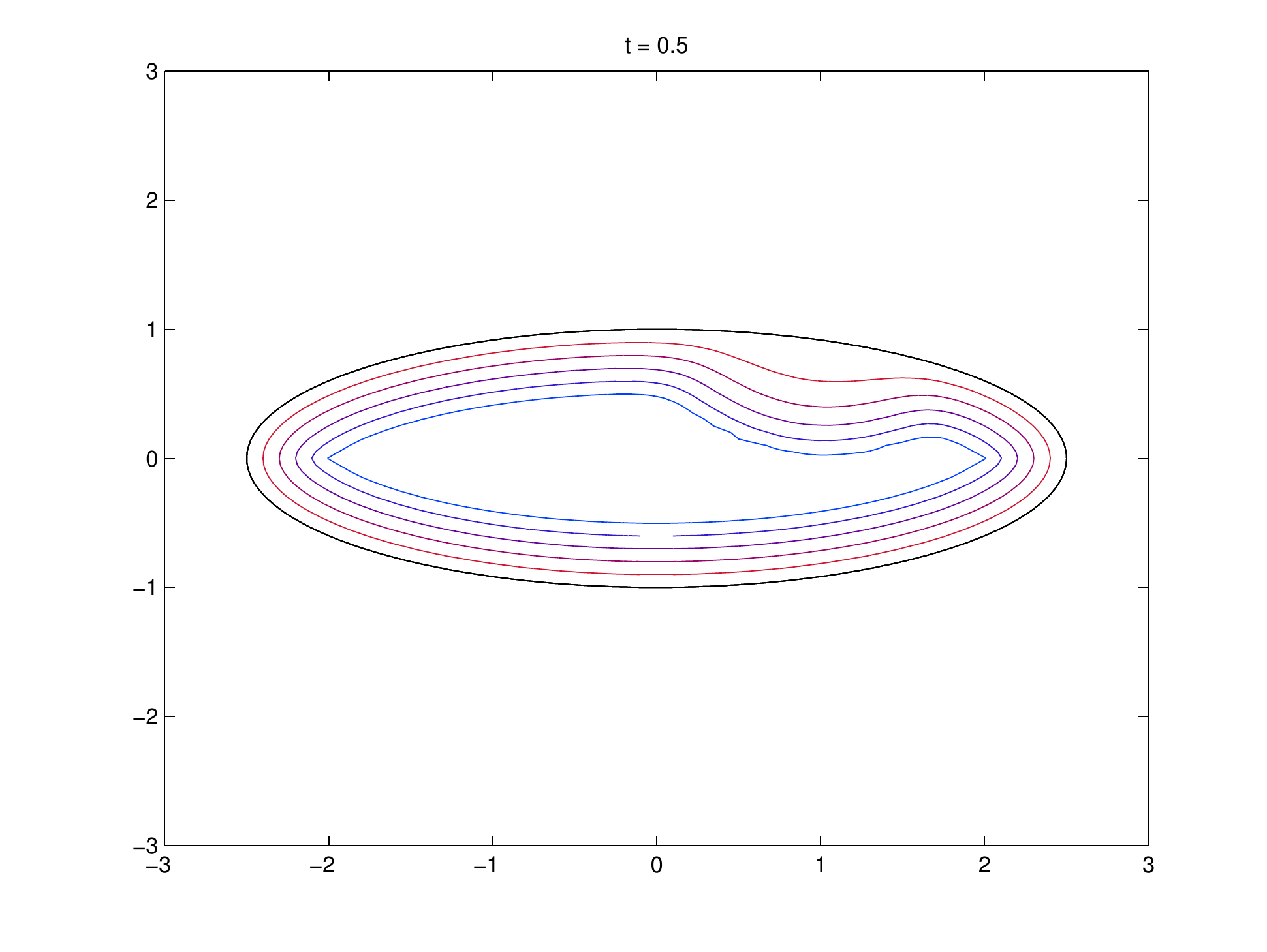}} \\
   \scalebox{0.4}{\includegraphics{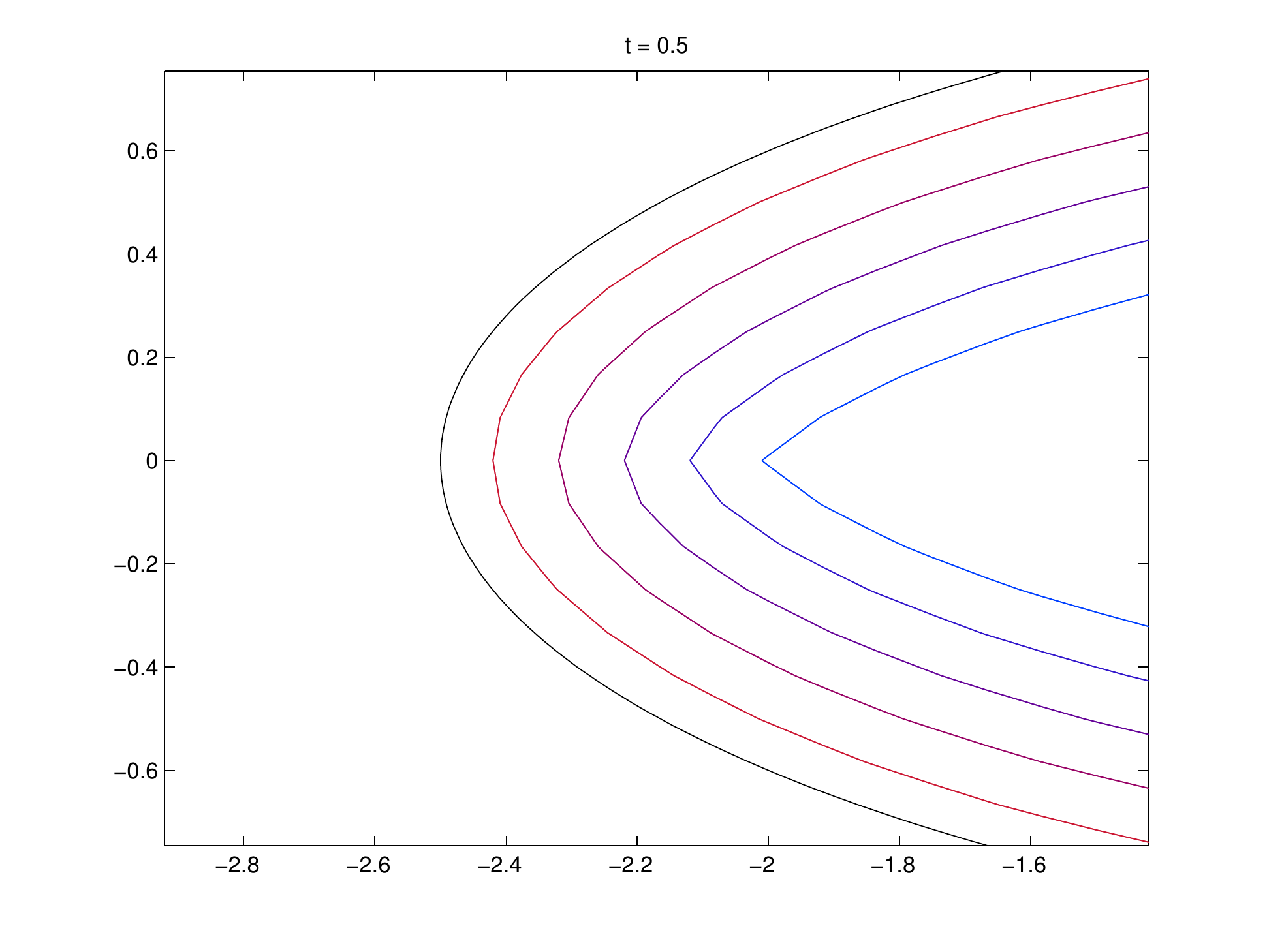}}
           \scalebox{0.4}{\includegraphics{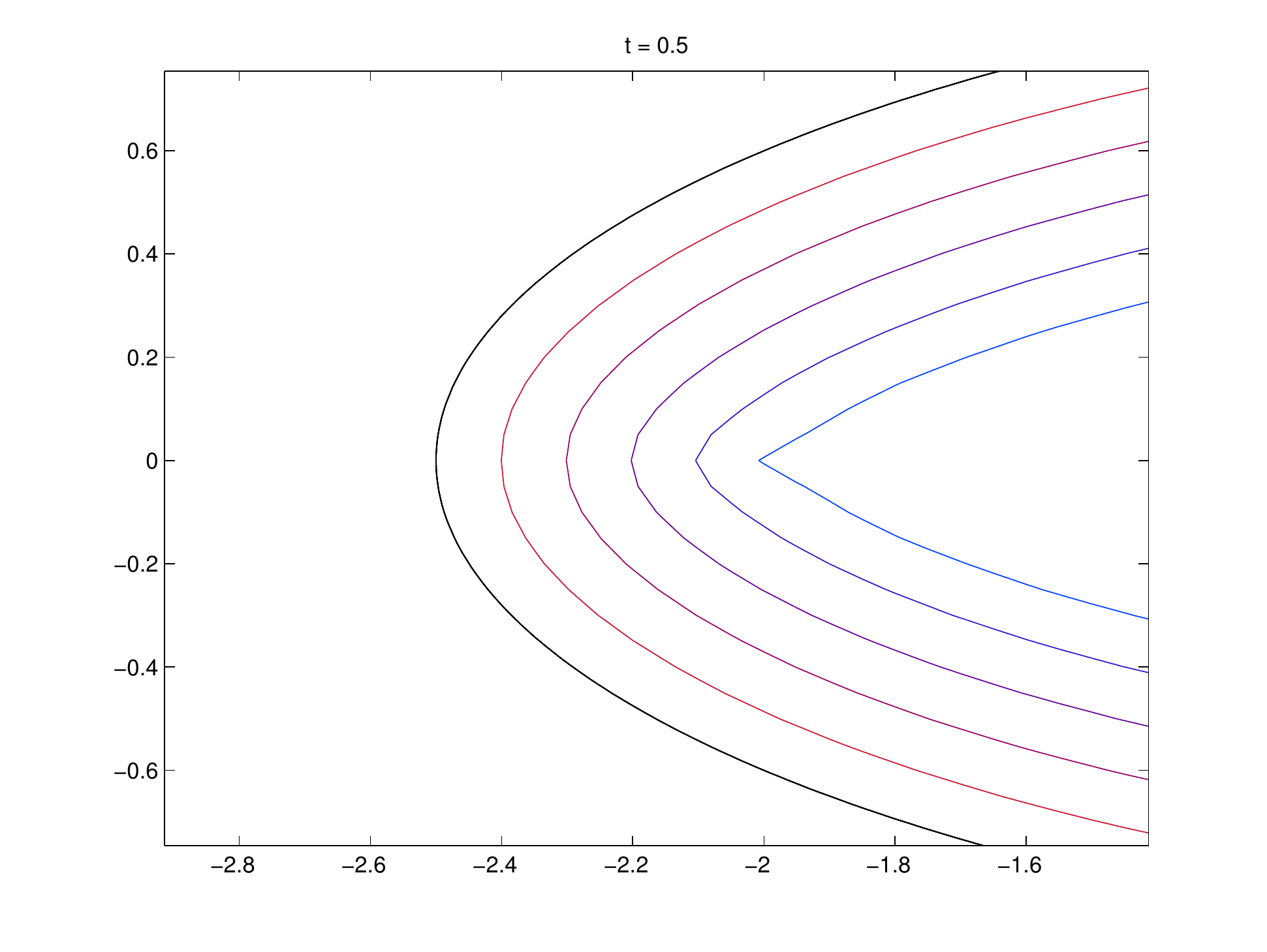}} \\
   \scalebox{0.4}{\includegraphics{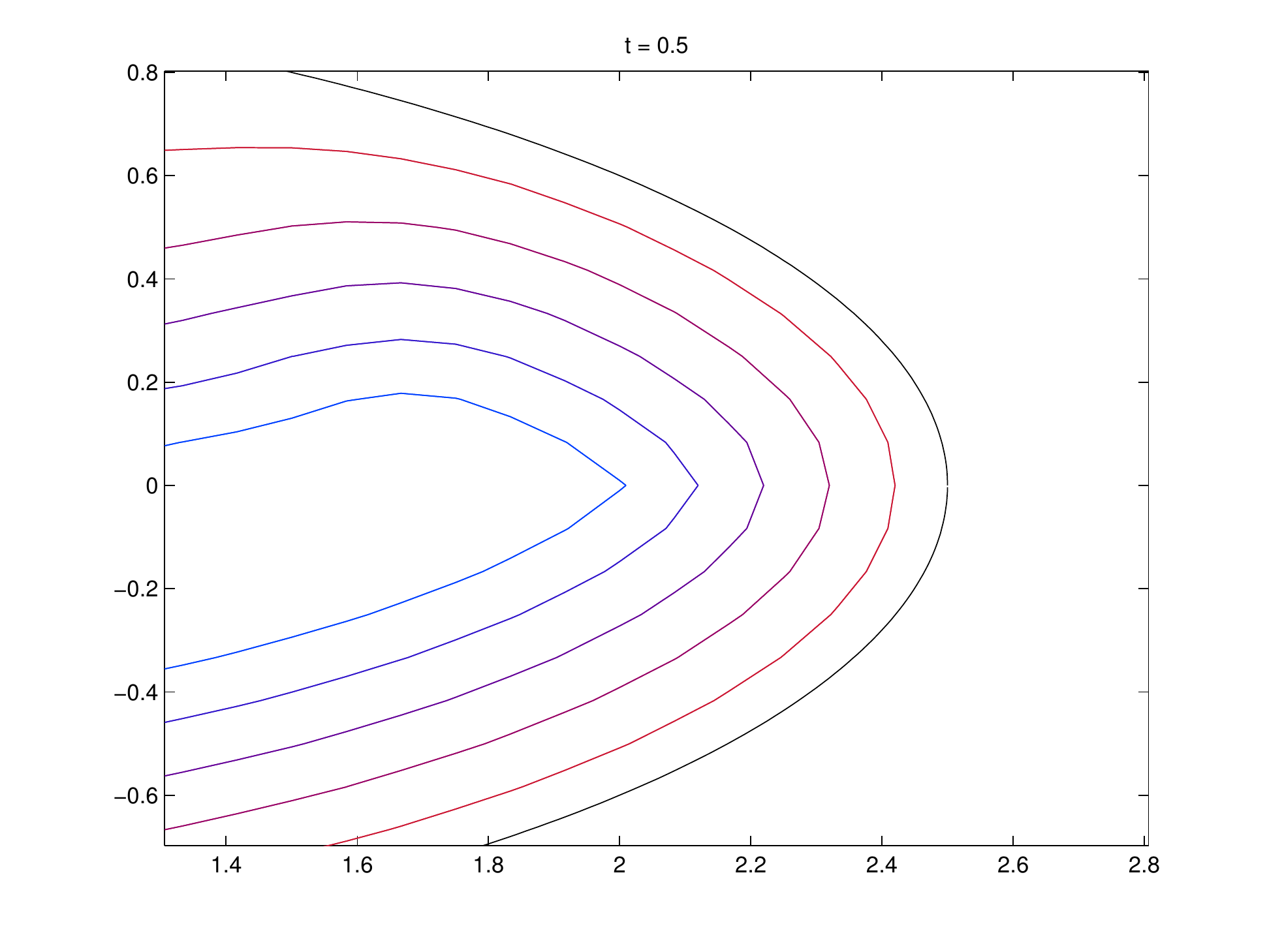}}
           \scalebox{0.4}{\includegraphics{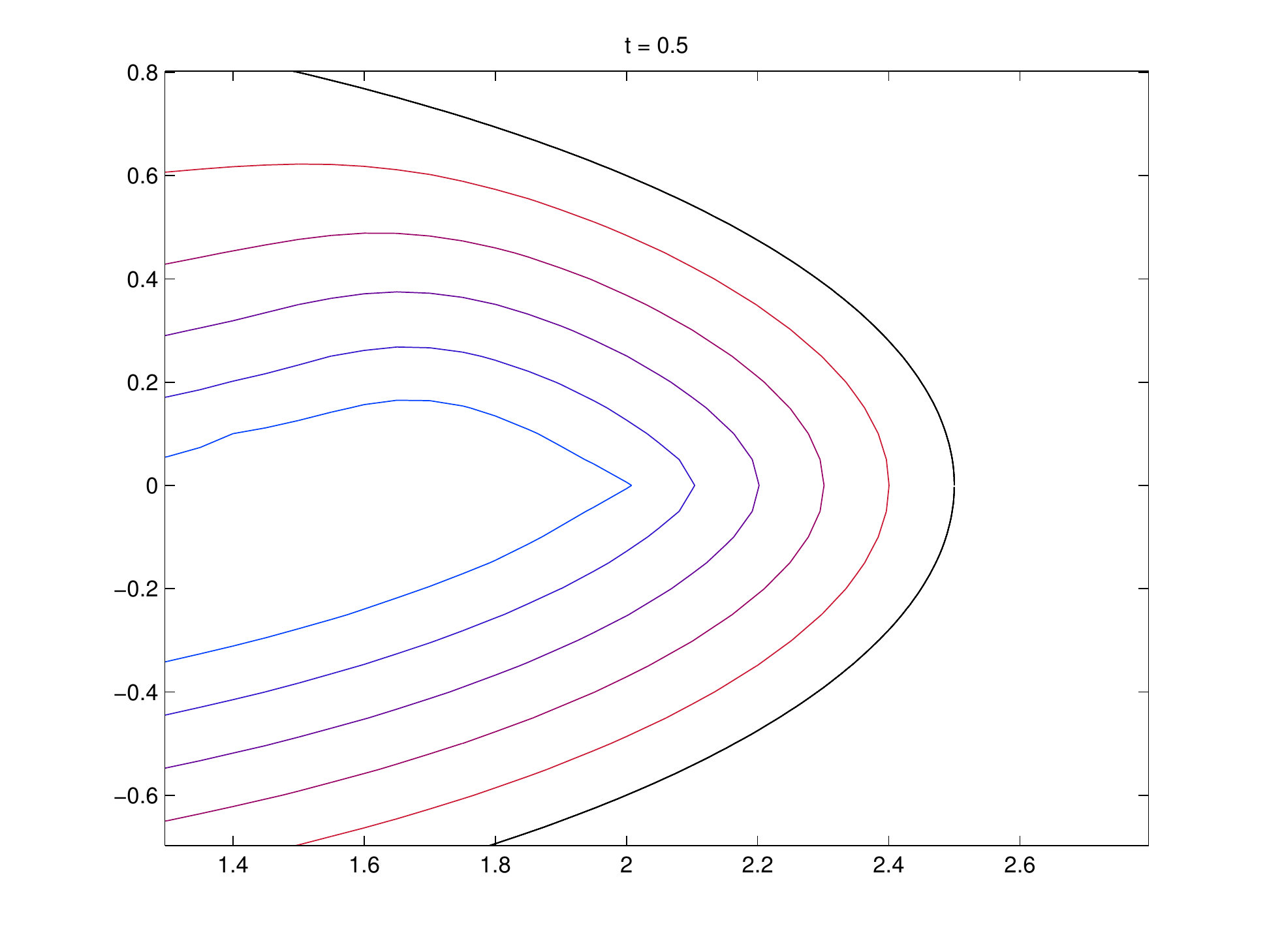}} \\
\caption{\small Zero level sets of the solution $\phi(\cdot,t)$ for $t = 0.1,0.2,...,0.5$  in \textbf{Example 3} with $H^{-}$ and $d =2$; left: minimization/maximization principle, right: Lax-Friedrichs.}\label{exp_2}
     \end{center}
 \end{figurehere}

\begin{figurehere}
     \begin{center}
     \vskip -0.3truecm
   \scalebox{0.4}{\includegraphics{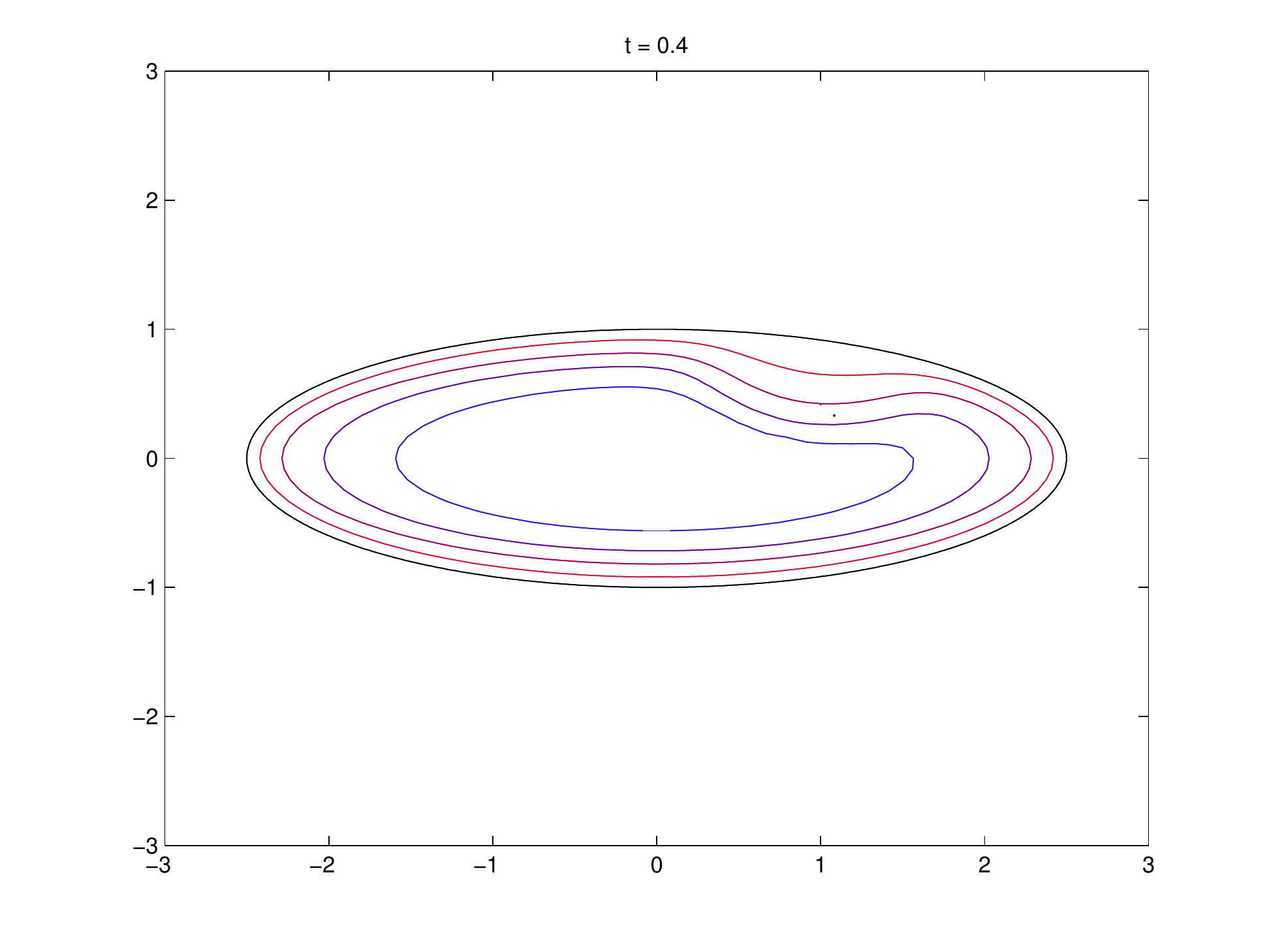}} \\
   \scalebox{0.4}{\includegraphics{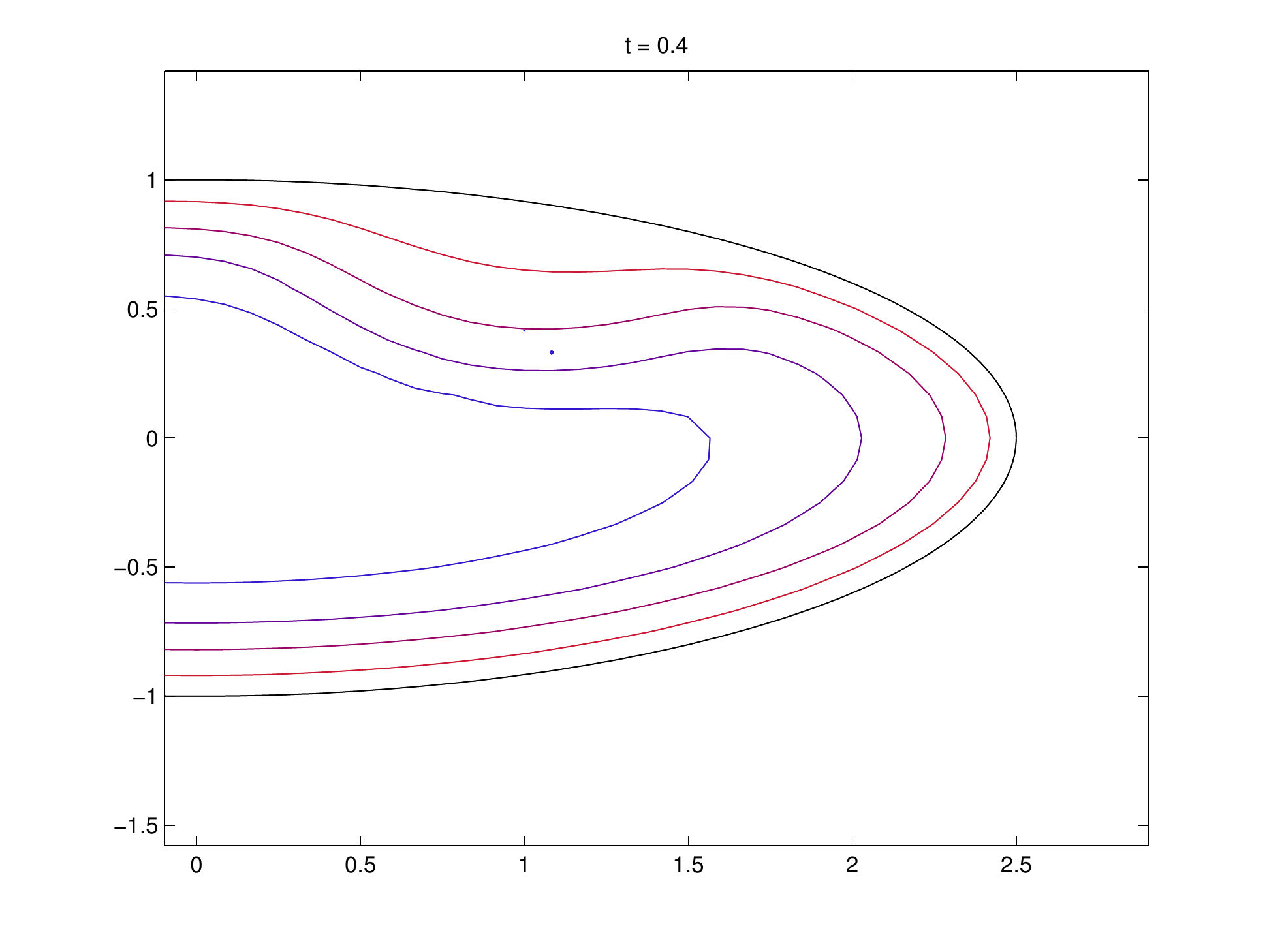}} \\
 \scalebox{0.4}{\includegraphics{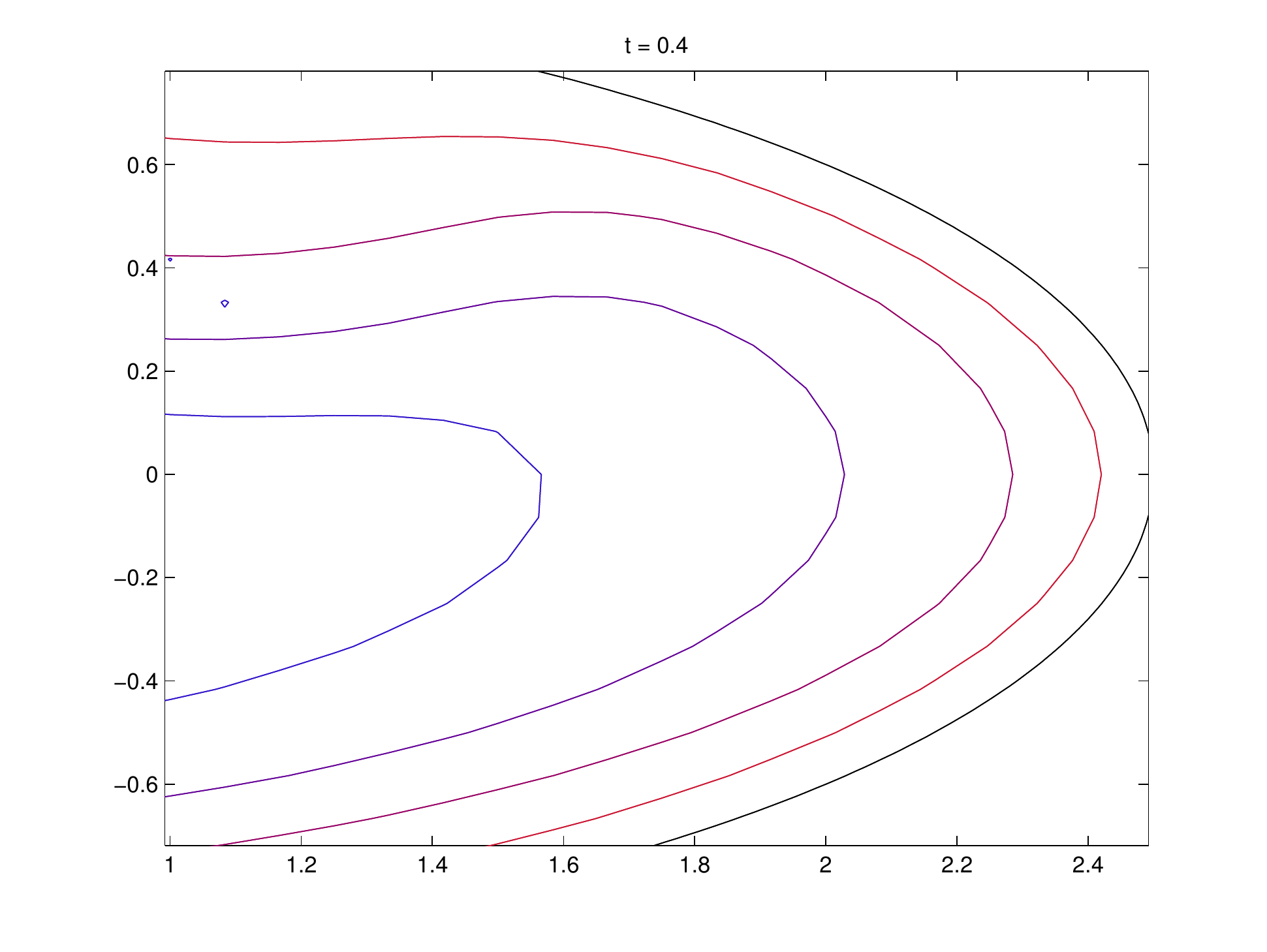}} \\
\caption{\small Zero level sets of the solution $\phi(\cdot,t)$ for $t = 0.1,0.2,...,0.4$  by minimization/maximization principle in \textbf{Example 3} with $H^{-}$ and $d =10$; top: full-size; middle and bottom: close-up.}\label{exp_3}
     \end{center}
 \end{figurehere}

In order to compare convergence of the method with respect to different discretization parameters, we compute the solutions with different $\Delta s$ and $\sigma$ to compare our result.  We choose $H^+$ and compare our solutions computed at the point $x = ( -0.93 , -0.35 ), t = 0.3$ with that of using the finest set of parameter $\Delta s = 0.005 $ and $\sigma= 0.01 $.
We first fix $\Delta s = 0.005 $ and compute the values of the function $\phi(x,t)$ using parameter different $\sigma = 0.02, 0.03,..,0.06$. The convergence table of the value with respect to the discretization parameter is shown Table \ref{convergencetable1}.  
In the comparison of this discretization parameter $\sigma$, the error does not show a clear trend of diminishing and shows a bit of oscillatory behaviour, although it is overall small as $\Delta s$ goes to $0$.
We then fix $\sigma = 0.01 $  and compute the values of the function $\phi(x,t)$ using parameter different $\Delta s =  0.01, 0.015,..,0.03 $. The convergence table of the value with respect to the discretization parameter is shown Table \ref{convergencetable2}.
In the comparison of this discretization parameter $\Delta s$, the error shows a clear trend of converging as goes to $0$.
In order to show convergence rate, as an example, Figure \ref{convergence} shows the convergence of the algorithm $\Delta s = 0.01 $ and $\sigma= 0.005$. As we may observe from the figure, the convergence is sublinear.

\begin{center}
\begin{tablehere}
\begin{tabular}{c|c}
$\sigma$ & Error \\
\hline
$ 0.06 $ &  $ 3.539 \times 10^{-4}   $  \\
$ 0.05 $  & $ 2.903 \times 10^{-4}   $ \\
$ 0.04 $ & $ 4.468 \times 10^{-4}   $ \\
$ 0.03 $  & $ 6.185   \times 10^{-4}   $  \\
$ 0.02 $  & $ 1.818 \times 10^{-4}   $
\end{tabular}
\caption{\small Convergence table of the value with respect to $\sigma = 0.02, 0.03,..,0.06$ and $\Delta s = 0.005 $
with $H^+$ in Example 3 at the point $x = ( -0.93 , -0.35 ), t = 0.3$.
.}\label{convergencetable1}
\end{tablehere}
\end{center}

\begin{center}
\begin{tablehere}
\begin{center}
\begin{tabular}{c|c}
$\Delta s$ & Error \\
\hline
$ 0.03 $ &  $ 9.466 \times 10^{-3}   $   \\
$ 0.025 $  & $ 7.804 \times 10^{-3}   $  \\
$ 0.02 $ &$ 6.083 \times 10^{-3}   $  \\
$ 0.015 $  & $ 4.314   \times 10^{-3}   $    \\
$ 0.01 $  &  $ 1.024 \times 10^{-3}   $
\end{tabular}
\end{center}
\caption{\small Convergence table of the value with respect to $\sigma = 0.01 $ and $\Delta s =  0.01, 0.015,..,0.03 $
with $H^+$ in Example 3 at the point $x = ( -0.93 , -0.35 ), t = 0.3$.
.}\label{convergencetable2}
\end{tablehere}
\end{center}

\begin{figurehere}
    \begin{center}
     \vskip -0.3truecm
   \scalebox{0.4}{\includegraphics{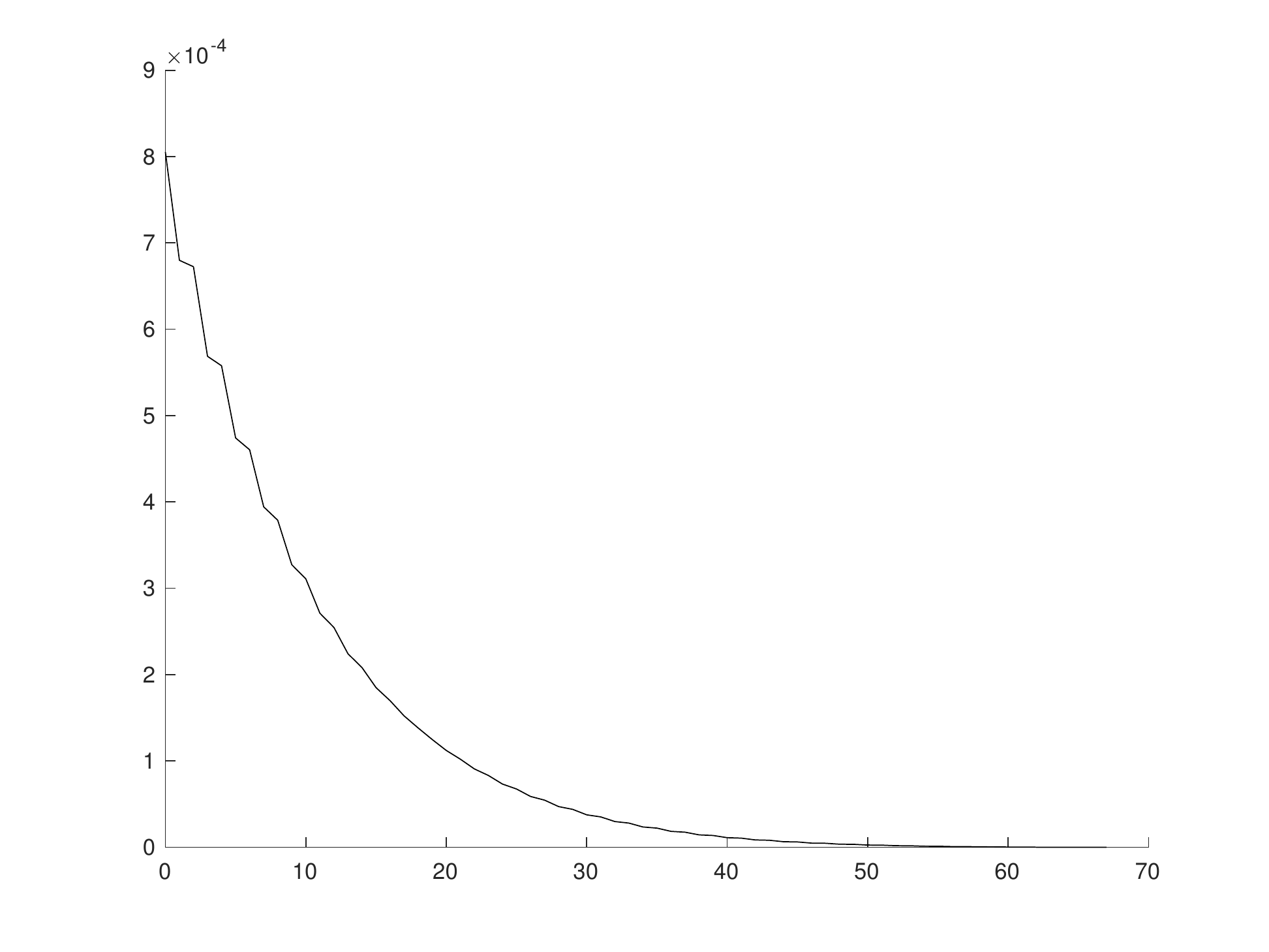}}
\caption{\small Error with respect to the number of iterations with $H^+$ in Example 3 at the point $x = ( -0.93 , -0.35 ), t = 0.3$ when
$\Delta s = 0.01 $ and $\sigma= 0.005$ .}\label{convergence}
     \end{center}
\end{figurehere}

\noindent \textbf{Example 4} To test our maximization principle for a general state-dependent non-convex Hamilton-Jacobi equation, we use a state-dependent non-convex Hamiltonian of the following form given in \cite{evans} but a different problem from \cite{evans}:
\begin{equation*}
H(x,p,t) = - c(x) p_1 + 2|p_2|  - \sqrt{|p_1|^2 + |p_2|^2} -1\,,
\end{equation*}
where we write $p = (p_1, p_2)$ and
\begin{equation*}
c(x) = 2 \left( 1 + 3 \exp( - 4 | x - (1,1 ) |_2^2) \right) \,.
\end{equation*}
The maximization principle \eqref{hopf_formula} is used to compute the solution $\varphi$.
The temporal stepsize is chosen as $\Delta s = 0.005$.  
The other constants are chosen as follows: stepsize $\sigma = 0.001$ and $L= 4$.
Figure \ref{exp_5} gives the solutions with $T= 0.1$. 
The runtime using C++ is $ 7.279 \times 10^{-2} s \times 5$ per point.
Figure \ref{exp_5} (right) is the solution computed by the Lax-Friedrichs scheme for comparison.
Thes example is to illustrate that the maximization principle coincide with the Lax-Friedrichs solution in a quite general case.
To remark, for this case, the example does not satisfy the assumptions of any of the lemmas proved in this work.

\begin{figurehere}
    \begin{center}
     \vskip -0.3truecm
   \scalebox{0.4}{\includegraphics{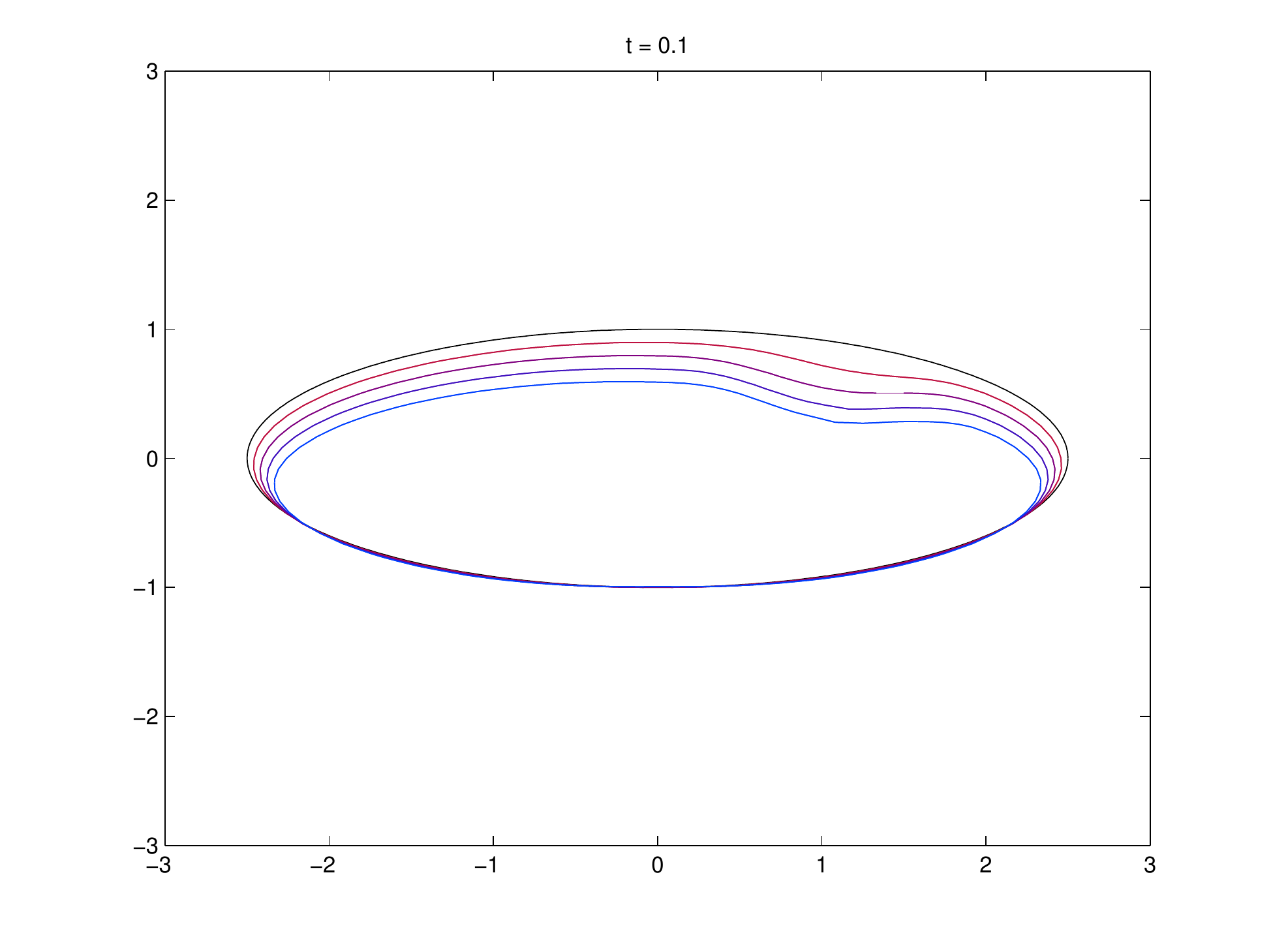}}
           \scalebox{0.4}{\includegraphics{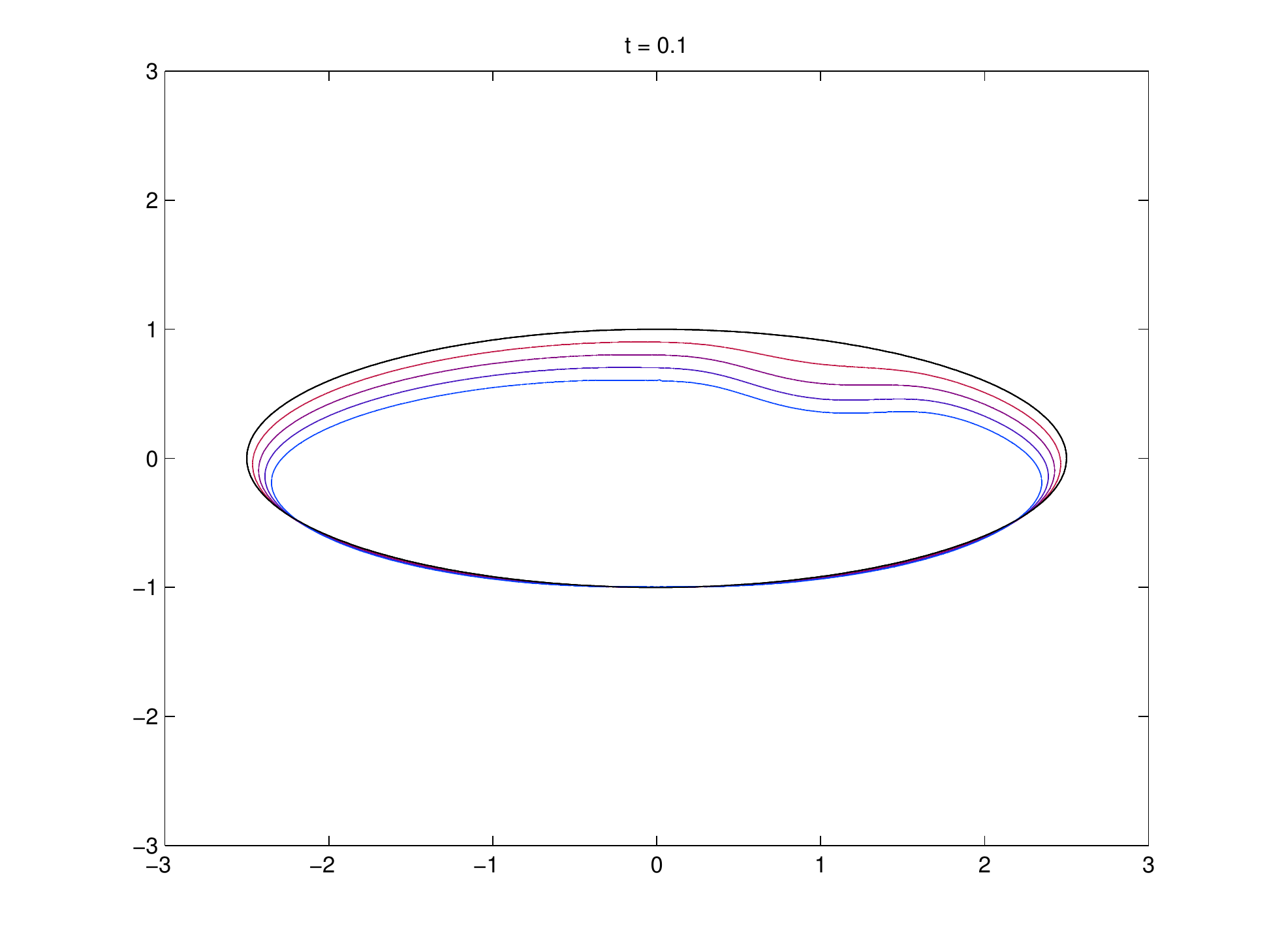}} \\
  \scalebox{0.4}{\includegraphics{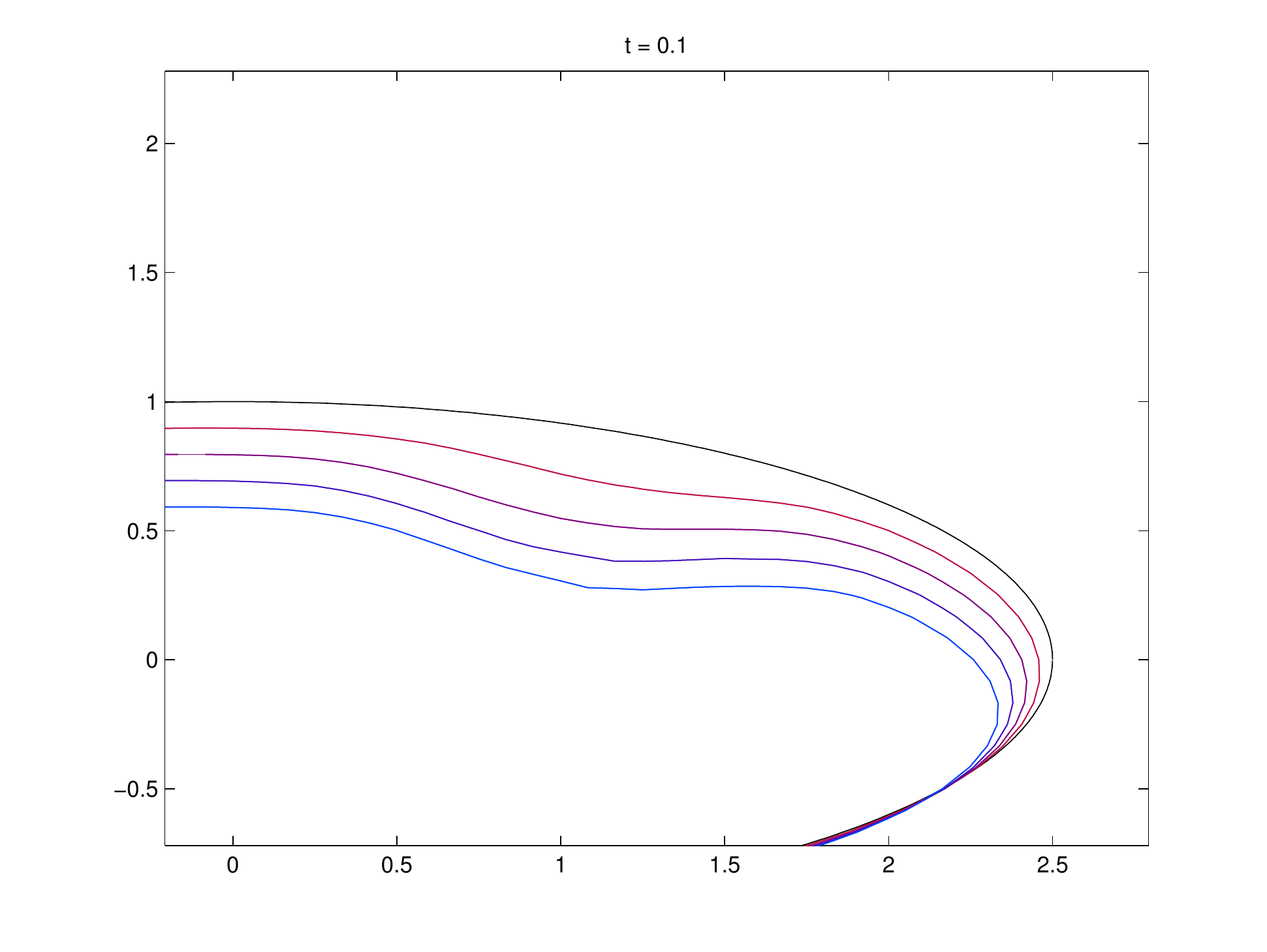}}
           \scalebox{0.4}{\includegraphics{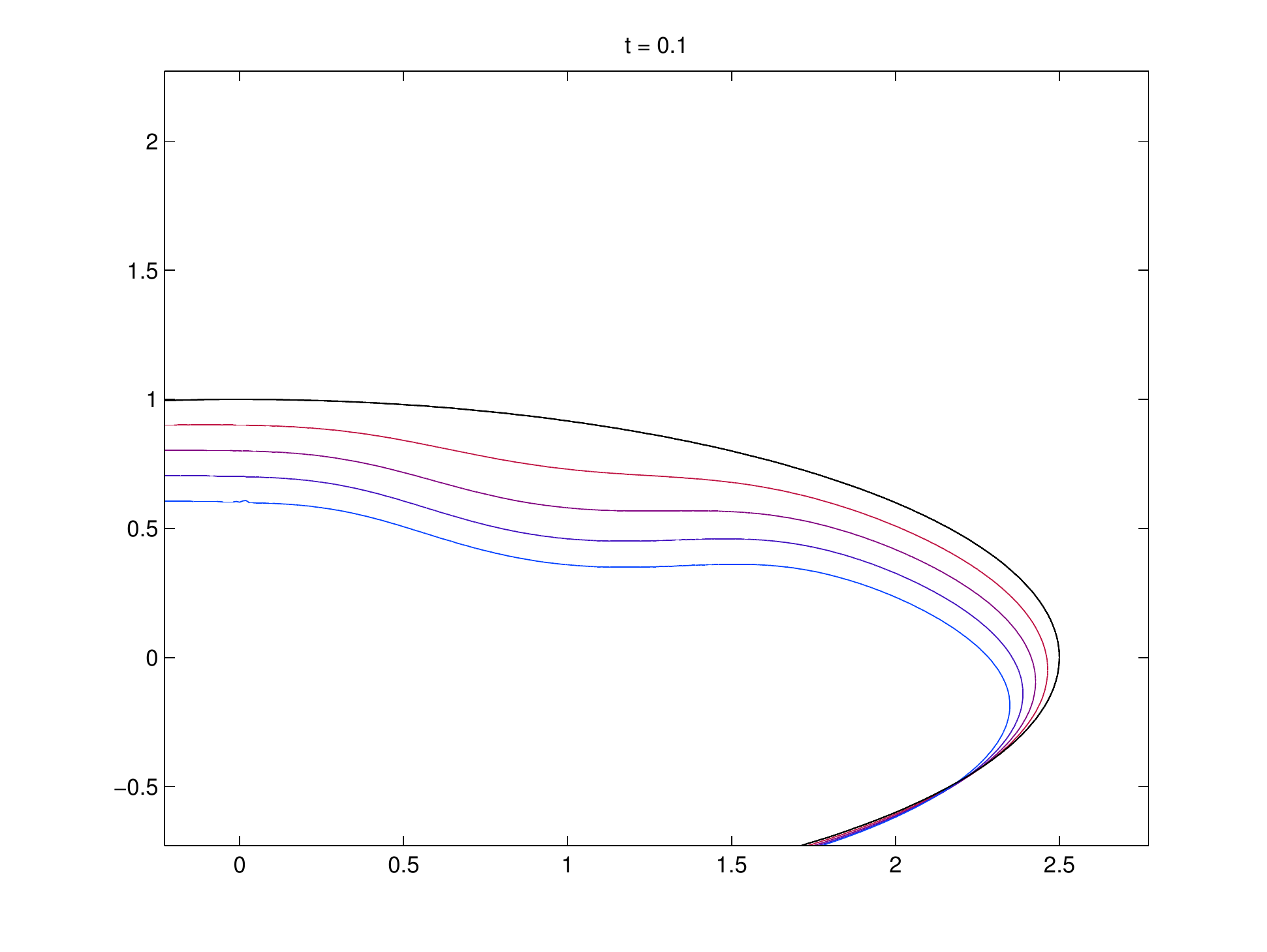}} \\
  \scalebox{0.4}{\includegraphics{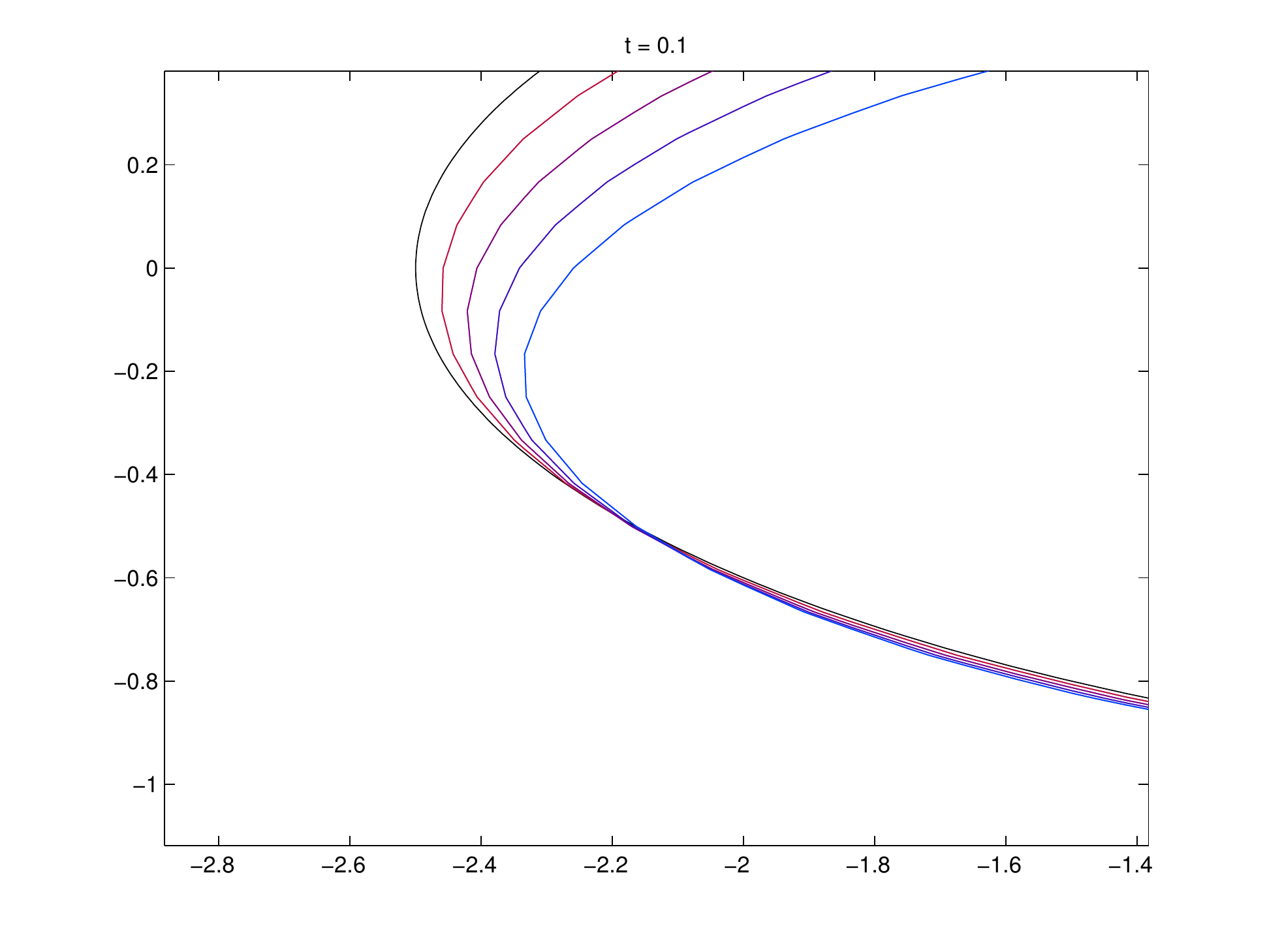}}
           \scalebox{0.4}{\includegraphics{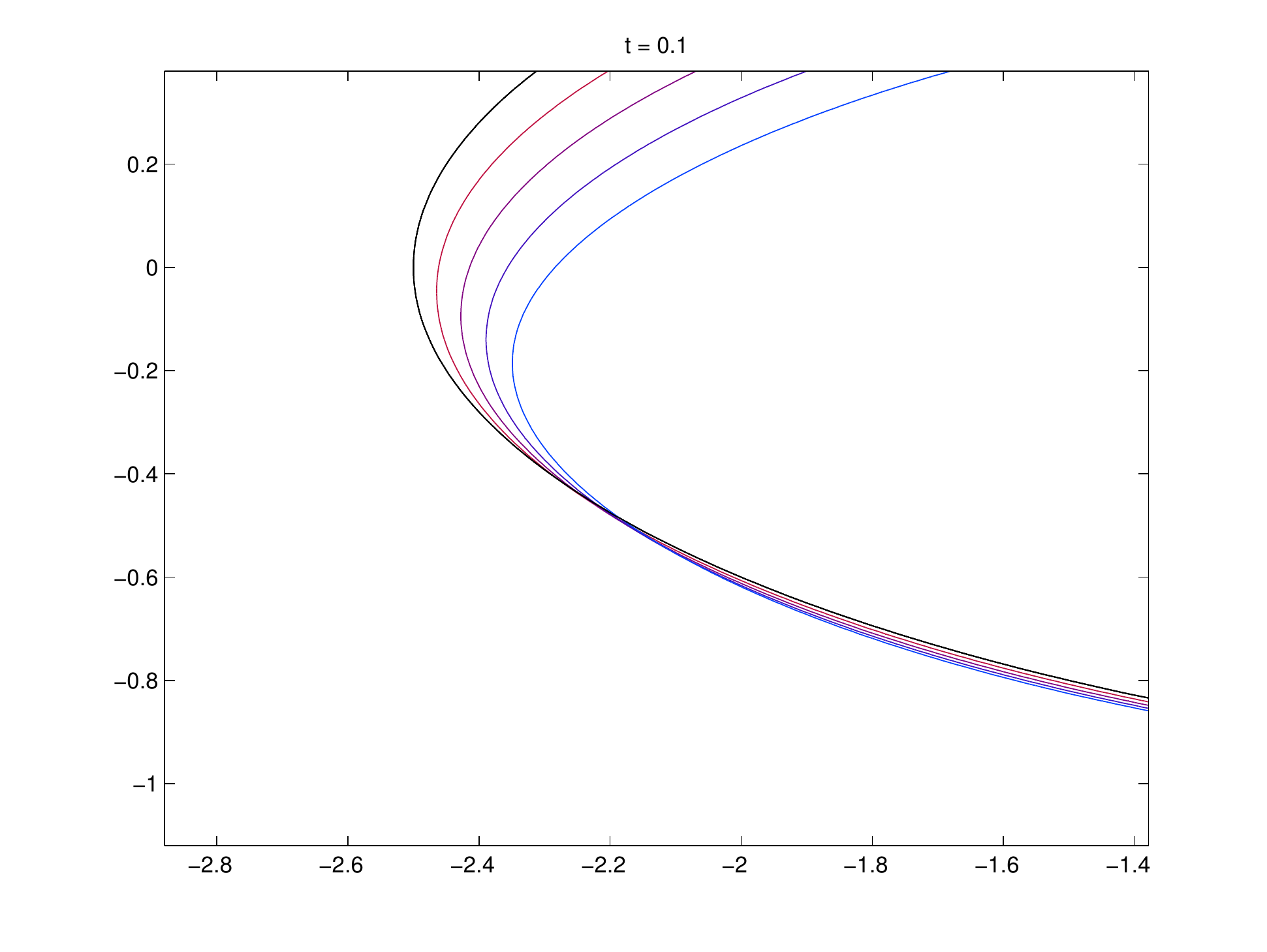}} \\
\caption{\small Zero level sets of the solution $\phi(\cdot,t)$ for $t = 0.025,0.05,...,0.1$  in \textbf{Example 4} with $d =2$; left: minimization/maximization principle, right: Lax-Friedrichs.}\label{exp_5}
     \end{center}
\end{figurehere}

\noindent \textbf{Example 5} We solve for the state-dependent non-convex Hamiltonian of the form
\begin{equation*}
H(x,p,t) = c_1(x) | p_{1,...,k} |_2 - c_2(x) | p_{k+1,...,d} |_2\,,
\end{equation*}
where
\begin{equation*}
c_1(x) = c(x)\,, \quad  c_2(x) = c(-x) \,, \quad c(x) = 2 \left( 1 + 3 \exp( - 4 | x - (1,1 ) |_2^2) \right) \,.
\end{equation*}
In this case
The maximization principle \eqref{hopf_formula} is used to compute the solution $\varphi$.
The constants are chosen as follows: the temporal stepsize $\Delta s = 0.02$, stepsize in numerical differentiation $\sigma = 0.001$ and $L= 50$.

Figure \ref{exp_4} gives the solutions when $d = 2$ and $k=1$, $T= 0.3$. 
The runtime using C++ is $ 9.094  \times 10^{-2} s \times 20$ per point.
To compare, Figure \ref{exp_4} (right) is the solution computed by Lax-Friedrichs scheme for comparison.  Now there is small defect in the solution computed by the maximization principle at one point of the wave-front close to $x = (-1,-0.4)$, owing to the high non-convexity and non-smoothness of the corresponding functionals around that point.

In order to pin down the exact problem causing the defect, let us fix $ x = (-0.93,-0.35)$ and $t = 0.3$
Figure \ref{exp_4def} (left) shows the functional to be minimizied.  The black star is the global minimizer.  It is clear that we now have a local minimizer (attractor) which has a comparable size of basin of attraction as the global minimizer.  Moreover consider Figure \ref{exp_4def} (right), which shows the norm of the gradient for the functional to be minimized, we can clearly see that around the two basins of attractions, there is a long V-shape trench with gradient of very small magnitude.  With this kind of structure, it is very likely that a gradient type method either get stuck at a local minimum (as one may see, there is almost at least $0.5$ probability that one may fall into the basin of attraction of the local minimum); or it will drag very slowly along the trench and is very difficult to leave it before converging to a global minimum.
In fact, this truly happen quite often.  As shown in Figure \ref{exp_4def} (left) the  red stars shows the path of an iteration of coordinate descent with a random initial guess and green star is the final iterate.  It is clear that the iterates drag along the trench and fall into the attraction basin of a local minimum. 
In fact, this structure of the functional make descent-type algorithms less easy to find the global minimum, unless better optimization methods are considered, e.g. adding momentum \cite{momentum2,momentum3,momentum1}.
Figure \ref{exp_45} gives the solutions when $d = 7$ and $k=1$, $T= 0.3$. 
The runtime using C++ is $ 5.428 \times 10^{-1} s \times 20$ per point.
The computation time is still acceptable for a $7$ dimensional problem which is fully non-convex and state-dependent.
To remark, again, this example does not satisfy the assumptions of any of the lemmas proved in this work.

\begin{figurehere}
     \begin{center}
     \vskip -0.3truecm
   \scalebox{0.4}{\includegraphics{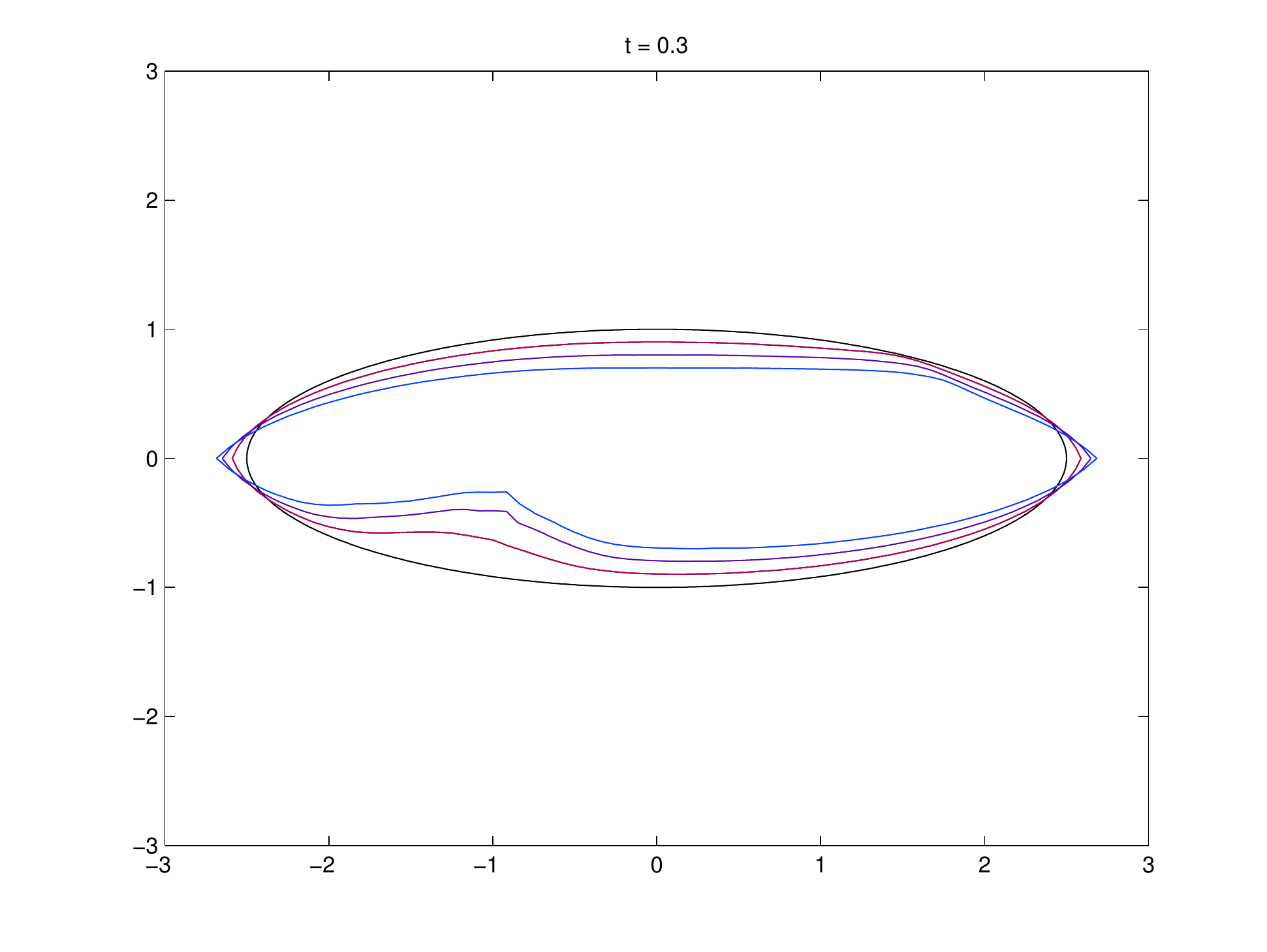}}
           \scalebox{0.4}{\includegraphics{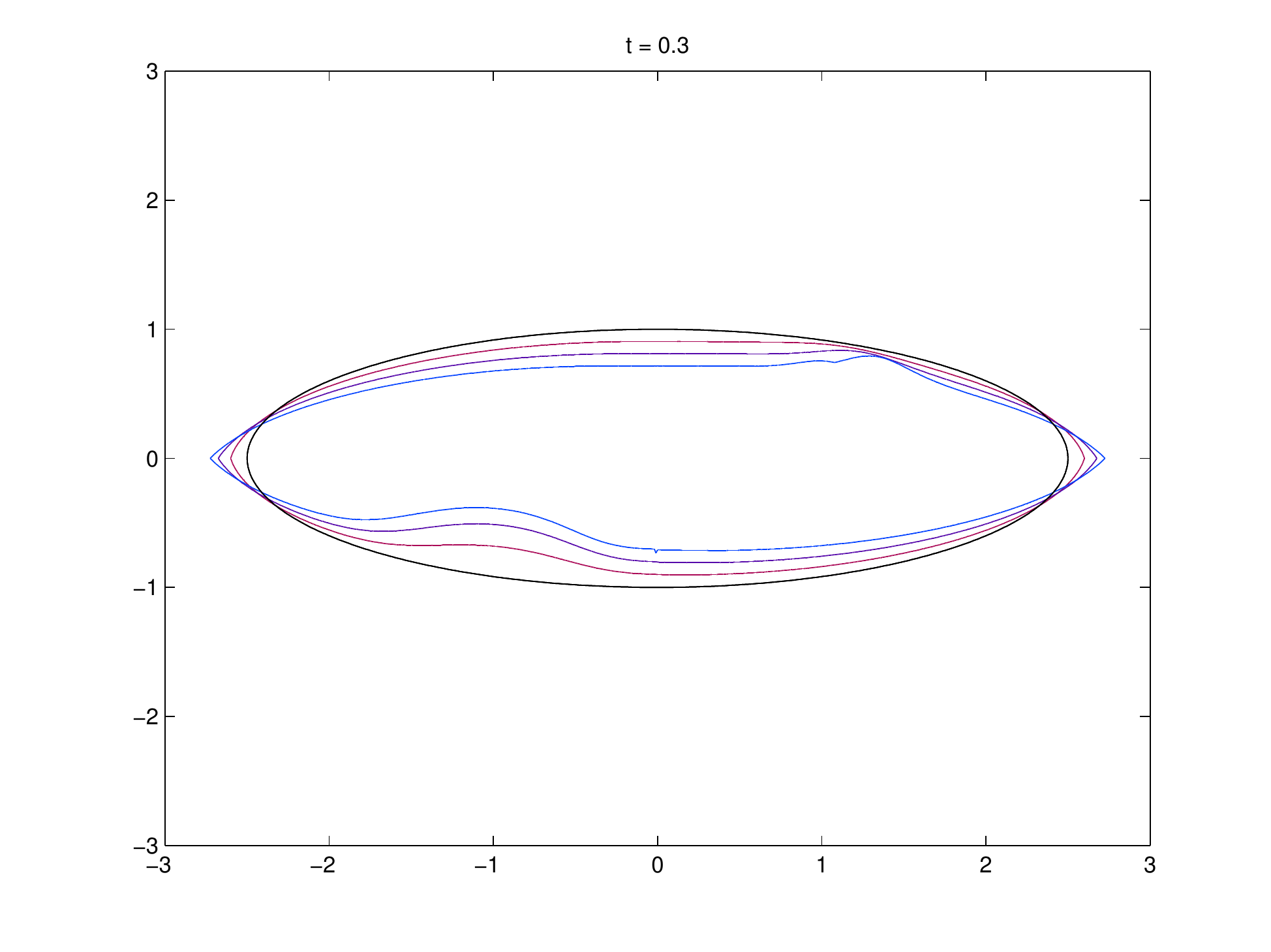}} \\
   \scalebox{0.4}{\includegraphics{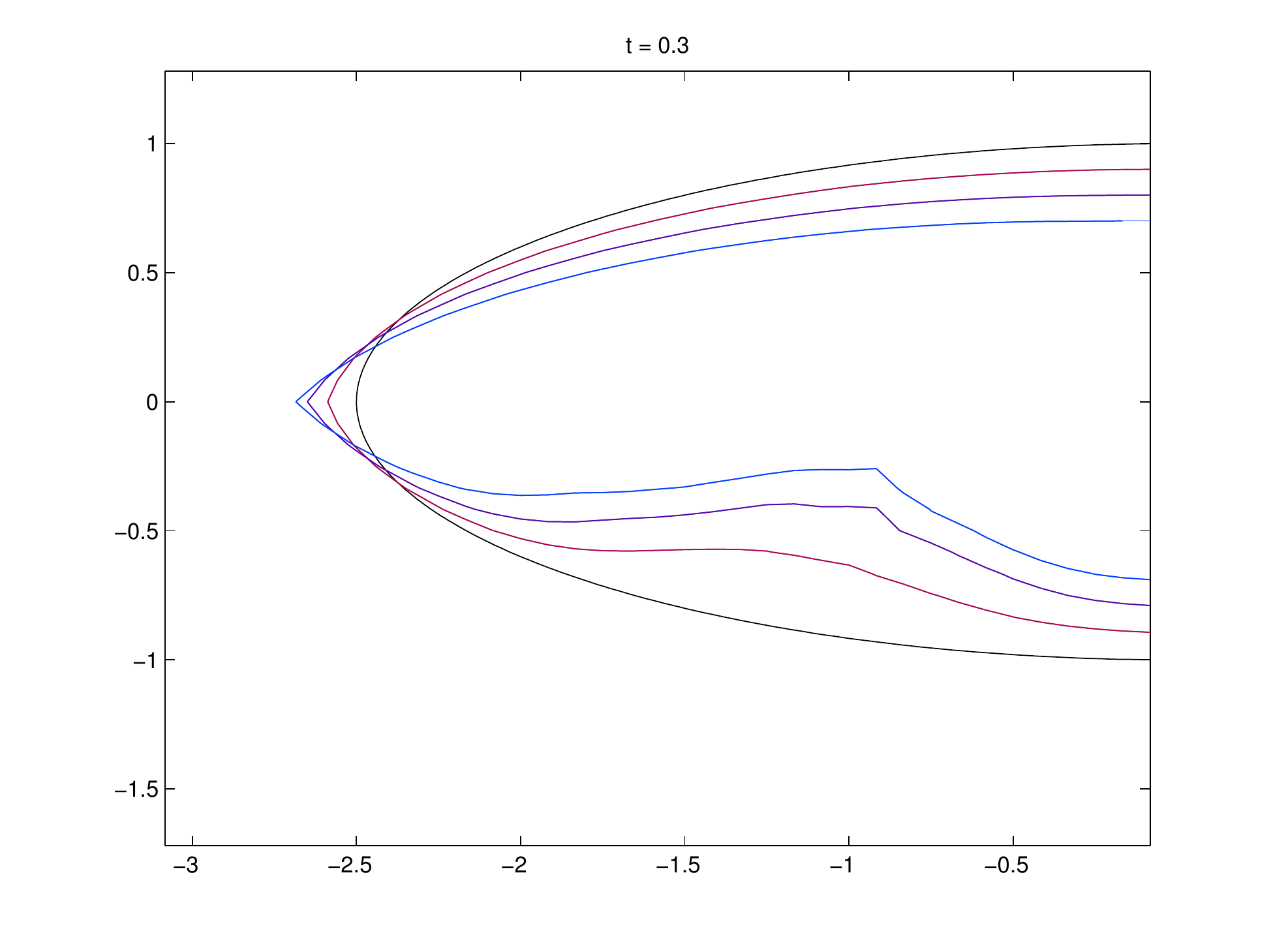}}
           \scalebox{0.4}{\includegraphics{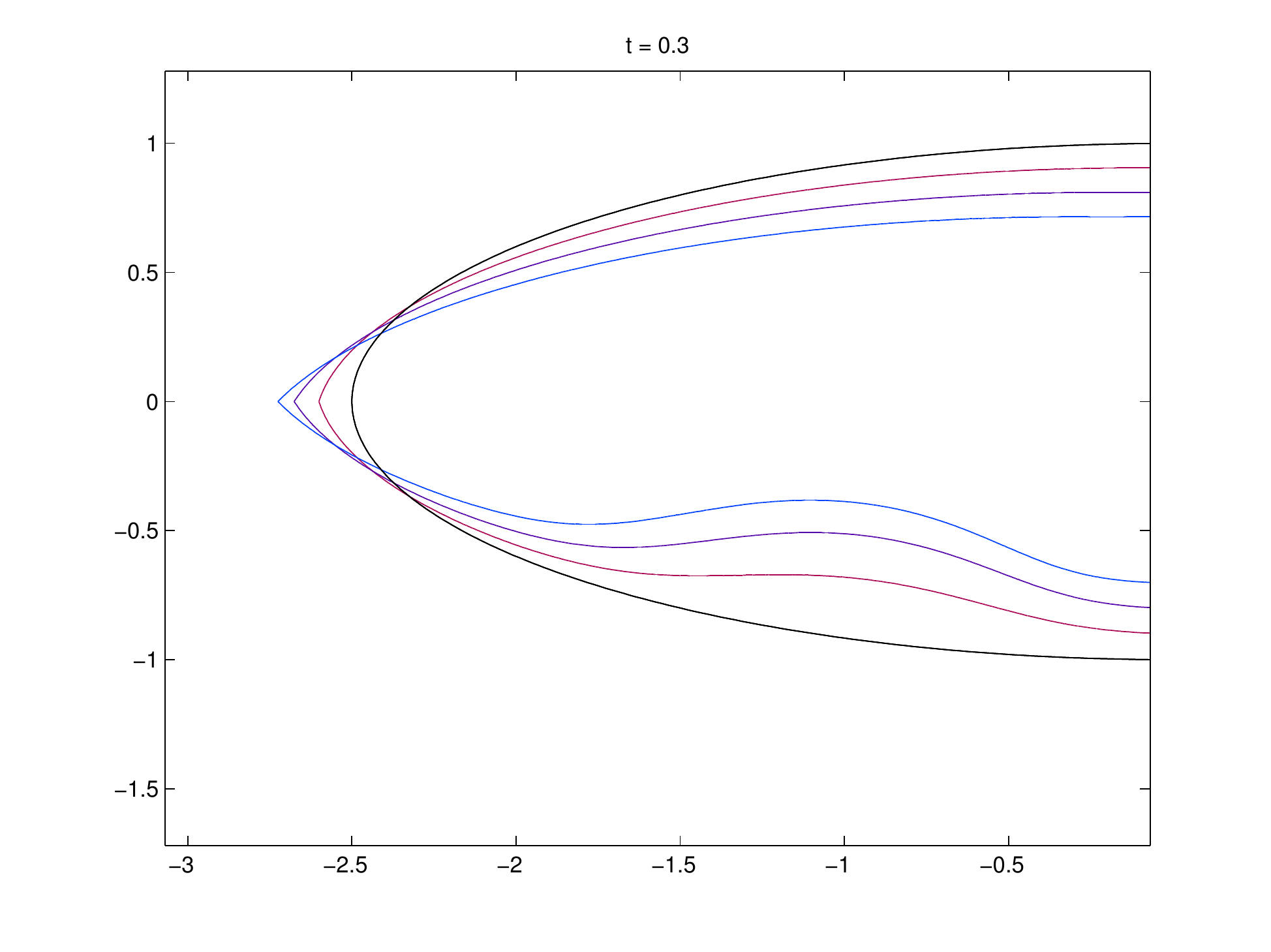}} \\
   \scalebox{0.4}{\includegraphics{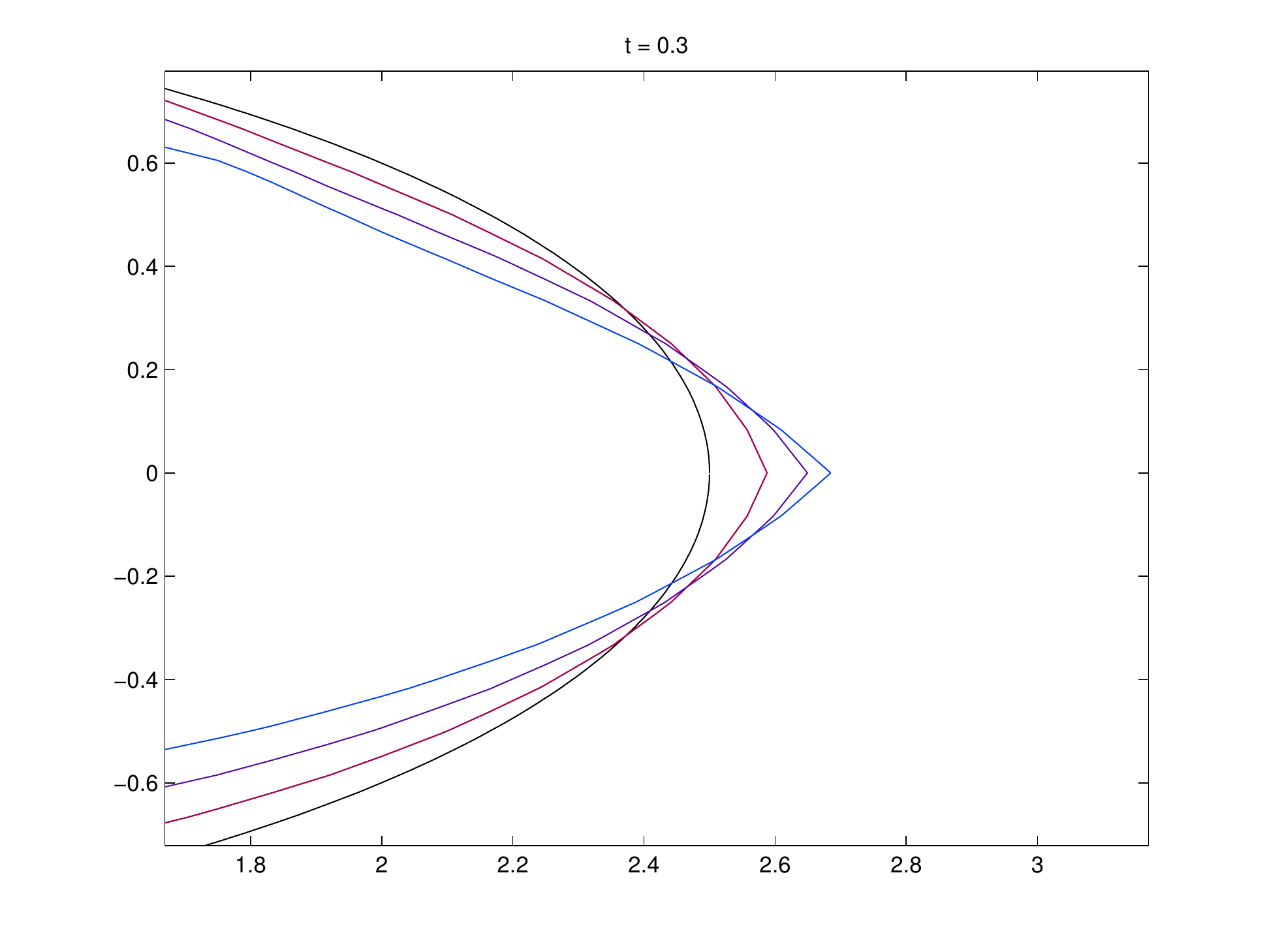}}
           \scalebox{0.4}{\includegraphics{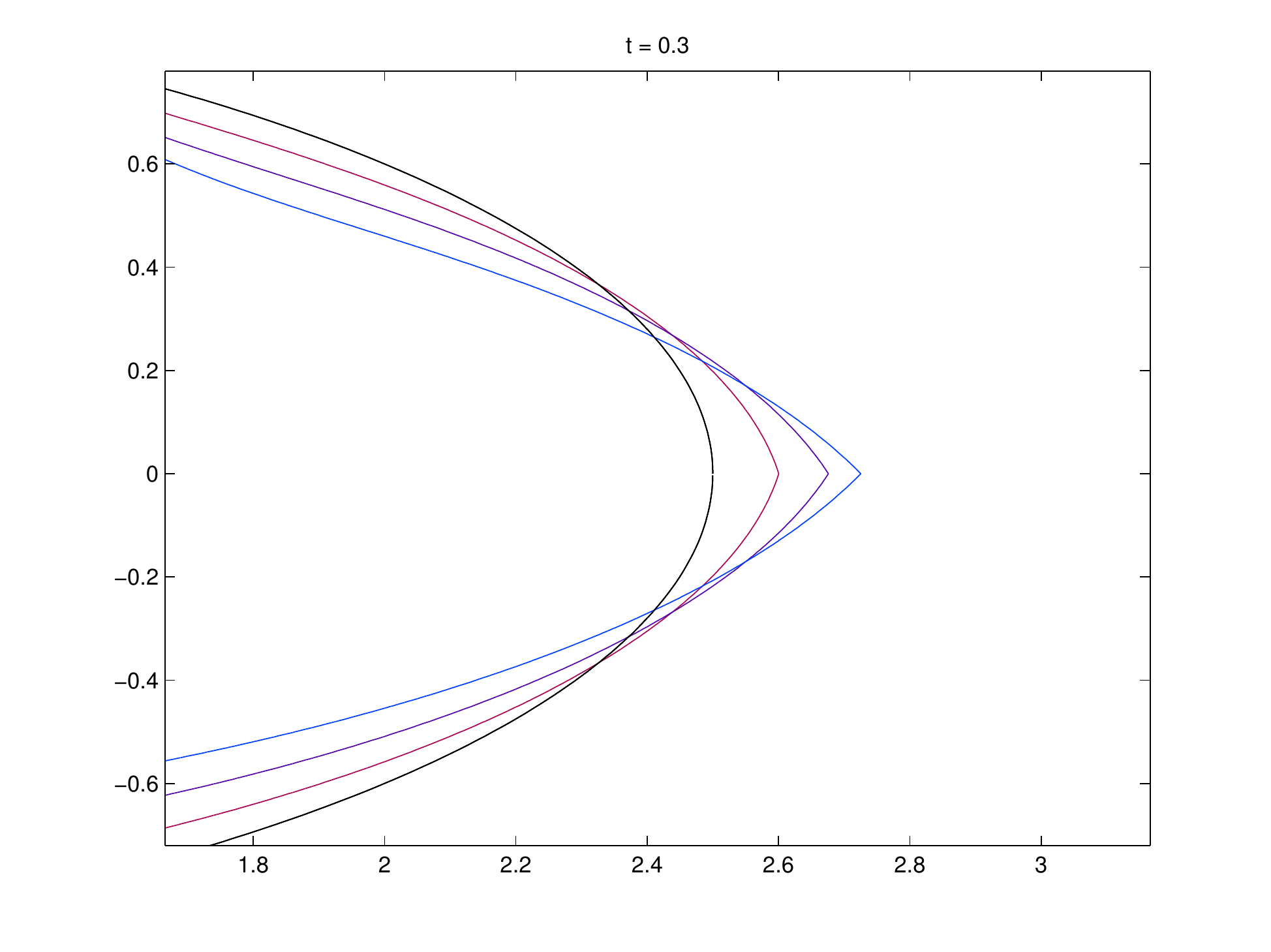}} \\
\caption{\small Zero level sets of the solution $\phi(\cdot,t)$ for $t = 0.1,0.2,0.3$; left: minimization/maximization principle in \textbf{Example 5} with $k=1$ and $d =2$, right: Lax-Friedrichs.}\label{exp_4}
     \end{center}
 \end{figurehere}

\begin{figurehere}
     \begin{center}
     \vskip -0.3truecm
   \scalebox{0.4}{\includegraphics{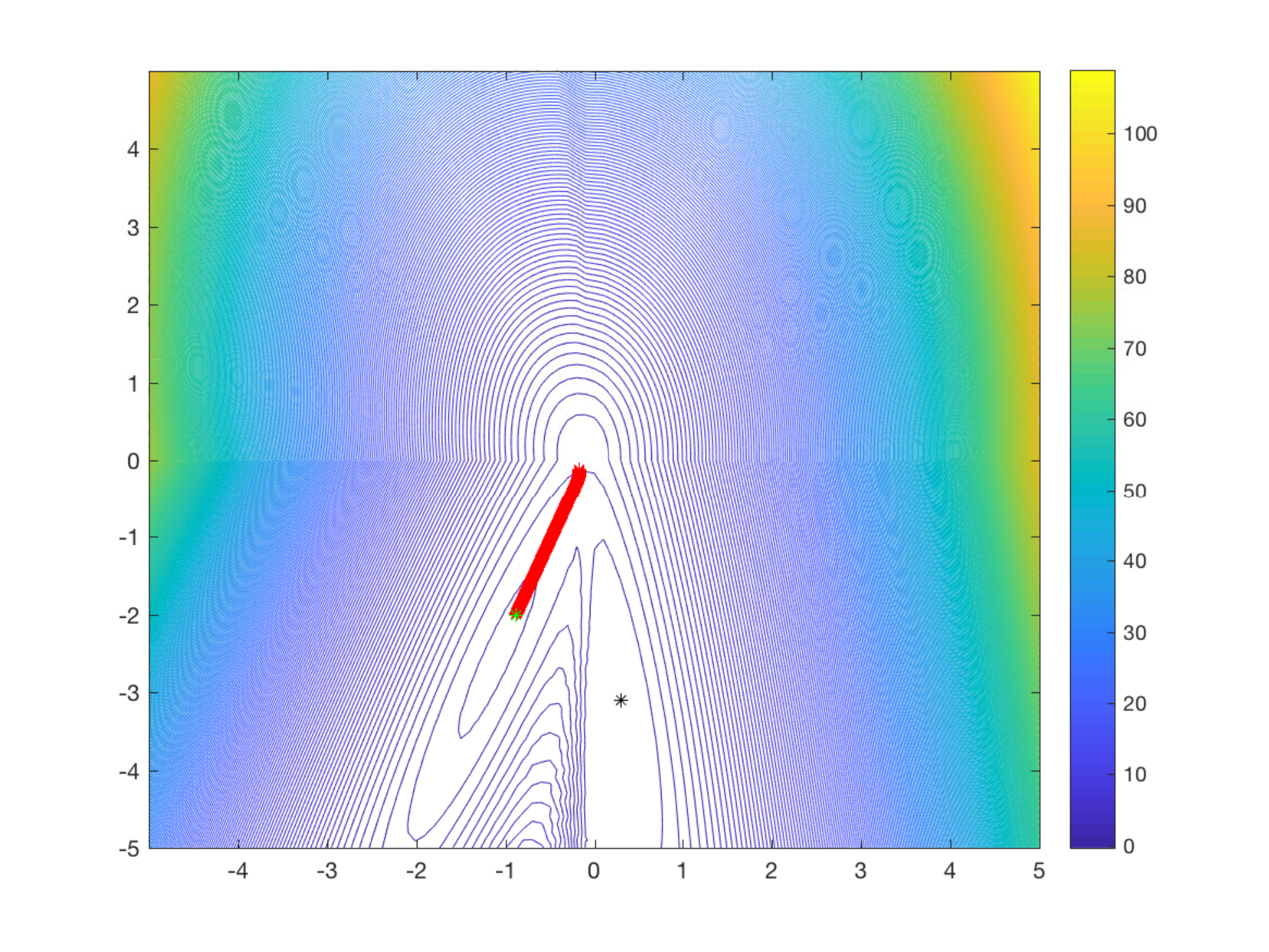}}
           \scalebox{0.4}{\includegraphics{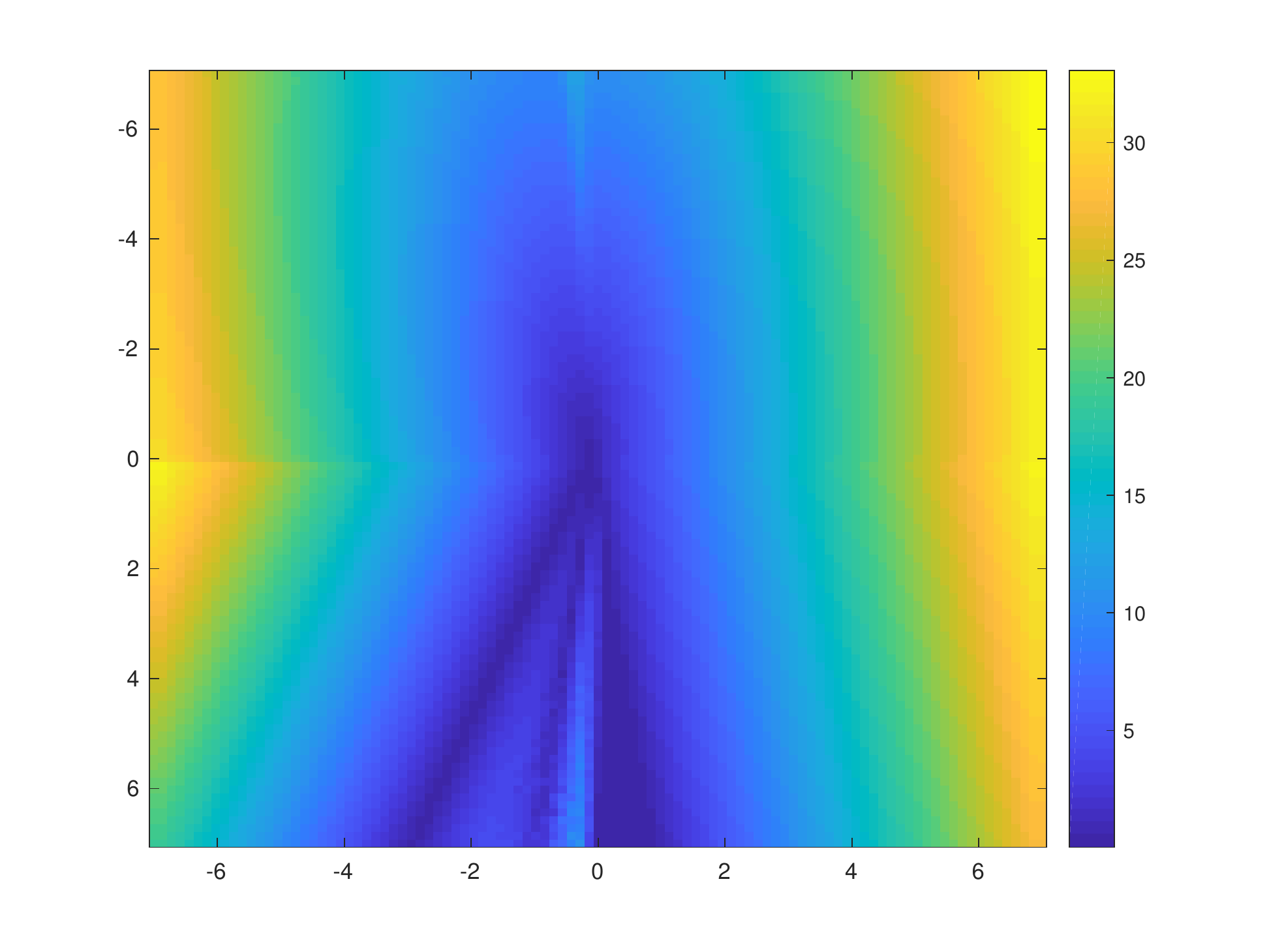}} \\
\caption{\small  \textbf{Example 5} with $k=1$ and $d =2$.  Left: functional to be minimized at $ x = (-0.93,-0.35)$, $t = 0.3$ (black star: global minimum;  red stars: an iterate of coordinate descent with random initial guess; green star: final iterate of the coordinate descent method.) Right: $2$ norm of the gradient of the functional.}\label{exp_4def}
     \end{center}
 \end{figurehere}

\begin{figurehere}
     \begin{center}
     \vskip -0.3truecm
   \scalebox{0.4}{\includegraphics{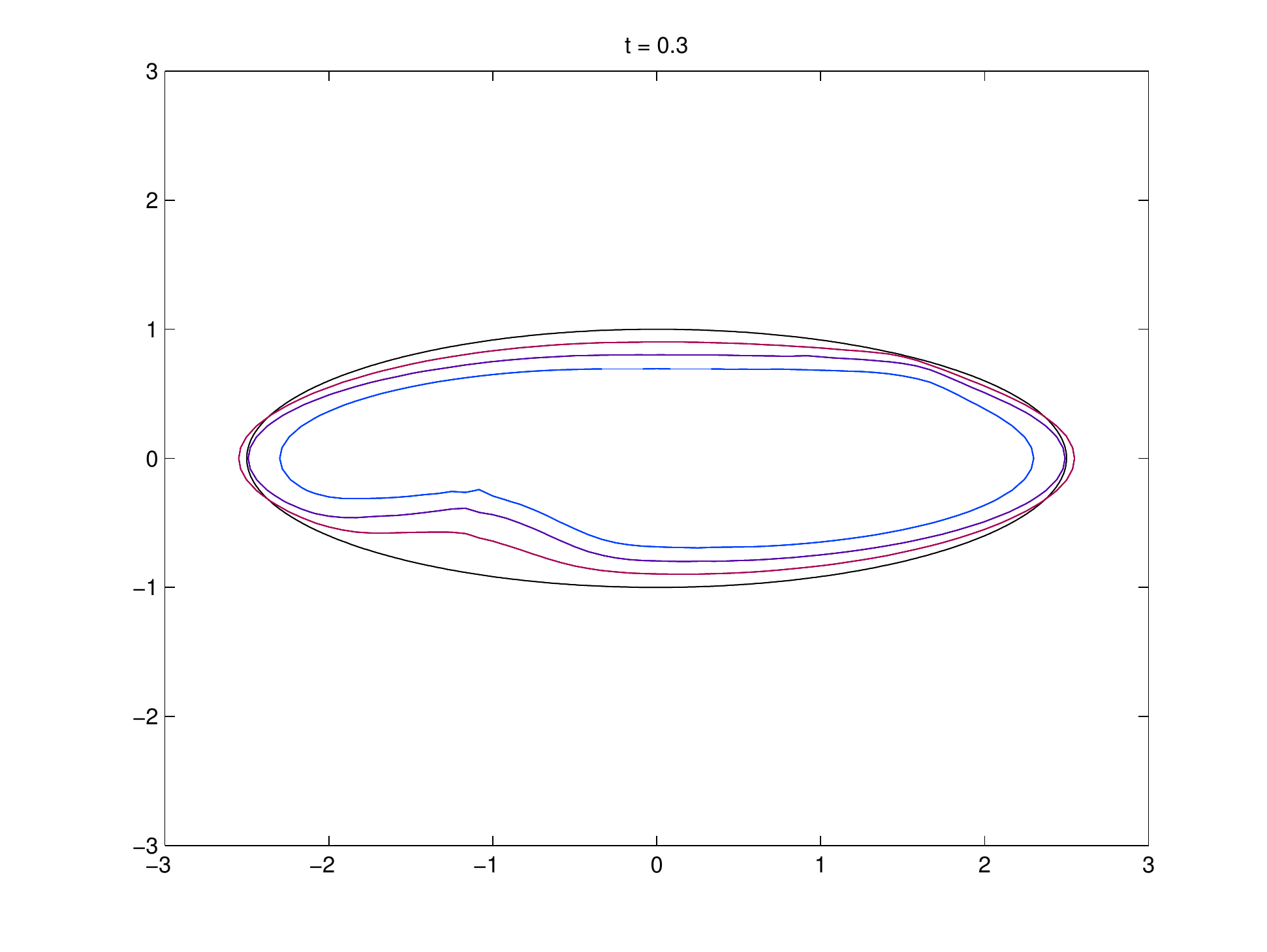}}\\
   \scalebox{0.4}{\includegraphics{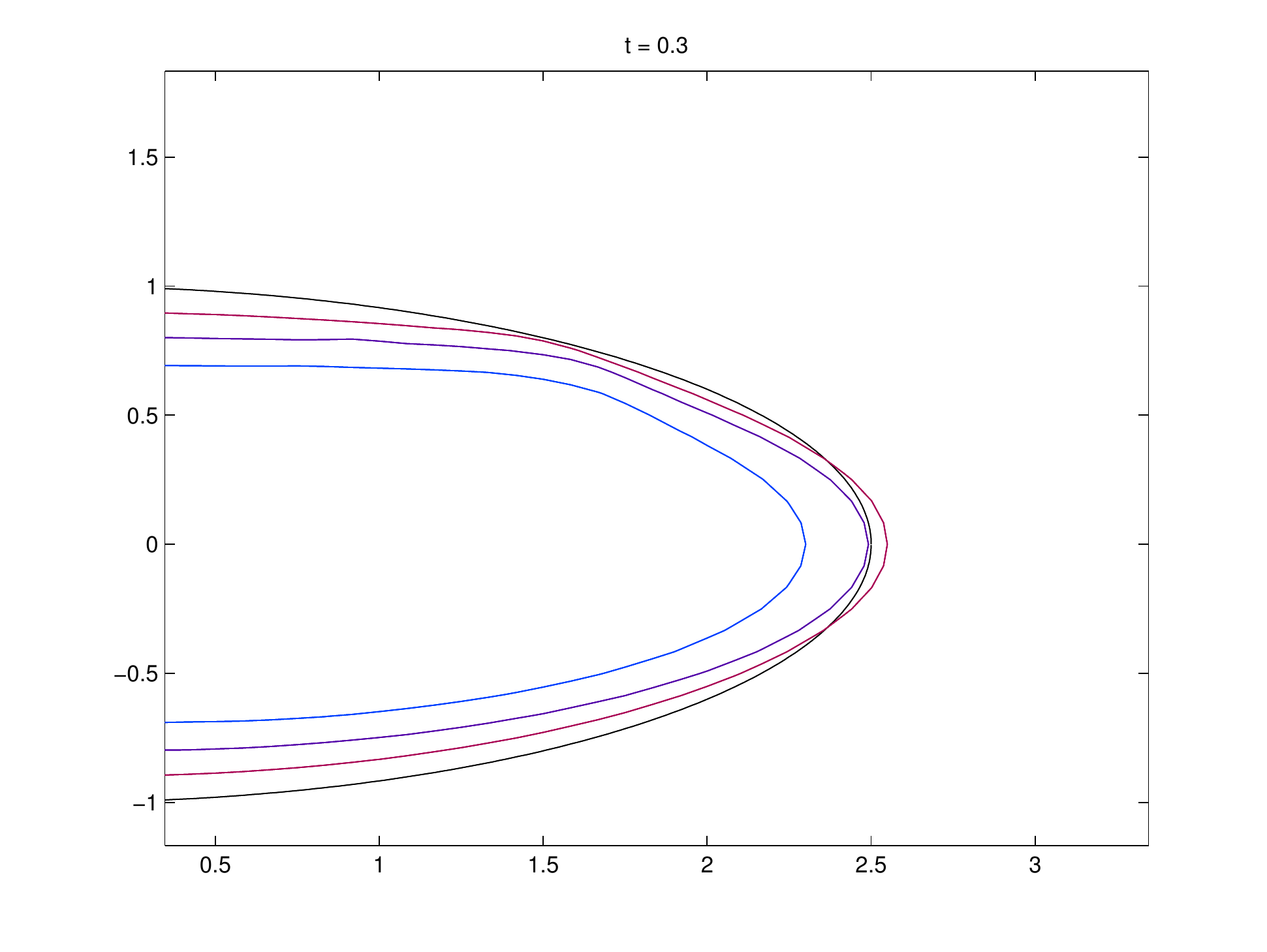}}\\
   \scalebox{0.4}{\includegraphics{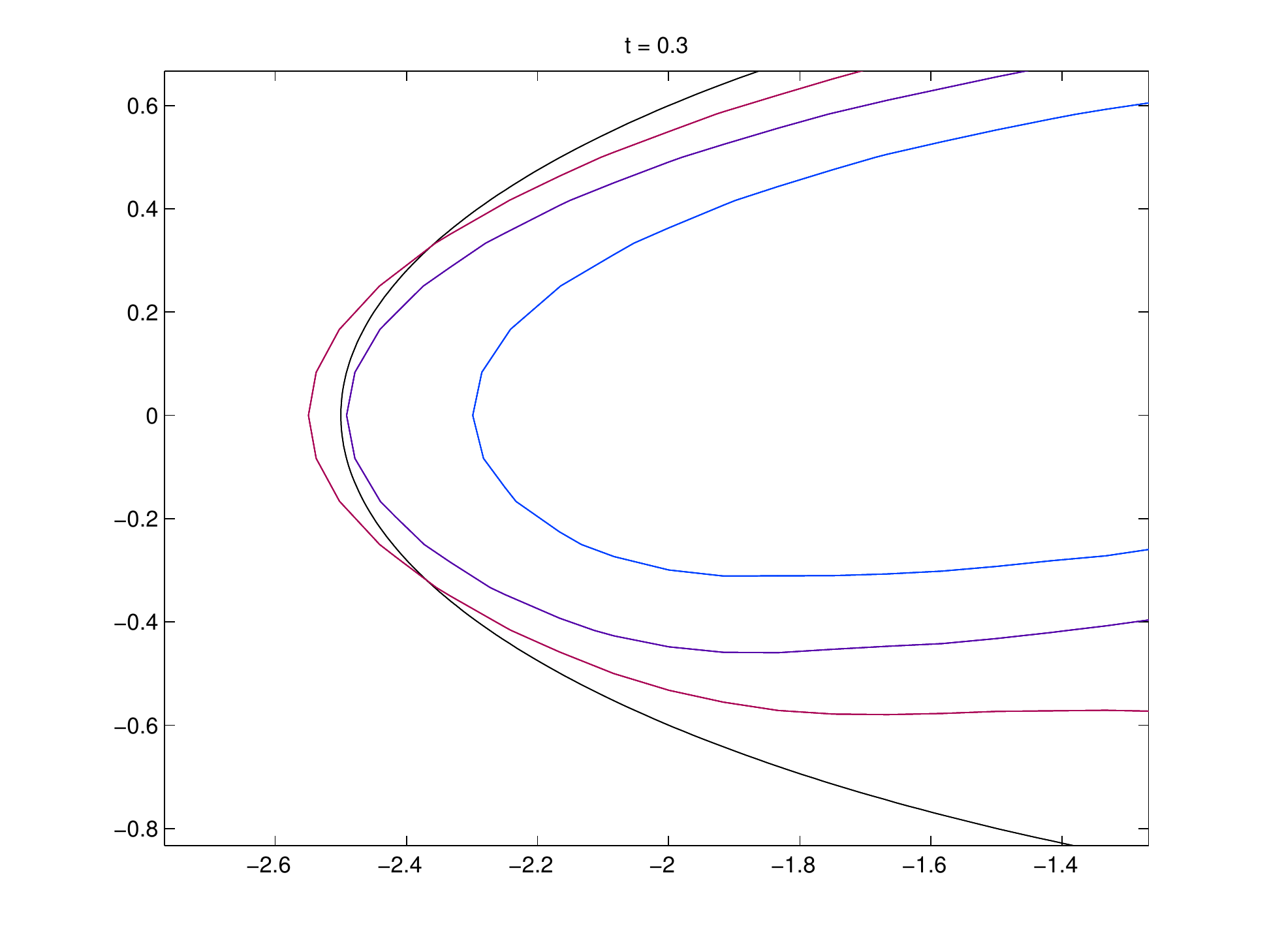}}
\caption{\small Zero level sets of the solution $\phi(\cdot,t)$ for $t = 0.1,0.2,0.3$  by minimization/maximization principle in \textbf{Example 5} with $k=1$ and $d =7$; top: full-size; middle and bottom: close-up.}\label{exp_45}
     \end{center}
 \end{figurehere}

\section{Acknowledgements}

We sincerely thank Dr. Gary Hewer and his colleagues (China Lake Naval Center) for providing help and guidance in practical optimal control and differential game problems, Prof. Frederic Gibou for private communication to explore the possibility to fully parallelize the method, and Prof. Lawrence C. Evans and Prof. Wilfrid Gangbo for exploring if the conjectures made in this paper are sound.
We also thank Alex Tong Lin for many discussions of insights in our paper.

We thank sincerely also the two anonymous referees for many insightful recommendations and comments to help substantially improve our paper.

Research of the first, the third and fourth authors are respectively supported by DOE grant DE-SC00183838 and ONR grant N000141410683, N000141210838, N000141712162.

\end{document}